\documentclass[11pt,a4paper]{amsart}
\usepackage{enumitem}
\usepackage{color}
\usepackage[numbers,sort&compress]{natbib}
\usepackage{pifont,latexsym,ifthen,amsthm,rotating,calc,textcase,booktabs}
\usepackage{amsfonts,amssymb,amsbsy,amsmath}
\usepackage{amsthm}
\numberwithin{equation}{section}
\usepackage[dvipsnames, svgnames, x11names]{xcolor}
\usepackage{graphicx}
\usepackage{tikz}
\usepackage{stmaryrd}
\usepackage{graphicx}
\usepackage[latin1]{inputenc}
\usepackage[T1]{fontenc}
\usepackage[english]{babel}
\usepackage{listings}
\usepackage{xcolor,mathrsfs,url}
\usepackage{ifthen}
\usepackage{fancyhdr}
\usepackage{verbatim}
\usepackage{todonotes}
\usepackage{float}
\usepackage{subfigure}

\usepackage{hyperref}
\usepackage{appendix}
\usepackage[percent]{overpic}

\renewcommand{\Im}{\mathrm{Im}\,}

\newtheorem{theorem}{Theorem}[section]

\newtheorem{remark}[theorem]{Remark}
\newtheorem{proposition}[theorem]{Proposition}
\newtheorem{definition}[theorem]{Definition}

\newtheorem{RHP}[theorem]{RH Problem}
\newtheorem{assumption}[theorem]{Assumption}
\usepackage {caption}

\hypersetup{hypertex=true,
            colorlinks=true,
            linkcolor=blue,
            anchorcolor=blue,
            citecolor=red}

\newcommand*{\dif}{\mathop{}\!\mathrm{d}}
\newcommand*{\E}{\boldsymbol{E}}
\newcommand*{\I}{$\mathrm{I}$}
\newcommand*{\II}{$\mathrm{II}$}
\newcommand*{\III}{$\mathrm{III}$}
\newcommand*{\IV}{$\mathrm{IV}$}
\allowdisplaybreaks[4]

\begin{document}
\title[Painlev\'{e}  XXXIV asymptotics for defocusing NLS equation]
{Painlev\'{e}  XXXIV  asymptotics for the defocusing nonlinear Schr\"odinger equation with a  finite-genus algebro-geometric  background }
\author[Engui Fan,\ Gaozhan Li,\ Yiling Yang \  and  \ Lun Zhang]{Engui Fan$^{1}$,\ Gaozhan Li$^2$,\ Yiling Yang$^{3}$ \  and  \ Lun Zhang$^{4}$   }
\footnotetext[1]{\ School of Mathematical Sciences  and Key Laboratory   for Nonlinear Science, Fudan   University, Shanghai 200433, P. R. China. E-mail: \texttt{faneg@fudan.edu.cn}}

\footnotetext[2]{\ Department of Mathematics,
Tsinghua University,
Beijing 100084, P. R. China.
E-mail: \texttt{gzli@mail.tsinghua.edu.cn}}
\footnotetext[3]{\ College of Mathematics and Statistics,
	Chongqing University,
	Chongqing 401331,
	P. R. China. E-mail: \texttt{ylyang19@fudan.edu.cn}}
\footnotetext[4]{\ School of Mathematical Sciences, Center
for Applied Mathematics and Shanghai
Key Laboratory for Contemporary
Applied Mathematics, Fudan   University, Shanghai 200433, P. R. China. E-mail: \texttt{lunzhang@fudan.edu.cn} }

\maketitle
\begin{abstract}
In this paper, we consider the Cauchy problem for the defocusing nonlinear Schr$\ddot{\text{o}}$dinger equation with a finite genus algebro-geometric  background. Long-time asymptotics of the solution are derived in four space-time regions. It comes out that the leading-order term in the expansion is, up to a constant, given by the background solution with a shift of the parameter. The subleading term, however, decays at different rates for different regions. We particularly highlight that in the two transition regions, they are of order $\mathcal{O}(t^{-1/3})$ and the coefficients involve an integral of the Painlev\'e XXXIV transcendent. We establish our results by applying a nonlinear steepest descent analysis to the associated Riemann-Hilbert problems. 
\\[4pt]

\end{abstract}

\setcounter{tocdepth}{2} 
\tableofcontents

\quad

\section{Introduction}
The one-dimensional, cubic, defocusing nonlinear Schr$\ddot{\text{o}}$dinger (NLS) equation
\begin{equation}\label{NLS}
iq_t+q_{xx}-2|q|^2q=0,\qquad x\in\mathbb{R},
\end{equation}
is one of the most fundamental integrable equations in mathematical physics. Since its connection with the theory of superconductivity \cite{G1955} found in the 1950s, the NLS equation \eqref{NLS} is widely viewed as an important model in describing a variety of physical phenomena, which include water waves \cite{Lake1977}, surface gravity waves \cite{OO2001}, nonlinear optics \cite{PRL2006,C1964}, plasmas \cite{PRL1970} and Bose-Einstein condensates \cite{PRL2005}.


Due to the integrable structure behind, the NLS equation \eqref{NLS} exhibits rich mathematical properties as well. It has recently been proved that equation \eqref{NLS} is locally and globally well-posed in  $  H^s(\mathbb{R}) $ \cite{killpnls} and  admits  a $  H^s(\mathbb{R})$ conserved energy \cite{Duke} for $s>-1/2$. Zakharov and Shabat derived a Lax pair for \eqref{NLS} and performed the inverse scattering transform for the initial value problem in a seminal work \cite{ZS1}; see also \cite{Abl74,ZS2}. Of particular interests are the long-time asymptotics for the NLS equation under different backgrounds. For Schwartz initial data, Zakharov and Manakov first established the long-time asymptotics   of the solution  for  the associated Cauchy problem by using the inverse scattering transform  method \cite{ZS3}; see also the work of Its  based on the stationary phase method \cite{its1}. A rigorous derivation of the leading and high-order long-time  asymptotics was given by Deift and Zhou  \cite{PX2,PX21} with the aid of powerful nonlinear steepest descent method introduced in \cite{sdmRHp}. By reducing the regularity of initial data to the weighted Sobolev space  $H^{1,1} (\mathbb{R})$,  Deift and Zhou  further showed that the same leading asymptotics holds but with an estimate $\mathcal{O}(t^{-1/2-\kappa})$ of the error term for $ 0<\kappa<1/4$, in \cite{PX3}. An improved and sharp error estimate was later reported by Dieng and McLaughlin  via $\bar{\partial}$ methods \cite{DM}; see also \cite{MK}. Quite recently, Lee and Lenells \cite{LenellsRobin} considered the NLS equation on the half line with a Robin boundary condition at $x=0$ and with initial data in the weighted Sobolev space $H^{1,1} (\mathbb{R}^+)$. For finite density initial data, Vartanian computed both the leading and first correction terms of the long-time asymptotic expansion in a series of papers \cite{VAH1,VAH2,VAH3}. Cuccagna and Jenkins \cite{CJ} confirmed the soliton resolution conjecture by establishing the leading order approximation to the solution of \eqref{NLS} in the solitonic region $|x/(2t)|<1$. The long-time asymptotic results in the solitonless region $|x/(2t)|>1$ and two transition regions $|x/(2t)| \approx 1$ were later obtained by Wang and Fan in \cite{WF} and \cite{WZY}, respectively. We also refer to \cite{lenells,Jenkins,wang} and references therein for the studies of the NLS equation with step-like  initial data.

In this paper, we consider Cauchy problem for the  NLS equation \eqref{NLS} with a finite-genus algebro-geometric  background, that is,
\begin{align}
q(x,0)=q_0(x)\sim q^{(AG)}(x,t)\big|_{t=0},\qquad x\to\pm\infty,\label{boundray q0}
\end{align}
where $q^{(AG)}(x,t)=q^{(AG)}(x,t; \E, \hat{\E}, \boldsymbol{\phi})$ is an algebro-geometric solution of \eqref{NLS} depending on some real vectors $\E=(E_0,\ldots,E_{n}), \hat{\E}=(\hat{E}_0,\ldots,\hat{E}_{n}), \boldsymbol{\phi}=(\phi_1,\ldots,\phi_n)$ with $n\in \mathbb{N}_0:=\{0,1,2,\ldots\}$. Following \cite{AGbook,AGbook2}, the construction of $q^{(AG)}$ is standard, which makes use of the Baker-Akhiezer function on the Riemann surface of genus $n$ defined by the equation
\begin{equation}\label{def:w}
	w(z)^2:=\prod_{j=0}^n(z-E_j)(z-\hat{E}_j)
\end{equation}
with
\begin{align}\label{eq:order}
	E_0<\hat{E}_0<E_1<\hat{E}_1<\cdots<E_n<\hat{E}_n.
\end{align}
We refer to Section \ref{algeso} below for the precise definition of $q^{(AG)}$. For completeness of the Cauchy problem, we supplement the following boundary condition:
\begin{align}
	\int_{\mathbb{R}}|q(x,t)-q^{(AG)}(x,t)|\dif x<\infty, \quad t\geq 0, \label{bj}
\end{align}
which is a natural extension of \eqref{boundray q0} to $t > 0$; see also \cite{Monvel2} for similar conditions for the focusing NLS with step-like oscillating background.

It is the aim of the present work to establish long-time asymptotic behavior of the Cauchy problem \eqref{NLS}, \eqref{boundray q0} and \eqref{bj} for the NLS equation by implementing the nonlinear steepest descent analysis of the associated Riemann-Hilbert (RH) problem. As far as we know, there is currently no literature available to study such problem for the NLS equation \eqref{NLS}; see however \cite{Alice} for the invetigation in the context of KdV equation. It comes out that $q(x,t)$ exhibits qualitatively different behaviors in four different regions of the $(x,t)$-plane. We particularly highlight the Painlev\'e XXXIV asymptotics of the solution in two transition regions. Our results are stated in the next section.


\section{Statement of results}

Let $E_j, \hat{E}_j$, $j=0,1,\ldots,n$, be fixed $2n+2$ real numbers ordering according to \eqref{eq:order}, we now set
\begin{equation}\label{def xij}
\xi_j=-\frac{4\prod_{k=0}^{n+1}(E_j-z^g_k)}{\prod_{k=0}^n(E_j-z^f_k)},
\qquad\hat{\xi}_j=-\frac{4\prod_{k=0}^{n+1}(\hat{E}_j-z^g_k)}{\prod_{k=0}^n(\hat{E}_j-z^f_k)}, 
\end{equation}
where $z^f_k$ and $z^g_k$ are related to $E_j$ and $\hat{E}_j$ in the following way. Let
\begin{align}
f(z):=\int_{\hat{E}_0}^z\frac{\prod_{k=0}^n(s-z_k^f)}{w(s)}\dif s,
\qquad g(z):=\int_{\hat{E}_0}^z\frac{4\prod_{k=0}^{n+1}(s-z_k^g)}{w(s)}\dif s,\label{def fg}
\end{align}
be two hyperelliptic integrals, where $w$ in \eqref{def:w} is analytic in $\mathbb{C}\setminus (\cup_{j=0}^n[E_j,\hat{E}_j])$ with the branch chosen such that $w(z)\sim z^{n+1}$ as $z\to\infty$. The constants $z^f_k$ and $z^g_k$ are then uniquely determined by the conditions
\begin{equation}\label{def:fg}
\left\{
  \begin{array}{ll}
    f(z)=z+\mathcal{O}(1), & \hbox{$z\to\infty$,} \\
    \int^{\hat{E}_j}_{E_j}\dif f=0, & \hbox{$j=1,\ldots,n$,}
  \end{array}
\right.
\quad
\textrm{and}
\quad
\left\{
  \begin{array}{ll}
    g(z)=2z^2+\mathcal{O}(1), & \hbox{$z\to\infty$,} \\
    \int^{\hat{E}_j}_{E_j}\dif g=0, & \hbox{$j=1,\ldots,n$,}
  \end{array}
\right.
\end{equation}
respectively. Moreover, $\left( \{z_k^f\}_{k=0}^n\cup\{z_k^g\}_{k=0}^{n+1}\right) \cap \{E_j,\hat{E}_j\}_{j=0}^n=\emptyset$.

The points $\xi_j$ and $\hat \xi_j$ are actually critical points of a phase function $\theta$ defined in \eqref{theta}, with which we divide the $(x, t)$-half plane into several different space-time regions.
\begin{definition}\label{def:regions}
For any positive constant $C$, we define
\begin{itemize}
	\item  transition region  \I: $\cup_{j=0}^n\{\xi: |\xi-\hat{\xi}_j|t^{2/3}\leq C\}$,
	\item  transition region  \II: $\cup_{j=0}^n\{\xi: |\xi-\xi_j|t^{2/3}\leq C\}$,
    \item  Zakharov-Manakov region \III: $\xi\in(-\infty,\hat{\xi}_n)\cup_{j=1}^n(\xi_j,\hat{\xi}_{j-1})\cup(\xi_0,+\infty)$,
	\item  fast decaying region \IV: $\xi \in \cup_{j=0}^n(\hat{\xi}_j,\xi_j)$,
\end{itemize}
where $\xi:=x/t$.
\end{definition}
An illustration of the above four regions is given in Figure \ref{Fig x/t}. We will show long time asymptotics of $q(x,t)$ in these regions under the following assumptions. 

\begin{assumption}\label{asum1}
Throughout this paper, we make the following assumptions. 
\begin{itemize}
	\item [$\mathrm{(a)}$] There exists a positive constant $C$ such that $q_0(x)=q^{(AG)}(x,0)$ for $|x|>C$, i.e., the initial data $q_0$ is identically equal to the background outside a compact set. 
	
 \item [$\mathrm{(b)}$] Let $r_i$, $i=1,\ldots,4$, defined in \eqref{def r} be the reflection coefficients associated with the initial data. It is assumed that $|r_i(p)|=1$ for $p\in\{E_j, \hat{E}_j\}_{j=0}^n$ and  $r_i$ is bounded on $\cup_{j=0}^n(E_j,\hat{E}_j)$.
\end{itemize}
\end{assumption}
Item (a) of this assumption is to avoid some technical work during our asymptotic analysis, which actually only affects the error estimates. The assumption that $|r_i(p)|=1$ in item (b) is a generic one similar to \cite{DT1979}, and the boundness of $r_i$ on $\cup_{j=0}^n(E_j,\hat{E}_j)$ is to avoid the spectral singularity. We refer to Section \ref{sec:Jost} below for more explanations of Assumption  \ref{asum1}. 


\begin{figure}

\tikzset{every picture/.style={line width=0.75pt}} 

\begin{tikzpicture}[x=0.75pt,y=0.75pt,yscale=-0.65,xscale=0.65]

\draw [dash pattern={on 4.5pt off 4.5pt}] (50,397.8) -- (616.77,397.8)(332.06,79) -- (332.06,402.8)
(609.77,392.8) -- (616.77,397.8) -- (609.77,402.8) (327.06,86) -- (332.06,79) -- (337.06,86)  ;
\draw    (609.97,302.37) -- (332.06,397.8) ;
\draw      (614.75,102.75) -- (332.06,397.8) ;
\draw     (396.77,81.57) -- (332.06,397.8) ;
\draw      (222.56,80.23) -- (332.06,397.8) ;
\draw     (49.14,141.73) -- (332.06,397.8) ;
\draw     (50.77,312.37) -- (332.06,397.8) ;
\draw [draw opacity=0][fill={rgb, 255:red, 184; green, 233; blue, 134 }  ,fill opacity=0.3 ]   (332.06,397.8) -- (610.77,321.57) -- (610.77,321.57) -- (610.37,397.97) -- cycle ;
\draw [draw opacity=0][fill={rgb, 255:red, 248; green, 231; blue, 28 }  ,fill opacity=0.3 ]   (332.06,397.8) -- (610.89,286.06) -- (610.77,321.57) -- cycle ;
\draw [draw opacity=0][fill={rgb, 255:red, 74; green, 144; blue, 226 }  ,fill opacity=0.3 ]   (332.06,397.8) -- (616.03,125.03) -- (611.22,274.4) -- (610.89,286.06) -- cycle ;
\draw [draw opacity=0][fill={rgb, 255:red, 248; green, 231; blue, 28 }  ,fill opacity=0.3 ]   (609.53,77.03) -- (617.53,77.53) -- (616.03,125.03) -- (332.06,397.8) ;
\draw [draw opacity=0][fill={rgb, 255:red, 184; green, 233; blue, 134 }  ,fill opacity=0.3 ]   (332.06,397.8) -- (413.22,79.74) -- (609.53,77.03) -- cycle ;
\draw [draw opacity=0][fill={rgb, 255:red, 248; green, 231; blue, 28 }  ,fill opacity=0.3 ]   (332.06,397.8) -- (383.56,79.57) -- (413.22,79.74) -- cycle ;
\draw [draw opacity=0][fill={rgb, 255:red, 74; green, 144; blue, 226 }  ,fill opacity=0.3 ]   (332.06,397.8) -- (238.56,80.07) -- (383.56,79.57) -- cycle ;
\draw [draw opacity=0][fill={rgb, 255:red, 248; green, 231; blue, 28 }  ,fill opacity=0.3 ]   (238.56,80.07) -- (332.06,397.8) -- (207.56,79.9) -- cycle ;
\draw [draw opacity=0][fill={rgb, 255:red, 184; green, 233; blue, 134 }  ,fill opacity=0.3 ]   (207.56,79.9) -- (332.06,397.8) -- (48.81,130.4) -- (48.14,79.4) ;
\draw [draw opacity=0][fill={rgb, 255:red, 248; green, 231; blue, 28 }  ,fill opacity=0.3 ]   (48.81,130.4) -- (332.06,397.8) -- (49.48,160.06) ;
\draw [draw opacity=0][fill={rgb, 255:red, 74; green, 144; blue, 226 }  ,fill opacity=0.3 ]   (332.06,397.8) -- (51.14,296.4) -- (49.48,160.06) ;
\draw [draw opacity=0][fill={rgb, 255:red, 248; green, 231; blue, 28 }  ,fill opacity=0.3 ]   (51.14,296.4) -- (332.06,397.8) -- (50.81,328.73) ;
\draw [draw opacity=0][fill={rgb, 255:red, 184; green, 233; blue, 134 }  ,fill opacity=0.3 ]   (50.81,328.73) -- (332.06,397.8) -- (50.81,398.06) ;

\draw (614.4,293.73) node [anchor=north west][inner sep=0.75pt]  [font=\footnotesize,xscale=0.8,yscale=0.8]  {$\xi _{0}$};
\draw (620.97,88.38) node [anchor=north west][inner sep=0.75pt]  [font=\footnotesize,xscale=0.8,yscale=0.8]  {$\hat{\xi }_{0}$};
\draw (391.55,58.3) node [anchor=north west][inner sep=0.75pt]  [font=\footnotesize,xscale=0.8,yscale=0.8]  {$\xi _{1}$};
\draw (213.88,54.72) node [anchor=north west][inner sep=0.75pt]  [font=\footnotesize,xscale=0.8,yscale=0.8]  {$\hat{\xi }_{1}$};
\draw (26.55,129.63) node [anchor=north west][inner sep=0.75pt]  [font=\footnotesize,xscale=0.8,yscale=0.8]  {$\xi _{2}$};
\draw (24.88,300.38) node [anchor=north west][inner sep=0.75pt]  [font=\footnotesize,xscale=0.8,yscale=0.8]  {$\hat{\xi }_{2}$};
\draw (329,399.3) node [anchor=north west][inner sep=0.75pt]  [xscale=0.8,yscale=0.8]  {$0$};
\draw (620.96,395.63) node [anchor=north west][inner sep=0.75pt]  [xscale=0.8,yscale=0.8]  {$x$};
\draw (328.21,60.46) node [anchor=north west][inner sep=0.75pt]  [xscale=0.8,yscale=0.8]  {$t$};
\draw (74,364.61) node [anchor=north west][inner sep=0.75pt]  [font=\small,xscale=0.8,yscale=0.8]  {\III};
\draw (192.4,196.27) node [anchor=north west][inner sep=0.75pt]  [font=\small,xscale=0.8,yscale=0.8]  {\III};
\draw (442,180.67) node [anchor=north west][inner sep=0.75pt]  [font=\small,xscale=0.8,yscale=0.8]  {\III};
\draw (543.2,368.67) node [anchor=north west][inner sep=0.75pt]  [font=\small,xscale=0.8,yscale=0.8]  {\III};
\draw (82.8,253.07) node [anchor=north west][inner sep=0.75pt]  [font=\small,xscale=0.8,yscale=0.8]  {\IV};
\draw (292.8,143.47) node [anchor=north west][inner sep=0.75pt]  [font=\small,xscale=0.8,yscale=0.8]  {\IV};
\draw (544,244.67) node [anchor=north west][inner sep=0.75pt]  [font=\small,xscale=0.8,yscale=0.8]  {\IV};
\draw (52.77,315.77) node [anchor=north west][inner sep=0.75pt]  [font=\small,xscale=0.8,yscale=0.8]  {\I};
\draw (237.6,103.07) node [anchor=north west][inner sep=0.75pt]  [font=\small,xscale=0.8,yscale=0.8]  {\I};
\draw (604.8,113.87) node [anchor=north west][inner sep=0.75pt]  [font=\small,xscale=0.8,yscale=0.8]  {\I};
\draw (47.77,155.57) node [anchor=north west][inner sep=0.75pt]  [font=\small,xscale=0.8,yscale=0.8]  {\II};
\draw (396.37,83.77) node [anchor=north west][inner sep=0.75pt]  [font=\small,xscale=0.8,yscale=0.8]  {\II};
\draw (590.97,308.97) node [anchor=north west][inner sep=0.75pt]  [font=\small,xscale=0.8,yscale=0.8]  {\II};
	\end{tikzpicture}
	\caption{\footnotesize  Four different asymptotic regions given in Definition \ref{def:regions} with $n=2$.}
	\label{Fig x/t}
\end{figure}
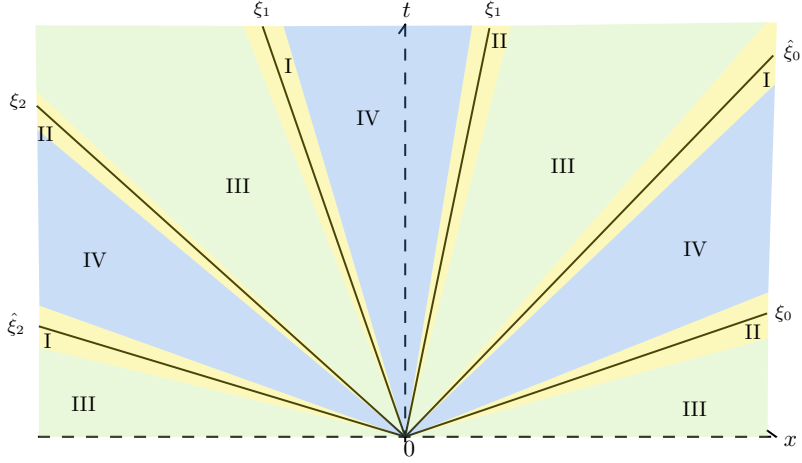

\begin{theorem} \label{mainthm}
Given  a finite genus algebro-geometric solution  $q^{(AG)}(x,t; \E, \hat{\E}, \boldsymbol{\phi})$ of the NLS equation \eqref{NLS},
let $q(x,t)$  be the solution of the Cauchy problem \eqref{NLS} and \eqref{boundray q0}, subject to the condition \eqref{bj}.
As $t\to\infty$, we have the following asymptotics of $q(x,t)$ in the regions \I--\IV given in Definition \ref{def:regions} under Assumption  \ref{asum1}. 
	\begin{itemize}\label{eq:asyq I}
		\item $\mathrm{(Asymptotics~in~transition~region~I)}$
		For $|\xi-\hat{\xi}_{j_0}|t^{2/3}\leq C$ with fixed $j_0\in\{0,\dots, n\}$,
		\begin{align}
			q(x,t)=
			e^{-2\delta(\infty)}q^{(AG)}(x,t;\E, \hat{\E}, \boldsymbol{\phi-\delta})
			 +H_{\hat{E}_{j_0}}a(s)t^{-1/3}+\mathcal{O}(t^{-\epsilon}), \quad  \epsilon\in(1/3,2/3), \label{rg1}
		\end{align}
where the real vector $\boldsymbol{\delta}=(\delta_1,\ldots,\delta_n)$ and the purely imaginary constant $\delta(\infty)$  are defined in \eqref{def deltaj I} and \eqref{def delta inf}, respectively, $H_{\hat{E}_{j_0}}$ given in \eqref{Hhj0} is a function of $x$ and $t$ constructed by the Riemann theta function and 
 \begin{equation}\label{def:a}
 a(s)=\int_{-\infty}^s\left(u(\zeta)+\frac{\zeta}{2}\right)\dif\zeta,
\end{equation}
 with
\begin{equation}
 s=-\frac{\theta^{(1,\hat{j_0})}}{(\theta^{(3,\hat{j_0})}(\xi))^{1/3}}(\xi-\hat{\xi}_{j_0})t^{2/3}.
\end{equation}
Here, $u(s)$ is unique the solution of the Painlev\'e XXXIV equation
\begin{align} \label{P34}
		 	u''(s)=4u(s)^2+2su(s)+\frac{u'(s)^2- {1}/{4}} {2u(s)}
\end{align}
satisfying
\begin{equation}\label{eq:uasy}
    u(s)=\left\{ \begin{array}{ll}
				-1/(4\sqrt{s})+\mathcal{O}(s^{-2}), &\qquad s\to +\infty,
				\\
				-s/2+\mathcal{O}(s^{-2}), &\qquad s\to -\infty,
			\end{array}
			\right .
\end{equation}
\begin{equation}
    \theta^{(1,\hat{j_0})}=-2\left[ \prod_{j=0}^n(\hat{E}_{j_0}-E_j)\cdot\prod_{j=0,\cdots,n,j\neq j_0}(\hat{E}_{j_0}-\hat{E}_{j})\right] ^{-\frac{1}{2}}\prod_{j=0}^n(\hat{E}_{j_0}-z_j^f),
\end{equation}
and
\begin{align}
 \theta^{(3,\hat{j_0})}(\hat{\xi}_{j_0})=&-\frac{\big( \hat{\xi}_{j_0}\prod_{j=0}^n(z-z^f_j)+4\prod_{j=0}^{n+1}(z-z^g_j)\big)'\Big|_{z=\hat{E}_{j_0}}}{\left[ \prod_{j=0}^n(\hat{E}_{j_0}-E_j)\cdot\prod_{j=0,\cdots,n,j\neq j_0}(\hat{E}_{j_0}-\hat{E}_{j})\right] ^{\frac{1}{2}}}.
\end{align}
	
	\item  $\mathrm{(Asymptotics~in~transition~region~II)}$ 
For $|\xi-\xi_{j_0}|t^{2/3}\leq C$ with fixed $j_0\in\{0,\cdots,n\}$,
		\begin{align}\label{eq:asyq II}
q(x,t)=e^{-2\delta(\infty)}q^{(AG)}(x,t;\E, \hat{\E}, \boldsymbol{\phi-\delta})
			 -\tilde{H}_{E_{j_0}}a(s)t^{-1/3}+\mathcal{O}(t^{-\epsilon}), \quad \epsilon\in(1/3,2/3),
		\end{align}
where the real vector $\boldsymbol{\delta}=(\delta_1,\ldots,\delta_n)$ and the  purely imaginary constant $\delta(\infty)$  are defined in \eqref{def deltaj II} and \eqref{def delta infII}, respectively, $\tilde{H}_{E_{j_0}}$ given in \eqref{HEj0} is a function of $x$ and $t$ constructed by the Riemann theta function and  $a(s)$ is an integral involving the Painlev\'e XXXIV transcendent defined through \eqref{def:a}, \eqref{P34} and \eqref{eq:uasy} with
		\begin{align*}
s=-\frac{\theta^{(1,j_0)}}{(\theta^{(3,j_0)}(\xi))^{1/3}}(\xi-\xi_{j_0})t^{2/3},
		\end{align*}
and where $\theta^{(1,j_0)}$ and $\theta^{(3,j_0)}(\xi)$ are given in \eqref{theta j01} and \eqref{theta j03}, respectively. 

	\item 
    $\mathrm{(Asymptotics~in~Zakharov-Manakov~region~III~and~fast ~decaying~region~IV) }$ For $\xi\in(-\infty,\hat{\xi}_n)\cup_{j=1}^n(\xi_j,\hat{\xi}_{j-1})\cup(\xi_0,+\infty)$,
		\begin{align}\label{eq:asyIII}
			& q(x,t)=e^{-2\delta(\infty)}q^{(AG)}(x,t;\E, \hat{\E}, \boldsymbol{\phi-\delta}) \nonumber \\
			&+\frac{2ie^{-2\delta(\infty)}}{2\sqrt{\theta^{(z_0,2)}(\xi)}}\left(\beta_{21}M^{(glo)}_{11}(z_0;\E, \hat{\E}, \boldsymbol{\phi-\delta})^2-\beta_{12}M^{(glo)}_{12}(z_0;\E, \hat{\E}, \boldsymbol{\phi-\delta})^2\right)t^{-1/2} \nonumber \\
			&
		 +\mathcal{O}(t^{-1}),
		\end{align}
        and for $\xi \in \cup_{j=0}^n(\hat{\xi}_j,\xi_j)$,
		\begin{align}\label{eq:asyq IV}
			q(x,t)= e^{-2\delta(\infty)}q^{(AG)}(x,t;\E, \hat{\E}, \boldsymbol{\phi-\delta})+\mathcal{O}(t^{-1}),
		\end{align}
		where the real vector $\boldsymbol{\delta}=(\delta_1,\ldots,\delta_n)$ and the  purely imaginary constant $\delta(\infty)$  are  defined in  \eqref{def deltaj III} and \eqref{def delta inf III}, respectively. In \eqref{eq:asyIII}, $\theta^{(z_0,2)}(\xi)$ is a nonzero function given in \eqref{def theta z0}, $M^{(glo)}_{11}$ and $M^{(glo)}_{12}$  given in \eqref{Mglo11} and \eqref{Mglo12} are two functions constructed by the Riemann theta function and
 \begin{align*}
 	\beta_{12}=\frac{\sqrt{2\pi}e^{\frac{1}{4}(\pi i-\log(1-|r_0|^2))}}{r_0\Gamma(i\log(1-|r_0|^2)/2\pi)},\qquad \beta_{21}=\frac{\log(1-|r_0|^2)}{2\pi \beta_{12}},
 \end{align*}
 with the constant $r_0$ defined in \eqref{def r0}.

	\end{itemize}

\end{theorem}

Our theorem indicates that for the Cauchy problem of NLS equation \eqref{NLS} with a finite-genus algebro-geometric background, the leading-order term of $q$ in long-time expansion is, up to a constant, given by the background solution with a shift of the parameter $\boldsymbol{\phi}$. The subleading term, however, decays at different rates for different regions. We particularly emphasize that in the two transition regions, they are of order $t^{-1/3}$ and the coefficients involve an integral of the Painlev\'e XXXIV transcendent. In general, the Painlev\'e XXXIV equation depends on a parameter $\alpha$ and reads as 
\begin{equation}\label{eq:P34}
    u''(s)=4u(s)^2+2su(s)+\frac{u'(s)^2-(2\alpha)^2}{2u(s)}.
\end{equation}
The equation \eqref{P34} corresponds to $\alpha=-1/4$ in \eqref{eq:P34}, and the particular solution characterized by the boundary condition \eqref{eq:uasy} is non-singular on the real line. 

The Painlev\'e transcendents and their higher order analogues play important roles in the studies of transient asymptotics of integrable PDEs. In the literature, the transient asymptotics was first understood in the case of the Korteweg-de Vries (KdV) equation by Segur and Ablowitz  \cite{AblwzP1} and of the modified Korteweg-de Vries (mKdV) equation by Deift and Zhou \cite{sdmRHp}, respectively, which are all given in terms of the Painlev\'e II transcendents. The appearances of  Painlev\'e II equation in asymptotics of other integrable systems can be found in
\cite{Monvel,lenellsboss,HL20,Miller1,WZY,XYZ}. In addition, Boutet de Monvel, Lenells  and Shepelsky found the  asymptotics of focusing NLS equation with step-like oscillating background in a transition zone between two genus-3 sectors is related to a model RH problem associated with the Painlev\'e IV equation \cite{Monvel5}. Bilman, Ling and Miller established a connection between fundamental rogue wave solutions of the focusing NLS equation in the limit of large order and certain members of the Painlev\'{e} III hierarchy \cite{Miller1}.
Huang and  Zhang  obtained Painlev\'{e}  II  hierarchy asymptotics   for the whole  mKdV hierarchy \cite{Huanglin}.
It has been known that the Painlev\'e XXXIV equation ever appears in the investigations of critical asymptotics of orthogonal polynomials and random matrix theory \cite{DXZ2019,ikj2008,XZ2011}.
However, to the best of our knowledge, there exist no results prior to this work that report Painlev\'e XXXIV asymptotics for integrable PDEs. 
 

\vspace{2mm}
\paragraph{\bf{Outline of the paper}}
The rest of this paper is organized as follows. In Section \ref{sec1}, we first give a precise definition of the algebro-geometric solution $q^{(AG)}$ and then establish two RH problems that characterize the Cauchy problem \eqref{NLS} and \eqref{boundray q0} based on the spectral analysis of  its Lax pair. We then carry out nonlinear steepest descent analysis \cite{sdmRHp} to the associated RH problems in Sections \ref{Sec 3}--\ref{Sec 5}, which corresponds to regions I, II and III-IV given in Definition \ref{def:regions}, respectively. The analysis consists of a series of explicit and invertible transformations which leads to a RH problem tending to the identity matrix as $t \to +\infty$. Despite the analysis in Zakharov-Manakov region III and fast decaying region IV is somehow standard, an important new feature of the analysis in transition regions is the involvement of a specified Painlev\'e XXXIV paramertix introduced in Appendix \ref{APP P34}. The outcome of our analysis is the proof of Theorem \ref{mainthm}, which is presented in Section \ref{sec:proof}.


\vspace{2mm}
\paragraph{\bf{Notation}} Throughout this paper, the following notation will be used.
\begin{itemize}
   \item[$\bullet$] If $A$ is a matrix, then $A_{ij}$ stands for its $(i,j)$-the entry.
	\item[$\bullet$] The norm in the classical space $ L^p (\mathbb{R})$, $ 1 \leqslant p \leqslant \infty $,  is written as $ \|\cdot \|_{L^p}$.  For a $2\times 2$ matrix $A$, we also define $\|A\|_{L^p}=\max_{i,j=1,2}\|A_{ij}\|_{L^p}$.
	\item[$\bullet$]
    For a complex-valued function $f(z)$, we use
	\begin{equation*}
		f^{*}(z):=\overline{f(\bar{z})}, \qquad z\in\mathbb{C},
	\end{equation*}
	to denote its Schwartz conjugation.
	\item[$\bullet$]For a region $U\subseteq \mathbb{C}$, we use $U^*$ to denote the conjugated region of $U$. We also set
	\begin{equation}\label{def:Cpm}
	\mathbb{C}^{\pm}:=\left\{z
		\in\mathbb{C}: \pm \text{Im}\ z>0 \right\}.
	\end{equation}
	\item[$\bullet$]As usual, the classical Pauli matrices $\{\sigma_j\}_{j=1,2,3}$ are defined by
	\begin{equation}\label{def:PauliM}
		\sigma_1:=\begin{pmatrix}0 & 1 \\ 1 & 0\end{pmatrix}, \quad
		\sigma_2:=\begin{pmatrix}0 & -i \\ i & 0\end{pmatrix}, \quad
		\sigma_3:=\begin{pmatrix}1 & 0 \\ 0 & -1\end{pmatrix}.
	\end{equation}
   For a $2\times 2$ matrix $A$, we also define
	\begin{equation}\label{def:hatsigma}
		e^{\hat{\sigma}_j}A:=e^{\sigma_j}Ae^{-\sigma_j}, \quad j=1,2,3.
	\end{equation}
	\item[$\bullet$]For any smooth oriented curve $\Sigma$, the Cauchy operator $\mathcal{C}$ on $\Sigma$ is defined  by
	\begin{align*}
		\mathcal{C}f(z)=\frac{1}{2\pi i}\int_{\Sigma}\frac{f(\zeta)}{\zeta-z}\dif\zeta, \qquad  z\in\mathbb{C}\setminus \Sigma.
	\end{align*}
	Given a function $f \in L^p(\Sigma)$, $1\leqslant p<\infty$,
	\begin{align}\label{def:opCpm}
		\mathcal{C}_\pm f(z):=\lim_{z'\to z\in\Sigma}\frac{1}{2\pi i}\int_{\Sigma}\frac{f(\zeta)}{\zeta-z'}\dif\zeta
	\end{align}
   stands for the positive/negative (according to the orientation of $\Sigma$) non-tangential boundary value of $\mathcal{C}f$.
	\item[$\bullet$]We write $a\lesssim b$ to denote the inequality $a\leqslant Cb$ for some constant $C>0$.
    \item Given a function $f$  and an orientated curve, $f_+$ and $f_-$ stand for the limiting values of $f$ when approaching the curve from the left and the right, respectively. Moreover, we always set the orientations of intervals on the real axis to be directed from the left to the right.

\end{itemize}



\section{Preliminaries}\label{sec1}


\subsection{The algebro-geometric solution} \label{algeso}
To give a precise definition of the algebro-geometric solution $q^{(AG)}$, we start with some ingredients from the theory of Riemann surface \cite{AGbook,BA,Farbook}.

Let $\mathcal{R}$ be the Riemann surface of genus $n$ defined by all points $P:=(w,z)\in \mathbb{C}^2$ satisfying \eqref{def:w} and \eqref{eq:order},
and $z=\pi(P)$ is a standard projection. The Riemann surface $\mathcal{R}$ can be built in the following way. We first glue two sheets along the cuts $(E_j,\hat{E}_j)$, $j=0,1,\ldots,n$, and then compactify the resulting surface by adding the point at infinity on each sheet. On the Riemann surface $\mathcal{R}$, we construct a canonical homology basis consisting of closed curves $a_j$ and $b_j$, $j=1,\ldots,n$. Here, $a_j$ is a closed curve on the upper sheet going counterclockwise around the interval $(E_j,\hat{E}_j)$, and $b_j$ starts from $(E_0, \hat{E}_0)$, goes on the upper sheet to $(E_j, \hat{E}_j)$, and returns on the lower sheet to the starting point; see Figure \ref{fig ajbj} for an illustration.
\begin{figure}	
	\tikzset{every picture/.style={line width=0.75pt}} 
	\begin{tikzpicture}[x=0.75pt,y=0.75pt,yscale=-0.7,xscale=0.7]
		\draw [color={rgb, 255:red, 74; green, 144; blue, 226 }  ,draw opacity=1 ]   (105.4,209.4) .. controls (111.4,169.4) and (214.9,172.9) .. (218.9,210.9) ;
		\draw [shift={(167.77,181.12)}, rotate = 182.35] [fill={rgb, 255:red, 74; green, 144; blue, 226 }  ,fill opacity=1 ][line width=0.08]  [draw opacity=0] (10.72,-5.15) -- (0,0) -- (10.72,5.15) -- (7.12,0) -- cycle    ;
		\draw [color={rgb, 255:red, 74; green, 144; blue, 226 }  ,draw opacity=1 ] [dash pattern={on 4.5pt off 4.5pt}]  (218.9,210.9) .. controls (212.9,250.9) and (109.4,247.4) .. (105.4,209.4) ;
		\draw [shift={(156.53,239.18)}, rotate = 2.35] [fill={rgb, 255:red, 74; green, 144; blue, 226 }  ,fill opacity=1 ][line width=0.08]  [draw opacity=0] (10.72,-5.15) -- (0,0) -- (10.72,5.15) -- (7.12,0) -- cycle    ;
		\draw [color={rgb, 255:red, 65; green, 117; blue, 5 }  ,draw opacity=1 ]   (186.9,211.4) .. controls (191.9,184.9) and (285.9,184.9) .. (290.9,211.9) ;
		\draw [shift={(244.09,191.7)}, rotate = 181.08] [fill={rgb, 255:red, 65; green, 117; blue, 5 }  ,fill opacity=1 ][line width=0.08]  [draw opacity=0] (10.72,-5.15) -- (0,0) -- (10.72,5.15) -- (7.12,0) -- cycle    ;
		\draw [color={rgb, 255:red, 65; green, 117; blue, 5 }  ,draw opacity=1 ]   (290.9,211.9) .. controls (286.9,234.4) and (191.9,232.4) .. (186.9,211.4) ;
		\draw [shift={(234.06,227.82)}, rotate = 1.47] [fill={rgb, 255:red, 65; green, 117; blue, 5 }  ,fill opacity=1 ][line width=0.08]  [draw opacity=0] (10.72,-5.15) -- (0,0) -- (10.72,5.15) -- (7.12,0) -- cycle    ;
		\draw [color={rgb, 255:red, 74; green, 144; blue, 226 }  ,draw opacity=1 ]   (73.4,209.9) .. controls (86.9,122.4) and (356.9,158.4) .. (354.4,212.9) ;
		\draw [shift={(217.83,158.01)}, rotate = 183.45] [fill={rgb, 255:red, 74; green, 144; blue, 226 }  ,fill opacity=1 ][line width=0.08]  [draw opacity=0] (10.72,-5.15) -- (0,0) -- (10.72,5.15) -- (7.12,0) -- cycle    ;
		\draw [color={rgb, 255:red, 74; green, 144; blue, 226 }  ,draw opacity=1 ] [dash pattern={on 4.5pt off 4.5pt}]  (354.4,212.9) .. controls (338.9,287.9) and (73.4,269.4) .. (73.4,209.9) ;
		\draw [shift={(208.09,261.84)}, rotate = 1.83] [fill={rgb, 255:red, 74; green, 144; blue, 226 }  ,fill opacity=1 ][line width=0.08]  [draw opacity=0] (10.72,-5.15) -- (0,0) -- (10.72,5.15) -- (7.12,0) -- cycle    ;
		\draw [color={rgb, 255:red, 74; green, 144; blue, 226 }  ,draw opacity=1 ]   (57.4,209.4) .. controls (81.4,77.4) and (557.9,157.9) .. (555.4,212.4) ;
		\draw [shift={(309.03,140.36)}, rotate = 184.23] [fill={rgb, 255:red, 74; green, 144; blue, 226 }  ,fill opacity=1 ][line width=0.08]  [draw opacity=0] (10.72,-5.15) -- (0,0) -- (10.72,5.15) -- (7.12,0) -- cycle    ;
		\draw [color={rgb, 255:red, 74; green, 144; blue, 226 }  ,draw opacity=1 ] [dash pattern={on 4.5pt off 4.5pt}]  (555.4,212.4) .. controls (539.9,287.4) and (31.9,316.4) .. (57.4,209.4) ;
		\draw [shift={(291.02,279.15)}, rotate = 358.31] [fill={rgb, 255:red, 74; green, 144; blue, 226 }  ,fill opacity=1 ][line width=0.08]  [draw opacity=0] (10.72,-5.15) -- (0,0) -- (10.72,5.15) -- (7.12,0) -- cycle    ;
		\draw [color={rgb, 255:red, 65; green, 117; blue, 5 }  ,draw opacity=1 ]   (335.4,212.9) .. controls (337.9,189.9) and (408.9,191.9) .. (411.4,213.39) ;
		\draw [shift={(378.35,196.68)}, rotate = 182.36] [fill={rgb, 255:red, 65; green, 117; blue, 5 }  ,fill opacity=1 ][line width=0.08]  [draw opacity=0] (10.72,-5.15) -- (0,0) -- (10.72,5.15) -- (7.12,0) -- cycle    ;
		\draw [color={rgb, 255:red, 65; green, 117; blue, 5 }  ,draw opacity=1 ]   (411.4,213.39) .. controls (408.4,232.4) and (345.4,236.4) .. (335.4,212.9) ;
		\draw [shift={(368.19,228.92)}, rotate = 1.75] [fill={rgb, 255:red, 65; green, 117; blue, 5 }  ,fill opacity=1 ][line width=0.08]  [draw opacity=0] (10.72,-5.15) -- (0,0) -- (10.72,5.15) -- (7.12,0) -- cycle    ;
		\draw [color={rgb, 255:red, 65; green, 117; blue, 5 }  ,draw opacity=1 ]   (520.9,214.4) .. controls (523.4,191.4) and (630.9,192.91) .. (633.4,214.4) ;
		\draw [shift={(582.09,197.81)}, rotate = 180.97] [fill={rgb, 255:red, 65; green, 117; blue, 5 }  ,fill opacity=1 ][line width=0.08]  [draw opacity=0] (10.72,-5.15) -- (0,0) -- (10.72,5.15) -- (7.12,0) -- cycle    ;
		\draw [color={rgb, 255:red, 65; green, 117; blue, 5 }  ,draw opacity=1 ]   (633.4,214.4) .. controls (630.4,233.41) and (530.9,237.9) .. (520.9,214.4) ;
		\draw [shift={(571.82,230.35)}, rotate = 0.18] [fill={rgb, 255:red, 65; green, 117; blue, 5 }  ,fill opacity=1 ][line width=0.08]  [draw opacity=0] (10.72,-5.15) -- (0,0) -- (10.72,5.15) -- (7.12,0) -- cycle    ;
		\draw [line width=1.5]    (41.9,209.9) -- (123.9,210.9) ;
		\draw [shift={(123.9,210.9)}, rotate = 0.7] [color={rgb, 255:red, 0; green, 0; blue, 0 }  ][fill={rgb, 255:red, 0; green, 0; blue, 0 }  ][line width=1.5]      (0, 0) circle [x radius= 4.36, y radius= 4.36]   ;
		\draw [shift={(41.9,209.9)}, rotate = 0.7] [color={rgb, 255:red, 0; green, 0; blue, 0 }  ][fill={rgb, 255:red, 0; green, 0; blue, 0 }  ][line width=1.5]      (0, 0) circle [x radius= 4.36, y radius= 4.36]   ;
		\draw  [dash pattern={on 0.84pt off 2.51pt}]  (428.4,214.4) -- (502.4,214.9) ;
		\draw [line width=1.5]    (186.9,211.4) -- (290.9,211.9) ;
		\draw [shift={(290.9,211.9)}, rotate = 0.28] [color={rgb, 255:red, 0; green, 0; blue, 0 }  ][fill={rgb, 255:red, 0; green, 0; blue, 0 }  ][line width=1.5]      (0, 0) circle [x radius= 4.36, y radius= 4.36]   ;
		\draw [shift={(186.9,211.4)}, rotate = 0.28] [color={rgb, 255:red, 0; green, 0; blue, 0 }  ][fill={rgb, 255:red, 0; green, 0; blue, 0 }  ][line width=1.5]      (0, 0) circle [x radius= 4.36, y radius= 4.36]   ;
		\draw [line width=1.5]    (335.4,212.9) -- (411.4,213.39) ;
		\draw [shift={(411.4,213.39)}, rotate = 0.37] [color={rgb, 255:red, 0; green, 0; blue, 0 }  ][fill={rgb, 255:red, 0; green, 0; blue, 0 }  ][line width=1.5]      (0, 0) circle [x radius= 4.36, y radius= 4.36]   ;
		\draw [shift={(335.4,212.9)}, rotate = 0.37] [color={rgb, 255:red, 0; green, 0; blue, 0 }  ][fill={rgb, 255:red, 0; green, 0; blue, 0 }  ][line width=1.5]      (0, 0) circle [x radius= 4.36, y radius= 4.36]   ;
		\draw [line width=1.5]    (520.9,214.4) -- (633.4,214.4) ;
		\draw [shift={(633.4,214.4)}, rotate = 0] [color={rgb, 255:red, 0; green, 0; blue, 0 }  ][fill={rgb, 255:red, 0; green, 0; blue, 0 }  ][line width=1.5]      (0, 0) circle [x radius= 4.36, y radius= 4.36]   ;
		\draw [shift={(520.9,214.4)}, rotate = 0] [color={rgb, 255:red, 0; green, 0; blue, 0 }  ][fill={rgb, 255:red, 0; green, 0; blue, 0 }  ][line width=1.5]      (0, 0) circle [x radius= 4.36, y radius= 4.36]   ;
		
		\draw (20,215) node [anchor=north west][inner sep=0.75pt]  [font=\footnotesize,xscale=0.8,yscale=0.8] [align=left] {$\displaystyle E_{0}$};
		\draw (132,212.5) node [anchor=north west][inner sep=0.75pt]  [font=\footnotesize,xscale=0.8,yscale=0.8] [align=left] {$\displaystyle \hat{E}_{0}$};
		\draw (166.5,216.5) node [anchor=north west][inner sep=0.75pt]  [font=\footnotesize,xscale=0.8,yscale=0.8] [align=left] {$\displaystyle E_{1}$};
		\draw (288.5,218) node [anchor=north west][inner sep=0.75pt]  [font=\footnotesize,xscale=0.8,yscale=0.8] [align=left] {$\displaystyle \hat{E}_{1}$};
		\draw (318,218.2) node [anchor=north west][inner sep=0.75pt]  [font=\footnotesize,xscale=0.8,yscale=0.8] [align=left] {$\displaystyle E_{2}$};
		\draw (413.4,216.39) node [anchor=north west][inner sep=0.75pt]  [font=\footnotesize,xscale=0.8,yscale=0.8] [align=left] {$\displaystyle \hat{E}_{2}$};
		\draw (504.4,217.9) node [anchor=north west][inner sep=0.75pt]  [font=\footnotesize,xscale=0.8,yscale=0.8] [align=left] {$\displaystyle E_{n}$};
		\draw (641.5,212.5) node [anchor=north west][inner sep=0.75pt]  [font=\footnotesize,xscale=0.8,yscale=0.8] [align=left] {$\displaystyle \hat{E}_{n}$};
		\draw (284.5,180.9) node [anchor=north west][inner sep=0.75pt]  [color={rgb, 255:red, 65; green, 117; blue, 5 }  ,opacity=1 ,xscale=0.8,yscale=0.8]  {$a_{1}$};
		\draw (155,185.9) node [anchor=north west][inner sep=0.75pt]  [color={rgb, 255:red, 74; green, 144; blue, 226}  ,opacity=1 ,xscale=0.8,yscale=0.8]  {$b_{1}$};
		\draw (285,145.9) node [anchor=north west][inner sep=0.75pt]  [color={rgb, 255:red, 74; green, 144; blue, 226}  ,opacity=1 ,xscale=0.8,yscale=0.8]  {$b_{2}$};
		\draw (395.5,123.4) node [anchor=north west][inner sep=0.75pt]  [color={rgb, 255:red, 74; green, 144; blue, 226}  ,opacity=1 ,xscale=0.8,yscale=0.8]  {$b_{n}$};
		\draw (630.5,180.9) node [anchor=north west][inner sep=0.75pt]  [color={rgb, 255:red, 65; green, 117; blue, 5}  ,opacity=1 ,xscale=0.8,yscale=0.8]  {$a_{n}$};
		\draw (403.07,180.9) node [anchor=north west][inner sep=0.75pt]  [color={rgb, 255:red, 65; green, 117; blue, 5}  ,opacity=1 ,xscale=0.8,yscale=0.8]  {$a_{2}$};	
	\end{tikzpicture}
	\caption{\footnotesize The canonical homology basis $\{a_j,b_j\}_{j=1}^n$ on the Riemann surface $\mathcal{R}$. Here the solid and dashed arcs
		indicate the parts on the upper and lower sheet, respectively. }\label{fig ajbj}
\end{figure}
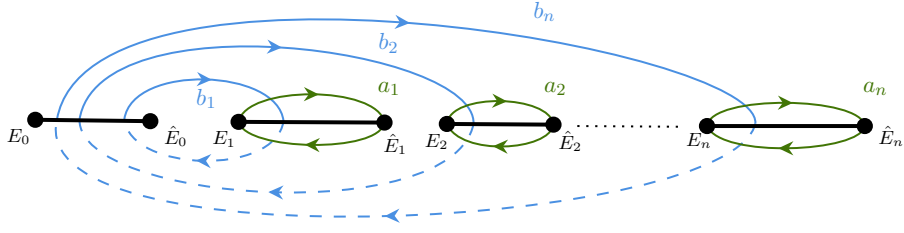
These cycles are chosen so that their intersection matrix reads
\begin{align*}
	a_i\circ b_j=\delta_{ij}, \qquad i,j=1,\ldots,n,
\end{align*}
where $\delta_{ij}$ denotes the Kronecker delta.

Let $\{\omega_j=\rho_j(z)\dif z\}_{j=1}^n$ be the canonical basis of holomorphic differentials
normalized with respect to the cycles $a_1,\ldots,a_n$, i.e.,
\begin{align*}
	\oint_{a_i} \omega_j=2\int_{E_i}^{\hat{E}_i}\rho_{j,+}(z)\dif z=\delta_{ij},
\end{align*}
where $\rho_j$ is analytic in $\mathbb{C}\setminus (\cup_{j=1}^n[E_j,\hat{E}_j])$.
The normalized holomorphic differentials then gives us the $b$-period matrix $B=(B_{ij})_{i,j=1}^n$ with
\begin{align*}
	B_{ij}:=\oint_{b_j}\omega_i=2\sum_{k=1}^{j}\int_{\hat{E}_{k-1}}^{E_k}\rho_i(z)\dif z,
\end{align*}
and the column vector of the Riemann constants
\begin{equation}\label{def:K}
\boldsymbol{K}:=(K_1,\ldots,K_n), \qquad K_j=\sum_{i=1}^nB_{ji}+\frac{j}{2}, \qquad j=1,\ldots,n.
\end{equation}

The Riemann theta function associated the matrix $B$ is defined by the Fourier series
\begin{equation}\label{def Theta}
	\Theta(\boldsymbol{\ell}):=\sum_{\boldsymbol{m}\in \mathbb{Z}^n}\exp\left(\pi i \boldsymbol{m}^TB\boldsymbol{m}+2\pi i\boldsymbol{m}^T\boldsymbol{\ell} \right) ,\qquad \boldsymbol{\ell}=(\ell_1,\ldots,\ell_n)\in\mathbb{C}^n.
\end{equation}
It is an even function with $\Theta(\boldsymbol{\ell})=\Theta(-\boldsymbol{\ell})$ and has periodicity properties
\begin{align*}
	\Theta(\boldsymbol{\ell}\pm \boldsymbol{e}_j)=\Theta(\boldsymbol{\ell}),\qquad \Theta(\boldsymbol{\ell}\pm B \boldsymbol{e}_j)=e^{\mp 2\pi i \ell_j-\pi i B_{jj}}\Theta(\boldsymbol{\ell}),
\end{align*}
where $\boldsymbol{e}_j = (0,\ldots, 0, 1, 0,\ldots, 0)$ is the $j$-th basis vector in $\mathbb{R}^n$. The Abel map $\mathcal{A}$ on $\mathcal{R}$ with the base point $\hat{P}_0=(0,\hat{E}_0)$ is defined as
\begin{equation}\label{def:Abel}
	\mathcal{A}(P):=(\mathcal{A}_1(P),\ldots, \mathcal{A}_n(P))^T,\qquad	\mathcal{A}_j(P):=\int_{\hat{P}_0}^{P}\omega_j.
\end{equation}

Recall the functions $f$ and $g$ given in \eqref{def:fg}, we also set
\begin{align}\label{def f0g0}
	f_0:=\lim_{z\to\infty}(f(z)-z),\qquad g_0:=\lim_{z\to\infty}(g(z)-2z^2),
\end{align}
and the real constants
\begin{align}\label{def:Bjfg}
	B_j^f:=\oint_{b_j}\dif f,\qquad B_j^g:=\oint_{b_j}\dif g, \qquad j=1,\ldots,n.
\end{align}

Let $\boldsymbol{\phi}=(\phi_1,\ldots,\phi_n)\in \mathbb{R}^n$ be an arbitrary real vector. Let  $\mathcal{D}=P_1+\cdots+P_n$ be the unique divisor of degree $n$ such that $P_j$, $j=1,\ldots,n$, are in the upper sheet of $\mathcal{R}$ with $\pi(P_j)\in (\hat{E}_{j-1},E_j)$ satisfying
\begin{align*}
	\prod_{i=0}^n(\pi(P_j)-E_i)=\prod_{i=0}^n(\pi(P_j)-\hat{E}_i).
\end{align*}
By \cite{AGbook,AGbook2,BA}, the finite-genus algebro-geometric solution $q^{(AG)}$ of the NLS equation \eqref{NLS} is given by
\begin{align}\label{def q AG}
q^{(AG)}(x,t;\E, \hat{\E}, \boldsymbol{\phi})&=\frac{i}{2}\left(\sum_{j=0}^{n}(\hat{E}_j-E_j) \right)e^{2i(f_0x+g_0t)}
\nonumber
\\
&\quad \times \dfrac{\Theta(\mathcal{A}(\infty)+\mathcal{A}(\mathcal{D})+\boldsymbol{K})\Theta(-\mathcal{A}(\infty)+\boldsymbol{C}(x,t;\boldsymbol{\phi})
+\mathcal{A}(\mathcal{D})+\boldsymbol{K})}{\Theta(-\mathcal{A}(\infty)+\mathcal{A}(\mathcal{D})+\boldsymbol{K})\Theta(\mathcal{A}(\infty)
+\boldsymbol{C}(x,t;\boldsymbol{\phi})+\mathcal{A}(\mathcal{D})+\boldsymbol{K})},
\end{align}
where 
$\boldsymbol{C}(x,t;\boldsymbol{\phi})=(c_1(x,t;\phi_1),\ldots,	c_n(x,t;\phi_n))^T$ is a column vector with
\begin{equation}
	c_j(x,t;\phi_j)=-\frac{1}{2\pi}\left(B_j^fx+B_j^gt+\phi_j \right) , \qquad j=1,\ldots,n.
\end{equation}

\subsection{The Jost functions}\label{sec:Jost}
The NLS equation \eqref{NLS} admits the following Lax pair \cite{ZS1}:
 \begin{equation}\label{laxx}
  \left\{
    \begin{array}{ll}
      \Psi_x = \left( - iz \sigma_3+ Q(q)\right) \Psi,  \\
      \Psi_t = \left( -2iz^2\sigma_3 +2zQ(q)-i(Q(q)^2+Q(q_x))\sigma_3\right) \Psi,
    \end{array}
  \right.
 \end{equation}
	where $\sigma_3$ is given in \eqref{def:PauliM} and
	\begin{equation*}
		Q(q)=\begin{pmatrix}
			0 & q \\
			\bar{q} & 0
		\end{pmatrix}.
	\end{equation*}

With $q=q^{(AG)}$ defined in \eqref{def q AG}, the associated Jost function $\Psi^{(AG)}$  of the Lax pair \eqref{laxx} is given by
(see \cite{BA})
\begin{align}
	\Psi^{(AG)}(z;x,t)=e^{i(f_0x+g_0t)\sigma_3}\mu^{(AG)}(z;x,t)e^{-i(f(z)x+g(z)t)\sigma_3},\label{def PsiAG}
\end{align}
where $f(z)$, $g(z)$ and $f_0$, $g_0$ are defined in \eqref{def fg} and \eqref{def f0g0}, respectively, and
\begin{align}
	\mu^{(AG)}(z;x,t)=&\frac{1}{2}\begin{pmatrix}
		(\varkappa(z)+\varkappa(z)^{-1})\frac{F_{11}(z;x,t,\boldsymbol{\phi})}{F_{11}(\infty;x,t,\boldsymbol{\phi})} & (\varkappa(z)-\varkappa(z)^{-1})\frac{F_{12}(z;x,t,\boldsymbol{\phi})}{F_{11}(\infty;x,t,\boldsymbol{\phi})}\\
		(\varkappa(z)-\varkappa(z)^{-1})\frac{F_{21}(z;x,t,\boldsymbol{\phi})}{F_{22}(\infty;x,t,\boldsymbol{\phi})} & 	(\varkappa(z)+\varkappa(z)^{-1})\frac{F_{22}(z;x,t,\boldsymbol{\phi})}{F_{22}(\infty;x,t,\boldsymbol{\phi})}
	\end{pmatrix}.\label{mu AG}
\end{align}
Here,
\begin{align}\label{def:varkappa}
	 \varkappa(z)=\left(\prod_{j=0}^n\frac{z-E_j}{z-\hat{E}_j} \right)^{1/4}
\end{align}
is analytic in $\mathbb{C}\setminus (\cup_{j=0}^n[E_j,\hat{E}_j])$ with the branch
chosen such that $\varkappa(\infty)=1$, and
\begin{align}
F_{11}(z;x,t,\boldsymbol{\phi})&
=\dfrac{\Theta(\mathcal{A}(z)+\boldsymbol{C}(x,t;\boldsymbol{\phi})+\mathcal{A}(\mathcal{D})+\boldsymbol{K})}{\Theta(\mathcal{A}(z)+\mathcal{A}(\mathcal{D})+\boldsymbol{K})},\\ 	F_{12}(z;x,t,\boldsymbol{\phi})&
=\dfrac{\Theta(-\mathcal{A}(z)+\boldsymbol{C}(x,t;\boldsymbol{\phi})+\mathcal{A}(\mathcal{D})+\boldsymbol{K})}{\Theta(-\mathcal{A}(z)+\mathcal{A}(\mathcal{D})+\boldsymbol{K})},\\ F_{21}(z;x,t,\boldsymbol{\phi})&=
\dfrac{\Theta(\mathcal{A}(z)+\boldsymbol{C}(x,t;\boldsymbol{\phi})-\mathcal{A}(\mathcal{D})-\boldsymbol{K})}{\Theta(\mathcal{A}(z)-\mathcal{A}(\mathcal{D})-\boldsymbol{K})},
\\	
F_{22}(z;x,t,\boldsymbol{\phi})&=
\dfrac{\Theta(\mathcal{A}(z)-\boldsymbol{C}(x,t;\boldsymbol{\phi})+\mathcal{A}(\mathcal{D})+\boldsymbol{K})}{\Theta(\mathcal{A}(z)+\mathcal{A}(\mathcal{D})+\boldsymbol{K})},
\label{def:F22}
\end{align}
with the same $\boldsymbol{\phi}$, $\boldsymbol{C}$, $\boldsymbol{K}$ and $\mathcal{D}$ as in \eqref{def q AG}.

By \cite{BA}, $\mu^{(AG)}$ solves the following RH problem.
\begin{RHP}\label{Pro mu0}
\hfill
	\begin{itemize}
		\item [$\mathrm{(a)}$] $\mu^{(AG)}(z)=\mu^{(AG)}(z;x,t)$ is analytic in $\mathbb{C}\setminus(\cup_{j=0}^n[E_j,\hat{E}_j])$.
		\item [$\mathrm{(b)}$]  $\mu^{(AG)}(z)$ satisfies the jump condition
$$ \mu_{+}^{(AG)}(z)=\mu_{-}^{(AG)}(z)J_0(z), \quad z\in (E_j,\hat E_j), \quad j=0,\ldots,n,$$
where
		\begin{align*}
			J_0(z)=J_0(z;x,t)=
				\begin{pmatrix}
					0 & -ie^{-i(B_j^fx+B_j^gt+\phi_j)}\\
					-ie^{i(B_j^fx+B_j^gt+\phi_j)} & 0
				\end{pmatrix}
		\end{align*}
        with $B_0^f=B_0^g=\phi_0=0$, $B_j^f, B_j^g$, $j=1,\ldots,n$, given in \eqref{def:Bjfg}.
		\item [$\mathrm{(c)}$] As $z\to\infty$, we have $\mu^{(AG)}(z)=I+\mathcal{O}(z^{-1})$.
        \item [$\mathrm{(d)}$] $\mu^{(AG)}(z)$ has at most $-1/4$-singularities near the branch points $E_j$ and $\hat{E}_j$, $j=0,1,\ldots,n$, i.e.,
        $$
        \mu^{(AG)}(z)=\mathcal{O}((z-p)^{-1/4}), \quad z\to p\in\{E_j, \hat{E}_j\}_{j=0}^n.
        $$
	\end{itemize}
\end{RHP}
In addition, the algebro-geometric solution $q^{(AG)}$ is related to $\mu^{(AG)}$  in the following way:
$$\underset{z\to\infty}{\lim}z\mu^{(AG)}_{12}(z)
=-\frac{i}{2}q^{(AG)}(x,t)e^{-2i(f_0x+g_0t)},  \quad \underset{z\to\infty}{\lim}z\mu^{(AG)}_{21}(z)=\frac{i}{2}\bar{q}^{(AG)}(x,t)e^{2i(f_0x+g_0t)}.$$

We next consider solutions of the Lax pair \eqref{laxx}  with  $q$  being the solution of Cauchy problem for the NLS equation \eqref{NLS} with the finite-genus algebro-geometric background \eqref{boundray q0}.
Let $\mu^\pm$ be solutions of  the  following Volterra integral equations:
\begin{align}\label{int mu}
	\mu^\pm(z;x,t)&= e^{i(f_0x+g_0t)\hat{\sigma}_3}\mu^{(AG)}(z;x,t)\nonumber\\
	&~~~+\int_{\pm\infty}^x\Gamma(z;s,t,x)Q(q-q^{(AG)})(s,t)\mu^\pm(z;s,t)e^{-i(f(z)-f_0)(s-x)\sigma_3}\dif s,
\end{align}
where
$\hat \sigma_3$ is defined in \eqref{def:hatsigma} and
\begin{align}\label{def Gamma}
	\Gamma(z;s,t,x)&=\Psi^{(AG)}(z;x,t)\Psi^{(AG)}(z;s,t)^{-1} \nonumber\\
	&=e^{i(f_0x+g_0t)\sigma_3}\mu^{(AG)}(z;x,t)e^{if(z)(s-x)\sigma_3}\mu^{(AG)}(z;s,t)^{-1}e^{-i(f_0s+g_0t)\sigma_3}.
\end{align}
By setting
\begin{align}\label{eq:muipm}
	\mu^\pm(z;x,t)&=\begin{pmatrix}
		\mu_{1}^\pm(z;x,t), & \mu_{2}^\pm(z;x,t)
	\end{pmatrix}, 
	\\
	\mu^{(AG)}(z;x,t)&=
	\begin{pmatrix}  \mu_{1}^{(AG)}(z;x,t), & \mu_{2}^{(AG)}(z;x,t)
		\end{pmatrix},  \label{def:muAG12}
\end{align}
where the subscripts $1$ and $2$  denote the first and the second column of a matrix, respectively, it is then readily seen that we can rewrite the integral equations in \eqref{int mu} as
\begin{align}
		\mu^\pm_1(z;x,t)=& \begin{pmatrix}
			1 & 0\\
			0 & e^{-2i(f_0x+g_0t)}
		\end{pmatrix}\mu^{(AG)}_1(z;x,t)\nonumber\\
	&+\int_{\pm\infty}^xe^{-i(f(z)-f_0)(s-x)}\Gamma(z;s,t,x)Q(q-q^{(AG)})(s,t)\mu^\pm_1(z;s,t)\dif s,\label{int mu1}\\
		\mu^\pm_2(z;x,t)=& \begin{pmatrix}
			e^{2i(f_0x+g_0t)} & 0\\
			0 & 1
		\end{pmatrix}\mu^{(AG)}_2(z;x,t)\nonumber\\
	&+\int_{\pm\infty}^xe^{i(f(z)-f_0)(s-x)}\Gamma(z;s,t,x)Q(q-q^{(AG)})(s,t)\mu^\pm_2(z;s,t)\dif s,\label{int mu2}
\end{align}
We collect some basic properties of $\mu^\pm$ in what follows.
\begin{proposition}\label{Pro mu1}
With $\mathbb{C}^{\pm}$ defined in \eqref{def:Cpm}, the following properties of $\mu_{i}^{\pm}$, $i=1,2$, in \eqref{eq:muipm} hold.
	\begin{itemize}
		\item [$\mathrm{(a)}$] $\mu^{-}_1 (z;x,t)$ and $\mu^{+}_2(z;x,t)$ exist uniquely for $z\in\overline{\mathbb{C}^{+}}\setminus \{E_j,\hat{E}_j\}_{j=0}^n$, which are
		 analytic in $\mathbb{C}^+$ and admit continuous extensions to $\overline{\mathbb{C}^{+}}\setminus \{E_j,\hat{E}_j\}_{j=0}^n$.
		\item [$\mathrm{(b)}$]  $\mu_{1}^+(z;x,t)$ and $\mu_{2}^-(z;x,t)$
		exist uniquely for $z\in\overline{\mathbb{C}^{-}}\setminus \{E_j,\hat{E}_j\}_{j=0}^n$, which are analytic for $z$ in $\mathbb{C}^-$ and admit continuous extensions to $\overline{\mathbb{C}^{-}}\setminus \{E_j,\hat{E}_j\}_{j=0}^n$.
		\item [$\mathrm{(c)}$] 
  $\det \mu^\pm(z;x,t)=1$, $z\in\mathbb{R}\setminus(\cup_{j=0}^n[E_j,\hat{E}_j])$.
		\item [$\mathrm{(d)}$]  $\mu_{i}^{\pm}$, $i=1,2$, satisfy the symmetry relations
		\begin{align}\label{eq:symm1}
\mu^\pm_1(z;x,t)=\sigma_1\overline{\mu^\pm_2(\bar{z};x,t)},\qquad
\mu^\pm_2(z;x,t)=\sigma_1\overline{\mu^\pm_1(\bar{z};x,t)},
		\end{align}
		where $\sigma_1$ is given in \eqref{def:PauliM}, provided both sides of the above equalities are well-defined.
		For $z\in (E_j,\hat{E}_j)$, $j=0,1,\ldots,n$, we have
		 \begin{align}
		 	\mu^+_{2,+}(z;x,t)&=-ie^{2i(f_0x+g_0t)-i(B_j^fx+B_j^gt+\phi_j)}\mu^+_{1,-}(z;x,t),\\ \mu^-_{2,-}(z;x,t)&=ie^{2i(f_0x+g_0t)-}i(B_j^fx+B_j^gt+\phi_j)\mu^-_{1,+}(z;x,t).
		 	\label{eq:symm2}
		 \end{align}
\item [$\mathrm{(e)}$]
As $z\to\infty$, we have
	\begin{align*}
		\begin{pmatrix}
			\mu^+_1(z;x,t)&\mu^-_2(z;x,t)
		\end{pmatrix}&=I+\mathcal{O}(z^{-1}), \qquad z\in {\mathbb{C}^{-}},
		\\
		\begin{pmatrix}
			\mu^-_1(z;x,t)&\mu^+_2(z;x,t)
		\end{pmatrix}&=I+\mathcal{O}(z^{-1}),\qquad z\in {\mathbb{C}^{+}}.
	\end{align*}
	\end{itemize}
\end{proposition}
\begin{proof}
For items (a) and (b), we only present the proof of the properties for $\mu_1^+$, since the other claims can be proved in a similar manner. By \eqref{def Gamma}, \eqref{def fg} and the RH problem \ref{Pro mu0} for $\mu^{AG}$, it is easily seen that  $\Gamma(z;s,t,x)$ is a meromorphic matrix-valued function with at most $-1/2$-singularities at each branch point $E_j,\hat{E}_j$, $j=0,1,\ldots,n$. Thus, it is actually an entire function. From the definition of $f$ in \eqref{def fg} and \eqref{def:fg}, it follows that Im$f(z)=0$ on $\mathbb{R}\setminus(\cup_{j=0}^n(E_j,\hat{E}_j))$ and 
\begin{align*}
    f'(z)=\frac{\prod_{k=0}^n(z-z_k^f)}{w(z)}.
\end{align*}
Since $ \int^{\hat{E}_j}_{E_j}\dif f=0$, $j=1,\ldots,n$, the zeros of $f'(z)$, namely, $z_j^f$, $j=0,...,n$, are distributed such that each of the  intervals $(E_j,\hat{E}_j)$ contains exactly one zero. We thus conclude that  $f'(z)\neq0$ for $z\in\mathbb{R}\setminus(\cup_{j=0}^n(E_j,\hat{E}_j))$ and  the level curve $\Im f(z)=0$ in $\mathbb{C}^\pm$, if exists, does not intersect with $\mathbb{R}\setminus(\cup_{j=0}^n(E_j,\hat{E}_j))$. Also note that $\Im f(E_j)=\Im f(\hat{E}_j)=0$, $j=1,\ldots,n$, one has $\Im f_+(z)=-\Im f_-(z) > 0$ for $z\in \cup_{j=0}^n(E_j,\hat{E}_j)$. As a consequence, the level curve $\Im f(z)=0$ in $\mathbb{C}^\pm$ does not intersect with the real axis. On the other hand, since $f(z)\sim z$ as $z\to\infty$ (see \eqref{def:fg}), it follows that $\pm \Im f(z)
> 0$ for $z\in\mathbb{C}^\pm$. Particularly, $\Im(f(z)-f_0)=\Im f(z) \leq0$ for $z\in\overline{\mathbb{C}^{-}}\setminus \{E_j,\hat{E}_j\}_{j=0}^n$, where $f_0$ is the real constant defined in  \eqref{def f0g0}.
 Therefore, for $z\in\overline{\mathbb{C}^{-}}\setminus \{E_j,\hat{E}_j\}_{j=0}^n$,  $\Gamma(z;s,t,x)e^{-i(f(z)-f_0)(s-x)}$ is uniformly bounded for $s>x$, $t>0$, which implies that $e^{-i(f(z)-f_0)(\cdot-x)}\Gamma(z;\cdot,t,x)Q(q-q^{(AG)})(\cdot,t)\in L^1(x,\infty)$. The property of $\mu_1^+$ then follows from Picard-Lindel\"of iteration of the integral equation \eqref{int mu1}.
The item (c) is a consequence of the Abel formula. For the symmetry relations in item (d), one can check that  both sides of \eqref{eq:symm1}--\eqref{eq:symm2} satisfy the same integral equation, the equalities then follow from
the uniqueness of the solutions for the Volterra  integral equations \eqref{int mu1} and \eqref{int mu2}. Finally, to show the asymptotic formulas in item (e), we observe from \eqref{eq:muipm} and the integral equation \eqref{int mu} for $\mu^-$ that
\begin{align*}
	\mu^-_1(z;x,t)&=e^{i(f_0x+g_0t)\hat{\sigma}_3}\mu^{(AG)}(z;x,t)\left(\begin{pmatrix}
		1\\0\\
	\end{pmatrix}\right.\\
	&~~\left.+\int_{-\infty}^x\mathrm{diag}(1,e^{-2i(f(z)-f_0)(s-x)})\mu^{(AG)}(z;s,t)^{-1}e^{-i(f_0s+g_0t)\sigma_3}Q(q-q^{(AG)})(s,t)\mu^-_1(z;s,t)\dif s\right).
\end{align*}
Taking $z\to\infty$ in $\mathbb{C}^{+}$, by combining  the definition of $\Gamma$ in \eqref{def Gamma} and item (c)  of RH problem \ref{Pro mu0}, we have
\begin{multline}
	e^{-if(z)(s-x)}\Gamma(z;s,t,x)Q(q-q^{(AG)})(s,t)
	\\
	=\begin{pmatrix}
		(\bar{q}-\bar{q}^{(AG)})(s,t)\mathcal{O}(z^{-1})&(q-q^{(AG)})(s,t)\mathcal{O}(1)\\
		(\bar{q}-\bar{q}^{(AG)})(s,t)\mathcal{O}(z^{-2})& (\bar{q}-\bar{q}^{(AG)})(s,t)\mathcal{O}(z^{-1})\\
	\end{pmatrix}.\label{asy Gama}
\end{multline}
From Lebesgue's dominated convergence theorem, it then follows that
\begin{align*}
	\lim_{z\to\infty}\mu^-_1(z;x,t)=\begin{pmatrix}
		1\\0\\
	\end{pmatrix}+\int_{-\infty}^xe^{-i(f_0s+g_0t)}\begin{pmatrix}
		0&(q-q^{(AG)})(s,t)\mathcal{O}(1)\\0&0\\
	\end{pmatrix}\lim_{z\to\infty}\mu^-_1(z;s,t)\dif s,
\end{align*} which implies \begin{align*}
	\lim_{z\to\infty}\mu^-_1(z;x,t)=\begin{pmatrix}
		1\\0\\
	\end{pmatrix}.
\end{align*}
Furthermore, by combining  \eqref{asy Gama} with the integral equation \eqref{int mu1},  we have that as $z\to\infty$ in $\mathbb{C}^{+}$,
\begin{align*}
	z\mu^-_{21}(z;x,t)&=z\mu^{(AG)}_{21}(z;x,t)
	\\&~~+z\int_{-\infty}^x\mathcal{O}(z^{-2})\mu^-_{11}(z;s,t) + \mathcal{O}(z^{-1})(\bar{q}-\bar{q}^{(AG)})(s,t)\mu^-_{21}(z;s,t)\dif s=\mathcal{O}(1).
\end{align*}
Therefore,
\begin{align*}
	&z(\mu^-_{11}(z;x,t)-1)=z(\mu^{(AG)}_{11}(z;x,t)-1)
	\\&+z\int_{-\infty}^x\mathcal{O}(z^{-1})(\bar{q}-\bar{q}^{(AG)})(s,t)\mu^-_{11}(z;s,t) + (q-q^{(AG)})(s,t)\mu^-_{21}(z;s,t)\mathcal{O}(1)\dif s=\mathcal{O}(1).
\end{align*}
A combination of the above three formulas yields that
\begin{equation}
	\mu^-_1(z;x,t)=\begin{pmatrix}
		1\\0\\
	\end{pmatrix}+\mathcal{O}(z^{-1}).
\end{equation}
The large $z$ asymptotics for other $\mu^\pm_j$ in item $\mathrm{(e)}$ can be proved similarly, and we omit the details here.
\end{proof}

With the aid of $\mu^\pm$ in \eqref{int mu}, we define
\begin{align}
	\Psi^\pm(z;x,t):=\mu^\pm(z;x,t)e^{-i((f(z)-f_0)x+(g(z)-g_0)t)\sigma_3},\label{def Psi pm}
\end{align}
which are two linearly independent solutions of the Lax pair \eqref{laxx}. Thus, there exists a scattering matrix $S(z)$   independent of $(x, t)$  for $z\in\mathbb{R}\setminus(\cup_{j=0}^n[E_j,\hat{E}_j])$ such that
\begin{align*}
    \Psi^+(z;x,t)=\Psi^-(z;x,t)S(z),
\end{align*}
or equivalently,
\begin{align}\label{Sz}
	\mu^+(z;x,t)=\mu^-(z;x,t)e^{-i((f(z)-f_0)x+(g(z)-g_0)t)\hat{\sigma}_3}S(z).
\end{align}
 By further using \eqref{eq:muipm}, it is readily seen that the
entries of $S(z)$ are given by
\begin{align}
	&S_{11}(z)=\det(\mu^+_1(z),\mu^-_2(z)),\quad S_{12}(z)=e^{2i((f(z)-f_0)x+(g(z)-g_0)t)}\det(\mu^+_2(z),\mu^-_2(z)),\label{def S11}\\
&S_{21}(z)=e^{-2i((f(z)-f_0)x+(g(z)-g_0)t)}\det(\mu^-_1(z),\mu^+_1(z)), \quad S_{22}(z)=\det(\mu^-_1(z),\mu^+_2(z)).\label{def S21}
\end{align} We then define the following reflection coefficients from the scattering matrix $S(z)$ for $z\in \mathbb{R}\setminus(\cup_{j=0}^n[E_j,\hat{E}_j])$
\begin{align}
	\label{def r}
	r_1(z)=\frac{S_{12}(z)}{S_{11}(z)},\quad r_2(z)=\frac{S_{21}(z)}{S_{22}(z)}, \quad r_3(z)=\frac{S_{21}(z)}{S_{11}(z)},\quad r_4(z)=\frac{S_{12}(z)}{S_{22}(z)}.
\end{align}
In view of Proposition \ref{Pro mu1}, the following proposition is immediate.

\begin{proposition}\label{Pro a r}
The scattering data and the reflection coefficients admit the following properties.
\begin{itemize}
		\item [$\mathrm{(a)}$] The scattering data  $S_{11}(z)$ and $ S_{22}(z)$ are  analytic in $\mathbb{C}^-$ and $\mathbb{C}^+$, respectively, and  admit continuous extensions to $\overline{\mathbb{C}^{-}}\setminus \{E_j,\hat{E}_j\}_{j=0}^n$ and $\overline{\mathbb{C}^{+}}\setminus \{E_j,\hat{E}_j\}_{j=0}^n$,  respectively.

\item[$\mathrm{(b)}$] The definitions of $S_{12}(z)$ and $ S_{21}(z)$  can also be extended to $\mathbb{R}\setminus \{E_j,\hat{E}_j\}_{j=0}^n$. Particularly, for $z\in \cup_{j=0}^n(E_j,\hat{E}_j)$, we have
\begin{align*}
	&S_{12}(z)=\det(e^{i((f_+(z)-f_0)x+(g_+(z)-g_0)t)}\mu^+_2(z),e^{i((f_-(z)-f_0)x+(g_-(z)-g_0)t)}\mu^-_2(z)),\\ &S_{21}(z)=\det(e^{-i((f_+(z)-f_0)x+(g_+(z)-g_0)t)}\mu^-_1(z),e^{-i((f_-(z)-f_0)x+(g_-(z)-g_0)t)}\mu^+_1(z)).
\end{align*}
\item [$\mathrm{(c)}$] For $z\in\mathbb{R}\setminus(\cup_{j=0}^n[E_j,\hat{E}_j])$, $\det S(z)=S_{11}(z)S_{22}(z)-S_{12}(z)S_{21}(z)=1$.
		\item [$\mathrm{(d)}$]
		We have
		$$S_{11}(z)=\overline{S_{22}(\bar{z})},\qquad z\in\mathbb{C}^-,$$
		and  \begin{align*}
		  S_{11}(z)=\overline{S_{22}(z)},~~  S_{12}(z)=\overline{S_{21}(z)}, ~~
		  r_1(z)=\overline{r_2(z)},~~
		  r_3(z)=\overline{r_4(z)},
		\end{align*}
		for $z\in \mathbb{R}\setminus(\cup_{j=0}^n[E_j,\hat{E}_j])$. Moreover,  if $z\in(E_j,\hat{E}_j)$, $j=0,1,\ldots,n$, we have
		\begin{align*}
			&S_{11}(z)=S_{22}(z),\qquad S_{12}(z)=e^{-2i\phi_j}S_{21}(z),\\
			&r_1(z)=e^{-2i\phi_j}r_2(z),\qquad r_3(z)=e^{2i\phi_j}r_4(z).
		\end{align*}
	\item [$\mathrm{(e)}$] As $z\to\infty$, we have
	$$S_{11}(z),\ S_{22}(z)=1+\mathcal{O}(z^{-1}), $$ and  $$r_j(z),\ S_{12}(z),\ S_{21}(z)=\mathcal{O}(z^{-1}),\qquad j=1,\ldots,4. $$
	\end{itemize}
\end{proposition}
Since $\mu^{(AG)}(z)=\mathcal{O}((z-p)^{-1/4})$ as $z\to p \in \{E_j, \hat{E}_j\}_{j=0}^n$, it follows from the integral equations \eqref{int mu} that $\mu$ also admits a $-1/4$ singularity at $p$. As $z\to p$, $S_{jk}(z)$, $j,k=1,2$, generically have $-1/2$-singularities, which corresponds to $|r_i(p)|=1$, $i=1,\ldots,4$, as stated in item (b) of Assumption \ref{asum1}. The non-generic case is that $S_{11}(z)$ is continuous at $p$, which is rarely happening as discussed in \cite{DT1979}. 
The following proposition gives us the asymptotic behaviors of the scattering matrix  and the reflection coefficients at $p$ for the generic case.
\begin{proposition}\label{Pro asy ar}
	Let $p\in\{E_j, \hat{E}_j\}_{j=0}^n$. For the generic case, we have, as
	 $z\to p \text{ from }\mathbb{R}\setminus (\cup_{j=0}^n[E_j,\hat{E}_j])$,
	 	\begin{align}
	 	&S_{11}(z), \ S_{12}(z)=\mathcal{O}((z-p)^{-1/2}),\label{eq:Sloc}  \\
	 	&r_1(z)=-e^{-i\phi_j}+\mathcal{O}((z-p)^{1/2}),\quad  r_3(z)=-e^{i\phi_j}+\mathcal{O}((z-p)^{1/2}). \label{eq:riloc}
	 \end{align}
	
\end{proposition}
\begin{proof}
Let $\nu^{\pm} (z;x,t)$ be the solution of the  integral equation
	\begin{align*}
		\nu^\pm(z;x,t)&= \begin{pmatrix}
		    1 & 0 \\
			0 & e^{-2i(f_0x+g_0t)}
		\end{pmatrix}
			\nonumber\\
		&~~ +\int_{\pm\infty}^xe^{if_0(s-x)}\Gamma(z;s,t,x)Q(q-q^{(AG)})(s,t)\nu^\pm(z;s,t)\Gamma(z;s,t,x)\dif s,
	\end{align*}
where  $\Gamma(z;s,t,x)$  is  defined in \eqref{def Gamma}. It then follows that $\nu^{\pm} (z;x,t)$  is  continuous, bounded on $\mathbb{R}$, and satisfies the symmetry relation $\sigma_1\overline{\nu^{\pm} (z;x,t)}\sigma_1=e^{-2i(f_0x+g_0t)}\nu^{\pm} (z;x,t)$. Moreover, as $z\to p \text{ from }\mathbb{C}^\mp$, 
	\begin{align*}
		\nu^{\pm} (z;x,t)= \nu^{\pm} (p;x,t)+\mathcal{O}((z-p)^{1/2}).
	\end{align*}

In view of the Volterra  integral equations \eqref{int mu1}  and \eqref{int mu2}, we note that 
	\begin{align}
	&\mu^\pm_1(z;x,t)=\nu^\pm(z;x,t)\mu^{(AG)}_1(z;x,t),\label{mu loc 1}\\
	&\mu^\pm_2(z;x,t)=e^{2i(f_0x+g_0t)}\nu^\pm(z;x,t)\mu^{(AG)}_2(z;x,t)\label{mu loc 2},
\end{align}
where $\mu^{(AG)}_i(z;x,t)$, $i=1,2$, is defined in \eqref{def:muAG12}. Since $S(z)$ is independent of $(x,t)$, we could simply take $(x,t)=(0,0)$ in  \eqref{def S11}, which leads to, for example, 
\begin{align*}
	S_{11}(z)
&=\frac{\big(\varkappa(z)^2+\varkappa(z)^{-2}\big)}{F_{11}(\infty)F_{22}(\infty)}\big(\nu^+_{11}(z)\nu^-_{22}(z)F_{11}(z)F_{22}(z)+\nu^+_{12}(z)\nu^-_{21}(z)F_{21}(z)F_{21}(z)\\
	&\qquad\qquad\qquad\qquad~~~-\nu^+_{21}(z)\nu^-_{12}(z)F_{11}(z)F_{22}(z)-\nu^+_{22}(z)\nu^-_{11}(z)F_{21}(z)F_{21}(z)\big)\\
	&~~~+\big(\varkappa(z)^2-\varkappa(z)^{-2}\big)\big(\nu^+_{11}(z)\nu^-_{21}(z)\frac{F_{11}(z)F_{12}(z)}{F_{11}^2(\infty)}+\nu^+_{12}(z)\nu^-_{22}(z)\frac{F_{21}(z)F_{22}(z)}{F_{22}^2(\infty)}\\
	&\qquad\qquad\qquad\qquad~~~-\nu^+_{21}(z)\nu^-_{11}(z)\frac{F_{11}(z)F_{12}(z)}{F_{11}^2(\infty)}-\nu^+_{22}(z)\nu^-_{12}(z)\frac{F_{21}(z)F_{22}(z)}{F_{22}^2(\infty)}\big)\\
	&~~~+\frac{2}{F_{11}(\infty)F_{22}(\infty)}\big(\nu^+_{11}(z)\nu^-_{22}(z)F_{11}(z)F_{22}(z)-\nu^+_{12}(z)\nu^-_{21}(z)F_{21}(z)F_{21}(z)\\
	&\qquad\qquad\qquad\qquad~~~+\nu^+_{21}(z)\nu^-_{12}(z)F_{11}(z)F_{22}(z)-\nu^+_{22}(z)\nu^-_{11}(z)F_{21}(z)F_{21}(z)\big),
\end{align*}
where $\varkappa$ and $F_{ik}$, $i,k=1,2$, are given in  \eqref{def:varkappa}--\eqref{def:F22}. We thus conclude \eqref{eq:Sloc}. 


To show the asymptotics of $r_1$ near $p$, we see from its definition in \eqref{def r} and \eqref{eq:Sloc} that $r_1$ is continuous at $p$. 
In addition, item (c) of Proposition \ref{Pro a r} implies  that  as $z\to p$  from $\mathbb{R}\setminus (\cup_{j=0}^n[E_j,\hat{E}_j])$,
\begin{align*}
 	\left|r_1(z)^2\right|=1-\frac{1}{\left|S_{11}(z)^2\right|}\to1.
 \end{align*}
This, together with item (d) of Proposition \ref{Pro a r}, gives us the asymptotics of $r_1$ in \eqref{eq:riloc}. 
The result for $r_3$ can be proved similarly, we omit the details here. 
\end{proof}

\begin{remark}\label{Remark 3.4}
Under item (a) of Assumption \ref{asum1}, $\mu^\pm(z;s,t)$ are analytic on $ \mathbb{C}\setminus (\cup_{j=0}^n[E_j,\hat{E}_j])$ since the  integral interval in the integral equations in \eqref{int mu} of $\mu^\pm(z;s,t)$ are bounded. Thus, the scattering matrix $S(z)$  is also analytic on $ \mathbb{C}\setminus (\cup_{j=0}^n[E_j,\hat{E}_j])$. The reflection coefficients will then have poles on the zeros of $S_{11}(z)$ and $S_{22}(z)$. However, the zeros of $S_{11}(z)$ and $S_{22}(z)$ on $\mathbb{C}\setminus\mathbb{R}$ are the eigenvalues of the $x$-part of the Lax pair \eqref{laxx}. Since the  spectral operator of the $x$-part of the Lax pair \eqref{laxx} is self-adjoint, its eigenvalues must be real, which implies that  $S_{11}(z), S_{22}(z)\neq 0$ on $\mathbb{C}\setminus\mathbb{R}$. By items (c) and (d) of Proposition \ref{Pro a r},  we note that $|S_{11}|^2=|S_{22}|^2\geq1$  on  $ \mathbb{R}\setminus (\cup_{j=0}^n[E_j,\hat{E}_j])$.  Therefore, the zeros of  $S_{11}(z)$ and $ S_{22}(z)$ will locate on $\cup_{j=0}^n(E_j,\hat{E}_j)$ for the generic case, which might lead to  the spectrum singularity of the Lax pair \eqref{laxx}. We exclude this situation in  item (b) of Assumption \ref{asum1}.   
\end{remark}


\subsection{Two RH problems}
With the aid of the functions $\mu_1^{\pm}$, $\mu_2^{\pm}$, $S_{11}$ and $S_{12}$, we define
\begin{align}
\label{def M}	
M(z)&=M(z;x,t):=	\left\{ \begin{array}{ll}
		\left(  \frac{\mu^-_1  (z;x,t)}{S_{22}(z)}  , \  \mu^+_2(z;x,t)\right), \qquad   z\in \mathbb{C}^+,\\
		\left( \mu^+_1(z;x,t), \ \frac{\mu^-_2(z;x,t)}{S_{11}(z)}\right)  , \qquad z\in \mathbb{C}^-,\\
	\end{array}\right.\\
	\label{def N}	N(z)&=N(z;x,t) :=	\left\{ \begin{array}{ll}
		\left(  \mu^-_1  (z;x,t)  , \  \frac{\mu^+_2(z;x,t)}{S_{22}(z)}\right),   \qquad  z\in \mathbb{C}^+,\\
		\left( \frac{\mu^+_1(z;x,t)}{S_{11}(z)}, \ \mu^-_2(z;x,t)\right)  , \qquad z\in \mathbb{C}^-.\\
	\end{array}\right.
\end{align}
It is then readily seen that $M$ satisfies the following RH problem.
\begin{RHP}\label{RHP1}
\hfil
		\begin{itemize}
		\item [$\mathrm{(a)}$] $M(z)$ is analytic in $\mathbb{C}\setminus \mathbb{R}$.
		\item [$\mathrm{(b)}$]  
           $M(z)$  satisfies the jump condition $M_{+}(z)=M_{-}(z)J(z)$ with
		\begin{align}\label{def:JumpM}
			J(z)=	\left\{ \begin{array}{ll}
			e^{i(f_0x+g_0t-(B_j^fx+B_j^gt+\phi_j)/2)\hat{\sigma}_3}\begin{pmatrix}
			 0&-i\\
			-i&0   
			\end{pmatrix}, & z\in (E_j,\hat{E}_j),\ j=0,\ldots,n,\\[12pt]
		\begin{pmatrix}
			1-r_1(z)r_2(z) & r_1(z)e^{2it\theta(z)}\\
			-r_2(z)e^{-2it\theta(z)} & 1
		\end{pmatrix}, & z\in \mathbb{R}\setminus (\cup_{j=0}^n[E_j,\hat{E}_j]),\\	\end{array}\right.
		\end{align}
		where
		\begin{align}
			\theta(z)=\theta(z;\xi):=-(f(z)-f_0)\xi-\left(g(z)-g_0 \right),\quad\xi=\frac{x}{t} \label{theta}.
		\end{align}
		\item [$\mathrm{(c)}$] As $z\to\infty$, we have $M(z)=I+\mathcal{O}(z^{-1})$.
		\item [$\mathrm{(d)}$] Let $p\in\{E_j,\hat{E}_j\}_{j=0}^n$, we have 
		\begin{align*}
			M(z)=(\mathcal{O}((z-p)^{1/4}),\mathcal{O}((z-p)^{-1/4})),\qquad z\to p\text{ from }\mathbb{C}^+,\\
			M(z)=(\mathcal{O}((z-p)^{-1/4}),\mathcal{O}((z-p)^{1/4})),\qquad z\to p \text{ from }\mathbb{C}^-.
  \end{align*}
	\end{itemize}
\end{RHP}
Similarly, $N$ solves the following RH problem. 
\begin{RHP} \label{RHP2}
\hfil
	\begin{itemize}
		\item [$\mathrm{(a)}$] $N(z)$ is analytic in $\mathbb{C}\setminus \mathbb{R}$. 
		\item [$\mathrm{(b)}$]    $N(z)$  satisfies the jump condition $N_{+}(z)=N_{-}(z)\tilde{J}(z)$ with
		\begin{align}\label{def:tildeJ}
				\tilde{J}(z )=	\left\{ \begin{array}{ll}
				e^{i(f_0x+g_0t-(B_j^fx+B_j^gt+\phi_j)/2)\hat{\sigma}_3}\begin{pmatrix}
					0 & -i\\
					-i& 0
				\end{pmatrix}, & z\in (E_j,\hat{E}_j),\ j=0,\ldots,n,
    \\
    [12pt]
				\begin{pmatrix}
					1 & r_4(z)e^{2it\theta(z)}\\
					-r_3(z)e^{-2it\theta(z)} & 1-r_3(z)r_4(z)
				\end{pmatrix}, & z\in \mathbb{R}\setminus (\cup_{j=0}^n[E_j,\hat{E}_j]),\\	\end{array}\right.
		\end{align}
        where $\theta$ is given in \eqref{theta}. 
		\item [$\mathrm{(c)}$] As $z\to\infty$, we have $N(z)=I+\mathcal{O}(z^{-1})$.
		\item [$\mathrm{(d)}$]  Let $p\in\{E_j,\hat{E}_j\}_{j=0}^n$, we have
			\begin{align*}
			N(z)=(\mathcal{O}((z-p)^{-1/4}),\mathcal{O}((z-p)^{1/4})),\ z\to p  \text{ from }\mathbb{C}^+,\\
			N(z)=(\mathcal{O}((z-p)^{1/4}),\mathcal{O}((z-p)^{-1/4})),\ z\to p \text{ from }\mathbb{C}^-.
		\end{align*}
	\end{itemize}
\end{RHP}
By Liouville's theorem, the above two RH problems are uniquely solvable. Moreover, they are related to the NLS equation in the following way.


\begin{proposition}\label{prop:reconstuction}
Let $M$ and $N$ be solutions of RH problems \ref{RHP1}  and \ref{RHP2},  respectively,  we have
	\begin{align}
	\lim_{z\to\infty}zM_{12}(z)=\lim_{z\to\infty}zN_{12}(z)=-\frac{i}{2}q,\qquad \lim_{z\to\infty}zM_{21}(z)=\lim_{z\to\infty}zN_{21}(z)=\frac{i}{2}\bar{q}	\label{res q},
	\end{align}
where the limit is taken nontangentially with respect to $\mathbb{R}$ and $q$ is the solution of the NLS equation \eqref{NLS}.
\end{proposition}
\begin{proof} 
By the constructions of $M$ and $N$ in \eqref{def M}--\eqref{def N} and the definitions of $S_{11}$ and $S_{22}$  in \eqref{def S11}, it is equivalent to consider the large $z$ behaviors of $\mu_\pm$. From Remark \ref{Remark 3.4}, it follows that  $\mu_\pm$ is analytic near $z=\infty$. This, together with item (e) of Proposition \ref{Pro mu1}, implies that $\mu_\pm\sim I+\sum_{j=1}^\infty\frac{V_{j}^\pm}{z^j}$   as $z\to\infty$. Inserting this  ansatz into the equation \eqref{int mu}, it is readily seen that
\begin{align*}
		 (V^\pm_{1})_{12}=-\frac{i}{2}q,\qquad \left(V^\pm_{1}\right)_{21}=\frac{i}{2}\bar{q}.
		 \end{align*} 
This completes the proof.
\end{proof}
We will perform asymptotic analysis of RH problems \ref{RHP1}  and \ref{RHP2} in Sections \ref{Sec 3} and \ref{Sec 4}, respectively, to derive asymptotics of $q$ in  transition regions I and II.


\subsection{Signature table  and  matrix factorizations}\label{subsec2.3}

For later use, we record several matrix factorizations of the jump matrices in RH problems \ref{RHP1}  and \ref{RHP2} that we will refer to when deforming contours onto steepest descent paths according to the signature table of phase function $\theta$.

\paragraph{\textbf{The signature table of $\Im \theta$}} 
Recall the phase function $\theta$ defined in \eqref{theta}, we start with the following definition.
\begin{definition}\label{def z0}
	We call a point $z_0=z_0(\xi) \in \mathbb{R}$ is a  saddle point of $\theta$ if it satisfies one of the following conditions
	\begin{itemize}
		\item [\rm(a) ] $\theta'(z_0)=0$ for $z_0\in \mathbb{R}\setminus (\cup_{j=0}^n(E_j,\hat{E}_j))$;
		\item [\rm(b) ] $\Im \theta(z_0)=0$ for $z_0\in \cup_{j=0}^n(E_j,\hat{E}_j)$.
	\end{itemize}
\end{definition}

From \eqref{theta} and \eqref{def fg}, it follows that
\begin{align*}
	\theta'(z)=-\frac{1}{w(z)}\left[\xi\prod_{j=0}^n(z-z^f_j)+4\prod_{j=0}^{n+1}(z-z^g_j)\right],
\end{align*}
whose zeros are  the same with  the equation
\begin{align}\label{def F}
	F(z;\xi):=\xi\prod_{j=0}^n(z-z^f_j)+4\prod_{j=0}^{n+1}(z-z^g_j)=0.
\end{align}
Since $f(E_j)=f(\hat{E}_j)$ and $g(E_j)=g(\hat{E}_j)$ for $j=0,1,\ldots,n$, we have $\theta(E_j)=\theta(\hat{E}_j)$. This implies that $F$ has at least one zero on each interval $(E_j,\hat{E}_j)$. Note that $F$ is a polynomial of degree $n+2$, we still have one zero left. If this zero lies outside the union of these intervals, namely $\mathbb{R}\setminus (\cup_{j=0}^n(E_j,\hat{E}_j))$, it  corresponds to case (a) in Definition \ref{def z0}. On the other hand, if there are two zeros within any single interval $(E_j,\hat{E}_j)$, it corresponds to case (b) in  Definition \ref{def z0}. Notably, $\Im \theta(E_j)=\Im \theta(\hat{E}_j)=0$, which indicates that the two zeros are distinct rather than a double zero, otherwise $\theta'(z)$ maintains its sign, leading to a contradiction. 

For case (a), $z_0=z_0(\xi)\in \mathbb{R}\setminus (\cup_{j=0}^n(E_j,\hat{E}_j))$. In view of the facts
\begin{align}\label{F xij}
F(E_{j};\xi_{j})=F(\hat{E}_{j};\hat{\xi}_{j})=0, \qquad j=0,1,\ldots,n,
\end{align}
where $\xi_{j}$ and $\hat{\xi}_{j}$ are defined in \eqref{def xij}, it follows that $z_0(\hat{\xi}_{j})=\hat{E}_{j}$ and  $z_0(\xi_{j})=E_{j}$.  Moreover, on each interval of $ \mathbb{R}\setminus (\cup_{j=0}^n(E_j,\hat{E}_j))$, $z_0(\xi)$ is a monotonically decreasing function with respect to $\xi$. Indeed, by implicit function theorem, taking the interval $(\hat{E}_j,E_{j+1})$ as an example, we have
\begin{align*}
	\partial_\xi z_0(\xi)=-\frac{	\partial_\xi F(z_0;\xi)}{\partial_{z_0} F(z_0;\xi)}=-\frac{	\prod_{k=0}^n(z_0-z^f_k)}{\partial_{z_0} F(z_0;\xi)}.
\end{align*}
If $n-j$ is even, both $\partial_{z_0} F(z_0;\xi)$ and 
$\prod_{k=0}^n(z_0-z^f_k)$ are positive, while if $n-j$ is odd, both $\partial_{z_0} F(z_0;\xi)$ and $\prod_{k=0}^n(z_0-z^f_k)$ are negative, we thus conclude that $\partial_\xi z_0(\xi)<0$ and 
\begin{align*}
	&\xi\mapsto z_0(\xi),\\
	& \cup_{j=1}^{n}[\xi_j,\hat{\xi}_{j-1}]\cup(-\infty,\hat{\xi}_n]\cup[\xi_0,+\infty)\to \mathbb{R}\setminus \cup_{j=0}^n(E_j,\hat{E}_j),
\end{align*}
is a monotonically decreasing map. 

Similarly, for case (b), $\xi \mapsto z_0(\xi)$ is continuous and monotonically decreasing from $\cup_{j=0}^{n}(\hat{\xi}_j,\xi_{j})$ to $\cup_{j=0}^n(E_j,\hat{E}_j)$. Moreover, $z_0(\xi)$ tends to $E_j$ and $\hat{E}_j$ if $\xi$ tends to  $\xi_{j}$
and $\hat{\xi}_{j}$, respectively. In summary, we have that
$\xi\to z_0(\xi):\ \mathbb{R}\to\mathbb{R}$ is a continuous map; see Figure \ref{sign theta} for the signature table of $\Im \theta$.


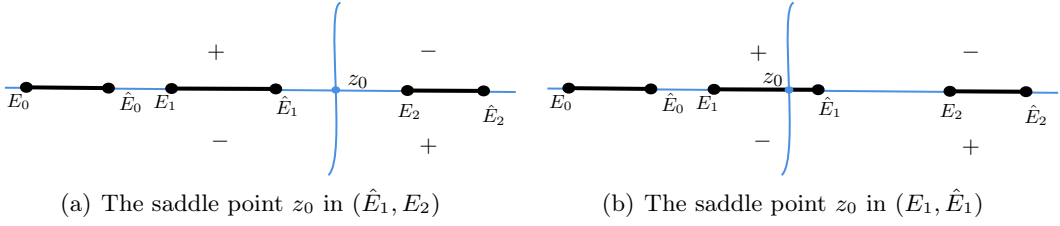
\begin{figure}
	\centering\subfigure[The saddle point $z_0$ in $(\hat{E}_1,E_2)$]{
		\tikzset{every picture/.style={line width=0.75pt}} 
		\begin{tikzpicture}[x=0.75pt,y=0.75pt,yscale=-0.4,xscale=0.5]
			\draw [color={rgb, 255:red, 74; green, 144; blue, 226 }  ,draw opacity=1 ]   (20.47,210.4) -- (530.47,215.8) ;
			\draw [line width=1.5]    (41.9,209.9) -- (123.9,210.9) ;
			\draw [shift={(123.9,210.9)}, rotate = 0.7] [color={rgb, 255:red, 0; green, 0; blue, 0 }  ][fill={rgb, 255:red, 0; green, 0; blue, 0 }  ][line width=1.5]      (0, 0) circle [x radius= 4.36, y radius= 4.36]   ;
			\draw [shift={(41.9,209.9)}, rotate = 0.7] [color={rgb, 255:red, 0; green, 0; blue, 0 }  ][fill={rgb, 255:red, 0; green, 0; blue, 0 }  ][line width=1.5]      (0, 0) circle [x radius= 4.36, y radius= 4.36]   ;

			\draw [line width=1.5]    (186.9,211.4) -- (290.9,211.9) ;
			\draw [shift={(290.9,211.9)}, rotate = 0.28] [color={rgb, 255:red, 0; green, 0; blue, 0 }  ][fill={rgb, 255:red, 0; green, 0; blue, 0 }  ][line width=1.5]      (0, 0) circle [x radius= 4.36, y radius= 4.36]   ;
			\draw [shift={(186.9,211.4)}, rotate = 0.28] [color={rgb, 255:red, 0; green, 0; blue, 0 }  ][fill={rgb, 255:red, 0; green, 0; blue, 0 }  ][line width=1.5]      (0, 0) circle [x radius= 4.36, y radius= 4.36]   ;
			\draw [line width=1.5]    (422.4,213.9) -- (498.4,214.39) ;
			\draw [shift={(498.4,214.39)}, rotate = 0.37] [color={rgb, 255:red, 0; green, 0; blue, 0 }  ][fill={rgb, 255:red, 0; green, 0; blue, 0 }  ][line width=1.5]      (0, 0) circle [x radius= 4.36, y radius= 4.36]   ;
			\draw [shift={(422.4,213.9)}, rotate = 0.37] [color={rgb, 255:red, 0; green, 0; blue, 0 }  ][fill={rgb, 255:red, 0; green, 0; blue, 0 }  ][line width=1.5]      (0, 0) circle [x radius= 4.36, y radius= 4.36]   ;
			\draw [color={rgb, 255:red, 74; green, 144; blue, 226 }  ,draw opacity=1 ]   (355.47,103.4) .. controls (339.47,108.4) and (362.47,298.4) .. (343.47,320.4) ;
			\draw [shift={(350.79,213.17)}, rotate = 88.57] [color={rgb, 255:red, 74; green, 144; blue, 226 }  ,draw opacity=1 ][fill={rgb, 255:red, 74; green, 144; blue, 226 }  ,fill opacity=1 ][line width=0.75]      (0, 0) circle [x radius= 3.35, y radius= 3.35]   ;
			
			\draw (20,215) node [anchor=north west][inner sep=0.75pt]  [font=\footnotesize,xscale=0.8,yscale=0.8] [align=left] {$\displaystyle E_{0}$};
			\draw (132,212.5) node [anchor=north west][inner sep=0.75pt]  [font=\footnotesize,xscale=0.8,yscale=0.8] [align=left] {$\displaystyle \hat{E}_{0}$};
			\draw (166.5,216.5) node [anchor=north west][inner sep=0.75pt]  [font=\footnotesize,xscale=0.8,yscale=0.8] [align=left] {$\displaystyle E_{1}$};
			\draw (288.5,218) node [anchor=north west][inner sep=0.75pt]  [font=\footnotesize,xscale=0.8,yscale=0.8] [align=left] {$\displaystyle \hat{E}_{1}$};
			\draw (410,227.2) node [anchor=north west][inner sep=0.75pt]  [font=\footnotesize,xscale=0.8,yscale=0.8] [align=left] {$\displaystyle E_{2}$};
			\draw (494.4,228.39) node [anchor=north west][inner sep=0.75pt]  [font=\footnotesize,xscale=0.8,yscale=0.8] [align=left] {$\displaystyle \hat{E}_{2}$};
			\draw (360,191.4) node [anchor=north west][inner sep=0.75pt]  [xscale=0.8,yscale=0.8]  {$z_{0}$};
			\draw (220,153.4) node [anchor=north west][inner sep=0.75pt]  [xscale=0.8,yscale=0.8]  {$+$};
			\draw (432,271.4) node [anchor=north west][inner sep=0.75pt]  [xscale=0.8,yscale=0.8]  {$+$};
			\draw (432,151.4) node [anchor=north west][inner sep=0.75pt]  [xscale=0.8,yscale=0.8]  {$-$};
			\draw (224,264.4) node [anchor=north west][inner sep=0.75pt]  [xscale=0.8,yscale=0.8]  {$-$};

	\end{tikzpicture}}
\subfigure[The saddle point $z_0$ in $(E_1,\hat{E}_1)$]{
	
	\tikzset{every picture/.style={line width=0.75pt}} 
	
	\begin{tikzpicture}[x=0.75pt,y=0.75pt,yscale=-0.4,xscale=0.5]
		\draw [color={rgb, 255:red, 74; green, 144; blue, 226 }  ,draw opacity=1 ]   (20.47,210.4) -- (530.47,215.8) ;
		\draw [line width=1.5]    (41.9,209.9) -- (123.9,210.9) ;
		\draw [shift={(123.9,210.9)}, rotate = 0.7] [color={rgb, 255:red, 0; green, 0; blue, 0 }  ][fill={rgb, 255:red, 0; green, 0; blue, 0 }  ][line width=1.5]      (0, 0) circle [x radius= 4.36, y radius= 4.36]   ;
		\draw [shift={(41.9,209.9)}, rotate = 0.7] [color={rgb, 255:red, 0; green, 0; blue, 0 }  ][fill={rgb, 255:red, 0; green, 0; blue, 0 }  ][line width=1.5]      (0, 0) circle [x radius= 4.36, y radius= 4.36]   ;
		\draw [line width=1.5]    (186.9,211.4) -- (290.9,211.9) ;
		\draw [shift={(290.9,211.9)}, rotate = 0.28] [color={rgb, 255:red, 0; green, 0; blue, 0 }  ][fill={rgb, 255:red, 0; green, 0; blue, 0 }  ][line width=1.5]      (0, 0) circle [x radius= 4.36, y radius= 4.36]   ;
		\draw [shift={(186.9,211.4)}, rotate = 0.28] [color={rgb, 255:red, 0; green, 0; blue, 0 }  ][fill={rgb, 255:red, 0; green, 0; blue, 0 }  ][line width=1.5]      (0, 0) circle [x radius= 4.36, y radius= 4.36]   ;
		\draw [line width=1.5]    (422.4,213.9) -- (498.4,214.39) ;
		\draw [shift={(498.4,214.39)}, rotate = 0.37] [color={rgb, 255:red, 0; green, 0; blue, 0 }  ][fill={rgb, 255:red, 0; green, 0; blue, 0 }  ][line width=1.5]      (0, 0) circle [x radius= 4.36, y radius= 4.36]   ;
		\draw [shift={(422.4,213.9)}, rotate = 0.37] [color={rgb, 255:red, 0; green, 0; blue, 0 }  ][fill={rgb, 255:red, 0; green, 0; blue, 0 }  ][line width=1.5]      (0, 0) circle [x radius= 4.36, y radius= 4.36]   ;
		\draw [color={rgb, 255:red, 74; green, 144; blue, 226 }  ,draw opacity=1 ]   (266.47,102.4) .. controls (250.47,107.4) and (273.47,297.4) .. (254.47,319.4) ;
		\draw [shift={(261.79,212.17)}, rotate = 88.57] [color={rgb, 255:red, 74; green, 144; blue, 226 }  ,draw opacity=1 ][fill={rgb, 255:red, 74; green, 144; blue, 226 }  ,fill opacity=1 ][line width=0.75]      (0, 0) circle [x radius= 3.35, y radius= 3.35]   ;
		
		\draw (20,215) node [anchor=north west][inner sep=0.75pt]  [font=\footnotesize,xscale=0.8,yscale=0.8] [align=left] {$\displaystyle E_{0}$};
		\draw (132,212.5) node [anchor=north west][inner sep=0.75pt]  [font=\footnotesize,xscale=0.8,yscale=0.8] [align=left] {$\displaystyle \hat{E}_{0}$};
		\draw (166.5,216.5) node [anchor=north west][inner sep=0.75pt]  [font=\footnotesize,xscale=0.8,yscale=0.8] [align=left] {$\displaystyle E_{1}$};
		\draw (288.5,218) node [anchor=north west][inner sep=0.75pt]  [font=\footnotesize,xscale=0.8,yscale=0.8] [align=left] {$\displaystyle \hat{E}_{1}$};
		\draw (410,227.2) node [anchor=north west][inner sep=0.75pt]  [font=\footnotesize,xscale=0.8,yscale=0.8] [align=left] {$\displaystyle E_{2}$};
		\draw (494.4,228.39) node [anchor=north west][inner sep=0.75pt]  [font=\footnotesize,xscale=0.8,yscale=0.8] [align=left] {$\displaystyle \hat{E}_{2}$};
		\draw (233,187.4) node [anchor=north west][inner sep=0.75pt]  [xscale=0.8,yscale=0.8]  {$z_{0}$};
		\draw (220,153.4) node [anchor=north west][inner sep=0.75pt]  [xscale=0.8,yscale=0.8]  {$+$};
		\draw (432,271.4) node [anchor=north west][inner sep=0.75pt]  [xscale=0.8,yscale=0.8]  {$+$};
		\draw (432,151.4) node [anchor=north west][inner sep=0.75pt]  [xscale=0.8,yscale=0.8]  {$-$};
		\draw (224,264.4) node [anchor=north west][inner sep=0.75pt]  [xscale=0.8,yscale=0.8]  {$-$};
	
\end{tikzpicture}}
\caption{\footnotesize  Two examples for the signature table of $\Im \theta$: (a) $\xi\in(\xi_2,\hat{\xi}_1)$; (b) $\xi\in(\hat{\xi}_1,\xi_1)$. $\Im \theta =0 $ on blue curve. The ``+'' represents where
$\Im \theta>0$ and ``-'' represents where
$\Im \theta<0$. }
\label{sign theta}
\end{figure}

\vspace{2mm}
\paragraph{\textbf{Matrix factorizations of the jump matrices}}
Recall the jump matrices $J$ and $\tilde{J}$ defined in \eqref{def:JumpM} and \eqref{def:tildeJ}, respectively, we note that they admit the following factorizations. 
\begin{itemize}
	\item If $z\in \cup_{j=0}^n(E_j,\hat{E}_j)$, 
	\begin{align}
	J(z)& =	\begin{pmatrix}
			1&r_1(z)e^{2it\theta_-(z)}\\
			0&1
		\end{pmatrix}	\begin{pmatrix}
			0&-ie^{i(t(\theta_+(z)+\theta_-(z))-\phi_j)}\\
			-ie^{-i(t(\theta_+(z)+\theta_-(z))-\phi_j)}&0
		\end{pmatrix}	
       \nonumber 
        \\  
        &\quad \times \begin{pmatrix} 
		1&0\\
		-r_2(z)e^{-2it\theta_+(z)}&1
	\end{pmatrix}\label{open jump J1}\\
&=	\begin{pmatrix}
	1&0\\
	-\frac{r_2(z)e^{-2it\theta_-(z)}}{1-r_1(z)r_2(z)}&1
\end{pmatrix}	\begin{pmatrix}
	0&-ie^{i(t(\theta_+(z)+\theta_-(z))-\phi_j)}\\
	-ie^{-i(t(\theta_+(z)+\theta_-(z))-\phi_j)}&0
\end{pmatrix}	\nonumber 
\\ 
&\quad \times \begin{pmatrix}
	1&\frac{r_1(z)e^{2it\theta_+(z)}}{1-r_1(z)r_2(z)}\\
	0&1
\end{pmatrix}.\label{open jump J2}
	\end{align}
\item If $z\in \mathbb{R}\setminus (\cup_{j=0}^n[E_j,\hat{E}_j])$,
	\begin{align}
		J(z)&=
		\begin{pmatrix}
		1&r_1(z)e^{2it\theta(z)}\\
		0&1
	\end{pmatrix}		\begin{pmatrix}
		1&0\\
		-r_2(z)e^{-2it\theta(z)}&1
	\end{pmatrix}\label{open jump J3}\\
	&=	\begin{pmatrix}
		1&0\\
		-\frac{r_2(z)e^{-2it\theta(z)}}{1-r_1(z)r_2(z)}&1
	\end{pmatrix}	(1-r_1(z)r_2(z))^{\sigma_3}	
    \begin{pmatrix}
		1&\frac{r_1(z)e^{2it\theta(z)}}{1-r_1(z)r_2(z)}
        \\
		0&1
	\end{pmatrix}.\label{open jump J4}
\end{align}
\end{itemize}
In the above factorizations, we have made use of 
the symmetry relations of the  reflection coefficients in item (d) of Proposition \ref{Pro a r} and the fact that for $z\in(E_j,\hat{E}_j)$, $j=0,1,\ldots,n$,
\begin{align}
	\theta_+(z)+\theta_-(z)=2f_0\xi+2g_0- B_j^f\xi-B_j^g.\label{theta Ej}
\end{align} 

Similarly, for the matrix $\tilde{J}$, one has
\begin{itemize}
	\item If $z\in \cup_{j=0}^n(E_j,\hat{E}_j)$,
	\begin{align}
		\tilde{J}(z) & =
		\begin{pmatrix}
			1&0\\
			-r_3(z)e^{-2it\theta_-(z)}&1
		\end{pmatrix}	\begin{pmatrix} 
		0&-ie^{i(t(\theta_+(z)+\theta_-(z))-\phi_j)}\\
		-ie^{-i(t(\theta_+(z)+\theta_-(z))-\phi_j)}&0
	\end{pmatrix} 
    \nonumber 
    \\
    & \quad \times \begin{pmatrix}
			1&r_4(z)e^{2it\theta_+(z)}\\
			0&1
		\end{pmatrix}\label{open jump tJ1}\\
		&=	\begin{pmatrix}
			1&\frac{r_4(z)e^{2it\theta_-(z)}}{1-r_3(z)r_4(z)}\\
			0&1
		\end{pmatrix}\begin{pmatrix}
		0&-ie^{i(t(\theta_+(z)+\theta_-(z))-\phi_j)}\\
		-ie^{-i(t(\theta_+(z)+\theta_-(z))-\phi_j)}&0
	\end{pmatrix}\nonumber
    \\
    & \quad \times \begin{pmatrix}
			1&0\\
			-\frac{r_3(z)e^{-2it\theta_+(z)}}{1-r_3(z)r_4(z)}0&1
		\end{pmatrix}.\label{open jump tJ2}
	\end{align}
	\item If $z\in \mathbb{R}\setminus (\cup_{j=0}^n[E_j,\hat{E}_j])$,
	\begin{align}
		\tilde J(z)&=
		 	\begin{pmatrix}
		1&0\\
		-r_3(z)e^{-2it\theta(z)}&1
		\end{pmatrix}\begin{pmatrix}
			1&r_4(z)e^{2it\theta(z)}\\
		0&1	
		\end{pmatrix}\label{open jump tJ3}\\
		&=	\begin{pmatrix}
		1&\frac{r_4(z)e^{-2it\theta(z)}}{1-r_3(z)r_4(z)}\\
		0&1
		\end{pmatrix}	(1-r_3(z)r_4(z))^{-\sigma_3}	\begin{pmatrix}
			1&0\\
		-\frac{r_3(z)e^{-2it\theta(z)}}{1-r_3(z)r_4(z)}&1	
		\end{pmatrix}.\label{open jump tJ4}
	\end{align}
\end{itemize}

\section{Asymptotic analysis in transition  region  I}\label{Sec 3}
In this section, we perform asymptotic analysis of RH problem \ref{RHP1} for $M$ to derive asymptotics of $q$ in transition region I. By Definition \ref{def:regions}, it is assumed that $-C\leq(\xi-\hat{\xi}_{j_0})t^{2/3}\leq0$ for some fixed $j_0\in\{0,\ldots,n\}$ throughout this section, since the analysis in the other half region is similar. Note that $\xi_{j_0+1}<\xi<\hat{\xi}_{j_0}$ for large $t$, it then follows from discussions in Section \ref{subsec2.3} that the saddle point $z_0\in[\hat{E}_{j_0},E_{j_0+1})$.


\subsection{First transformation}
The first transformation involves opening global lenses. To proceed, we start with the introduction of an auxiliary function. For  $-C\leq(\xi-\hat{\xi}_{j_0})t^{2/3}\leq0$, we set
 \begin{multline}\label{def delta 1}
 \delta(z)=\delta(z;\xi)
 \\
 :=\frac{w(z)}{2\pi i}\left[\sum_{j=1}^n\delta_j\int_{E_j}^{\hat{E}_j}\frac{i\dif s}{w_+(s)(s-z)} -\int_{(-\infty,E_{j_0})\setminus(\cup_{j=0}^{j_0-1}(E_j,\hat{E}_j))}\frac{\log (1-r_1(s)r_2(s))\dif s}{w(s)(s-z)}\right],
 \end{multline}
where the logarithm takes principal branch, the functions $w$, $r_1$ and $r_2$ are defined in \eqref{def:w} and \eqref{def r}, respectively, and the constants $\delta_j$, $j=1,\ldots,n$, are determined through the linear systems
 \begin{align}\label{def deltaj I}
 \int_{(-\infty,E_{j_0})\setminus(\cup_{j=0}^{j_0-1}(E_j,\hat{E}_j))}\frac{\log (1-r_1(s)r_2(s))s^k\dif s}{w(s)}=\sum_{j=1}^n\delta_j\int_{E_j}^{\hat{E}_j}\frac{is^k\dif s}{w_+(s)}, \quad k=0,\ldots,n-1.
 \end{align}
Here, it is understood that $(E_{j_0-1},\hat{E}_{j_0-1})=\emptyset$ if $j_0=0$. Since $\left\{\frac{s^k\dif s}{w(s)}\right\}_{k=0}^{n-1}$ forms a basis of the holomorphic differentials on the Riemann surface $\mathcal{R}$ defined by $w$, the coefficient matrix in linear system \eqref{def deltaj I} is invertible, which implies that $\delta_j$, $j=1,\ldots,n$, exist uniquely. In addition, from the symmetry relation of $r_1$ and $r_2$ in item (d) of Proposition \ref{Pro a r}, it turns out that each $\delta_j$ is real.

 \begin{proposition}\label{Pro delta1 }
 The function $\delta(z)$ defined in \eqref{def delta 1} admits the following properties.
 \begin{itemize}
 \item [$\mathrm{(a)}$]
 $\delta(z)$ is analytic in $\mathbb{C}\setminus(\cup_{j=0}^n[E_j,\hat{E}_j]\cup(-\infty,\hat{E}_{j_0}])$ and satisfies the symmetry relation $\delta(z)=-\overline{\delta(\bar{z})}$. Moreover,  we have 
 \begin{equation}\label{eq:deltaexp}
     \delta(z)=\delta(\infty)+\dfrac{\delta^{(1)}}{z}+\mathcal{O}(z^{-2}),\qquad z\to \infty,
 \end{equation}
 where
 	 \begin{align}\label{def delta inf}
 		\delta(\infty)& =\frac{1}{2\pi i}\left[\int_{(-\infty,E_{j_0})\setminus(\cup_{j=0}^{j_0-1}(E_j,\hat{E}_j))}\frac{\log (1-r_1(s)r_2(s))s^n\dif s}{w(s)}-\sum_{j=1}^n\delta_j\int_{E_j}^{\hat{E}_j}\frac{is^n\dif s}{w_+(s)} \right],\\
 		\delta^{(1)}& = \delta(\infty)\sum_{j=0}^{n}(E_j+\hat{E}_j)\nonumber\\
 		&\quad -\frac{1}{2\pi i}\left[\int_{(-\infty,E_{j_0})\setminus(\cup_{j=1}^{j_0-1}(E_j,\hat{E}_j))}\frac{\log (1-r_1(s)r_2(s))s^{n+1}\dif s}{w(s)}-\sum_{j=1}^n\delta_j\int_{E_j}^{\hat{E}_j}\frac{is^{n+1}\dif s}{w_+(s)} \right]. \label{def delta1}
 	\end{align}
 \item [$\mathrm{(b)}$ ] $\delta(z)$  satisfies the jump conditions
 \begin{align*}
 	&\delta_-(z)=\delta_+(z)+\log (1-r_1(z)r_2(z)),\qquad z\in (-\infty,E_{j_0})\setminus(\cup_{j=0}^{j_0-1}[E_j,\hat{E}_j]),\\
 	&\delta_+(z)+\delta_-(z)=i\delta_j,\hspace{3.5cm}z\in(E_j,\hat{E}_j),\ j=1,\ldots,n.
 \end{align*}
\item [$\mathrm{(c)}$] As $z\to p\in\{E_j \}_{j=0}^{j_0}\cup\{\hat{E}_j \}_{j=0}^{j_0-1}$ from
$\mathbb{C}^+$ (if $j_0=0, \ \{\hat{E}_j \}_{j=0}^{j_0-1}=\emptyset$), we have
\begin{align}\label{eq:estexpdelta1}
	e^{\delta(z)}=\mathcal{O}((z-p)^{1/2})
\end{align}
and 
\begin{align}\label{eq:estexpdelta2}
	\delta(z)=\frac{i}{2}\delta_{j_0}+\mathcal{O}((z-\hat{E}_{j_0})^{1/2}), \qquad \textrm{$z\to \hat{E}_{j_0}$ from $\mathbb{C} \setminus (-\infty, \hat{E}_{j_0})$}. 
\end{align}
\end{itemize}
\end{proposition}
\begin{proof}
The items (a) and (b) follow directly from \eqref{def deltaj I} and the definition of $\delta$ given in \eqref{def delta 1}. To show item (c), we
see from \eqref{eq:riloc} that as $z\to p\in\{E_j \}_{j=0}^{j_0}\cup\{\hat{E}_j \}_{j=0}^{j_0-1}$ from $\mathbb{C}^+$,
\begin{align}
 	\log(1-r_1(z)r_2(z))=-\log(z-p)+c_p +\mathcal{O}((z-p)^{1/2}),\label{asy logr}
 \end{align}
for some constant $c_p$, where the logarithm takes the principle branch.
We next define
    \begin{align}
   w^{(i)}(z)=\left[ -\prod_{k=0}^n(z-\hat{E}_{k})\cdot\prod_{k=0,\ldots,n,k\neq i}(z-E_k)\right] ^{\frac{1}{2}},\label{def w}\quad i=0,\ldots,n,
    \end{align}
where the principal branch is used for the square root and the cut is choose such that $w^{(j)}(z)$ is continuous on $(\hat{E}_{j-1},\hat{E}_{j})$.
As $z\to E_j$ from $\mathbb{C}^+$, $j=0,\ldots,j_0$, it is then readily seen from \eqref{asy logr} and \eqref{def delta 1} that 
\begin{align}\label{eq:estdelta1}
    	\delta(z) & = -\frac{w(z)}{2\pi i}\int_{E_j-\epsilon}^{E_j}\frac{\log (1-r_1(s)r_2(s))\dif s}{w^{(j)}(s)\sqrt{E_j-s}(s-z)}+\mathcal{O}(1) \nonumber 
        \\
    	& = \frac{w(z)}{2\pi i}\int_{-\infty}^{E_j}\frac{\log |E_j-s|\dif s}{w^{(j)}(E_j)\sqrt{|E_j-s|}(s-z)}+\mathcal{O}(1)
        \nonumber
        \\
    & = w(z)\frac{\log (z-E_j)}{2w^{(j)}(E_j)\sqrt{z-E_j}}+\mathcal{O}(1) = \frac{1}{2}\log(z-E_j)+\mathcal{O}(1),
\end{align}
where $\epsilon>0$ is a small fixed constant. Similarly, by setting  
    \begin{align}
    \hat{w}^{(i)}(z):=\left[ \prod_{k=0}^n(z-E_k)\cdot\prod_{k=0,\ldots,n,k\neq i}(z-\hat{E}_{k})\right] ^{\frac{1}{2}},\quad i=0,\ldots,n, 
        \label{def hatw}
    \end{align} 
we have, as  $z\to \hat{E}_j$ from $\mathbb{C}^+$, $j=0,\ldots,j_0-1$,
    \begin{align}\label{eq:estdelta2}
    	\delta(z)
    	&=-\frac{w(z)}{2\pi i}\int_{\hat{E}_j}^{\hat{E}_j+\epsilon}\frac{\log (1-r_1(s)r_2(s))\dif s}{\hat{w}^{(j)}(s)\sqrt{s-\hat{E}_j}(s-z)}+\mathcal{O}(1) \nonumber \\
    	&=\frac{w(z)}{2\pi i}\int_{\hat{E}_j}^{\infty}\frac{\log |s-\hat{E}_j|\dif s}{i\hat{w}^{(j)}(\hat{E}_j)\sqrt{|\hat{E}_j-s|}(s-z)}+\mathcal{O}(1) \nonumber \\
    	&= w(z)\frac{\log (z-\hat{E}_j)}{2\hat{w}^{(j)}(\hat{E}_j)\sqrt{z-\hat{E}_j}}+\mathcal{O}(1)
    	=\frac{1}{2}\log(z-\hat{E}_j)+\mathcal{O}(1).
    \end{align}
A combination of \eqref{eq:estdelta1} and \eqref{eq:estdelta2} gives us \eqref{eq:estexpdelta1}. 

Finally, as $z\to\hat{E}_{j_0}$ from $\mathbb{C} \setminus (-\infty, \hat{E}_{j_0})$, one has 
    \begin{align*}
    		\delta(z)&=\frac{w(z)}{2\pi i}\delta_{j_0}\int_{E_{j_0}}^{\hat{E}_{j_0}}\frac{i\dif s}{\hat{w}^{(j_0)}(s)\sqrt{s-\hat{E}_{j_0}}(s-z)}+\mathcal{O}((z-\hat{E}_{j_0})^{1/2})\\
    		&=\frac{w(z)}{2\pi i\hat{w}^{(j_0)}(\hat{E}_j)}\delta_{j_0}\int_{E_{j_0}}^{\hat{E}_{j_0}}\frac{\dif s}{\sqrt{|\hat{E}_{j_0}-s|}(s-z)}+\mathcal{O}((z-\hat{E}_{j_0})^{1/2})\\
    		&=\frac{w(z)}{2\pi i\hat{w}^{(j_0)}(\hat{E}_j)\sqrt{z-\hat{E}_{j_0}}}\delta_{j_0}\left(2\arctan\frac{\sqrt{z-\hat{E}_{j_0}}}{\sqrt{\hat{E}_{j_0}-E_{j_0}}}-\pi\right)+\mathcal{O}((z-\hat{E}_{j_0})^{1/2})\\
    		&=\frac{i}{2}\delta_{j_0}+\mathcal{O}((z-\hat{E}_{j_0})^{1/2}),
    \end{align*} 
as required. This completes the proof.
 \end{proof}

Let $z_{j_0}\in(E_{j_0},\hat{E}_{j_0})$ be a fixed point, we next set
\begin{align}
&\Omega_1:=\Omega_1(\xi)=\{z\in\mathbb{C}:\ 0\leq\arg(z-z_0)\leq\varphi_0\},\ \Sigma_1:=\Sigma_1(\xi)=z_0+e^{i\varphi_0}\mathbb{R}^+,\label{def Omega sigma1}	\\
	&\Omega_2:=\Omega_2(\xi)=\{z\in\mathbb{C}:\ \pi-\varphi_0\leq\arg(z-z_{j_0})\leq\pi\},\ \Sigma_2:=\Sigma_2(\xi)=z_{j_0}+e^{i(\pi-\varphi_0)}\mathbb{R}^+,\label{def Omega sigma2}
\end{align}
where $\varphi_0$ is chosen such that 
\begin{equation}\label{eq:inequality1}
    \Im \theta(z) \left\{
   \begin{array}{ll}
     <0, & \hbox{$z\in\Sigma_1\setminus\{z_0\}$,} \\
     >0, & \hbox{$z\in\Sigma_2\setminus\{z_{j_0}\}$,}
   \end{array}
 \right.
\end{equation}
see Figure \ref{fig I 1} for an illustration. 
The first transformation of RH problem \ref{RHP1} for $M$ is defined by
\begin{align}\label{trans 1 I}
 	M^{(1)}(z):=M^{(1)}(z;\xi,t)=e^{\delta(\infty)\sigma_3}M(z)G(z)e^{-\delta(z)\sigma_3},
 \end{align}
where $\delta$ is given in \eqref{def delta 1} and 
\begin{equation}\label{def:G}
 G(z): = G(z;\xi) =\left\{
   \begin{array}{ll}
     \begin{pmatrix}
 			1 & 0\\
 		r_2(z)e^{-2it\theta(z)} & 1
 	\end{pmatrix}, & \hbox{$z\in\Omega_1$,} \\[8pt]
     \begin{pmatrix}
 		1 & r_1(z)e^{2it\theta(z)} \\
 		0 & 1
 	\end{pmatrix}, & \hbox{$z\in\Omega_1^*$,}
    \\[8pt]
   \begin{pmatrix}
 	1 &	-\frac{ r_1(z)e^{2it\theta(z)}}{1-r_1(z)r_2(z)}\\
 	0 & 1
 	\end{pmatrix}, & \hbox{$z\in\Omega_2$,} \\[12pt]
     \begin{pmatrix}
 		1 &	0\\
 	-\frac{r_2(z)e^{-2it\theta(z)}}{1-r_1(z)r_2(z)} & 1
 	\end{pmatrix}, & \hbox{$z\in\Omega_2^*$,}
    \\
     I, & \hbox{elsewhere,}
   \end{array}
 \right.
\end{equation}
is a matrix-valued function. In view of the factorizations of the jump matrix $J$ given in \eqref{open jump J1}--\eqref{open jump J4},  it is then readily seen that $M^{(1)}$ satisfies the following RH problem.

\begin{RHP}
 \label{RHPM1}
 \hfil
	\begin{itemize}
		\item [$\mathrm{(a)}$] $M^{(1)}(z)$ is analytic in $\mathbb{C}\setminus\Sigma^{(1)}$, where  \begin{equation}\label{def:Sigma1}
		    \Sigma^{(1)}:=\Sigma^{(1)}(\xi)= (-\infty,z_0]\cup (\cup_{j=0}^n[E_j,\hat{E}_j])\cup\Sigma_1\cup\Sigma_1^*\cup\Sigma_2\cup\Sigma_2^*.
		\end{equation}
		\item [$\mathrm{(b)}$]  For $z\in\Sigma^{(1)}$, $M^{(1)}$  satisfies the jump condition $$M^{(1)}_{+}(z)=M^{(1)}_{-}(z)J^{(1)}(z),$$ where
		\begin{align}\label{def:J1}
				J^{(1)}(z)=	\left\{ \begin{array}{ll}
				e^{i(f_0x+g_0t-(B_j^fx+B_j^gt+\phi_j-\delta_j)/2)\hat{\sigma}_3}\begin{pmatrix}
				0 & -i\\
				-i& 0
			\end{pmatrix}, & z\in (E_j,\hat{E}_j),\\[4pt]
					\begin{pmatrix}
					1 & 0\\
					-r_2(z)e^{-2it\theta(z)-2\delta(z)} & 1	
				\end{pmatrix},&  z\in\Sigma_1,\\[9pt]
			\begin{pmatrix}
					1 & r_1(z)e^{2it\theta(z)+2\delta(z)} \\
					0 & 1
				\end{pmatrix}, &z\in\Sigma_1^*,\\[9pt]
				\begin{pmatrix}
			1 &	\frac{ r_1(z)e^{2it\theta(z)+2\delta(z)}}{1-r_1(z)r_2(z)}\\
			0 & 1
			\end{pmatrix},&z\in\Sigma_2,\\[12pt]
			\begin{pmatrix}
				1 &	0\\
			\frac{-r_2(z)e^{-2it\theta(z)-2\delta(z)}}{1-r_1(z)r_2(z)} & 1
			\end{pmatrix},&z\in\Sigma_2^*,\\[12pt]
		\begin{pmatrix}
			1-r_1(z)r_2(z) & r_1(z)e^{2it\theta(z)+2\delta(z)}\\
			-r_2(z)e^{-2it\theta(z)-2\delta(z)} & 1
		\end{pmatrix}, & z\in (\hat{E}_{j_0},z_0),
				\end{array}\right.
		\end{align}
         with $\phi_0=\delta_0=0$ and $j=0,\ldots,n$.
		\item [$\mathrm{(c)}$] As $z\to\infty$, we have $M^{(1)}(z)=I+\mathcal{O}(z^{-1})$.
		\item [$\mathrm{(d)}$]  Let $p\in\{E_j,\hat{E}_j\}_{j=0}^n\setminus\{\hat{E}_{j_0}\}$, we have
		\begin{align*}
			M^{(1)}(z)=\mathcal{O}((z-p)^{-1/4}),\qquad z\to p.
		\end{align*}
	\end{itemize}
\end{RHP}

 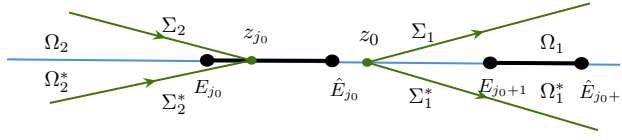
\begin{figure}

 \tikzset{every picture/.style={line width=0.75pt}} 

 \begin{tikzpicture}[x=0.75pt,y=0.75pt,yscale=-0.5,xscale=0.6]
 	\draw [color={rgb, 255:red, 74; green, 144; blue, 226 }  ,draw opacity=1 ]   (20.47,210.4) -- (530.47,215.8) ;

 	\draw [line width=1.5]    (186.9,211.4) -- (290.9,211.9) ;
 	\draw [shift={(290.9,211.9)}, rotate = 0.28] [color={rgb, 255:red, 0; green, 0; blue, 0 }  ][fill={rgb, 255:red, 0; green, 0; blue, 0 }  ][line width=1.5]      (0, 0) circle [x radius= 4.36, y radius= 4.36]   ;
 	\draw [shift={(186.9,211.4)}, rotate = 0.28] [color={rgb, 255:red, 0; green, 0; blue, 0 }  ][fill={rgb, 255:red, 0; green, 0; blue, 0 }  ][line width=1.5]      (0, 0) circle [x radius= 4.36, y radius= 4.36]   ;
 	\draw [line width=1.5]    (422.4,213.9) -- (498.4,214.39) ;
 	\draw [shift={(498.4,214.39)}, rotate = 0.37] [color={rgb, 255:red, 0; green, 0; blue, 0 }  ][fill={rgb, 255:red, 0; green, 0; blue, 0 }  ][line width=1.5]      (0, 0) circle [x radius= 4.36, y radius= 4.36]   ;
 	\draw [shift={(422.4,213.9)}, rotate = 0.37] [color={rgb, 255:red, 0; green, 0; blue, 0 }  ][fill={rgb, 255:red, 0; green, 0; blue, 0 }  ][line width=1.5]      (0, 0) circle [x radius= 4.36, y radius= 4.36]   ;
 	\draw [color={rgb, 255:red, 65; green, 117; blue, 5 }  ,draw opacity=1 ]   (72,163.27) -- (223.9,211.4) ;
 	\draw [shift={(223.9,211.4)}, rotate = 17.58] [color={rgb, 255:red, 65; green, 117; blue, 5 }  ,draw opacity=1 ][fill={rgb, 255:red, 65; green, 117; blue, 5 }  ,fill opacity=1 ][line width=0.75]      (0, 0) circle [x radius= 3.35, y radius= 3.35]   ;
 	\draw [shift={(152.72,188.85)}, rotate = 197.58] [fill={rgb, 255:red, 65; green, 117; blue, 5 }  ,fill opacity=1 ][line width=0.08]  [draw opacity=0] (10.72,-5.15) -- (0,0) -- (10.72,5.15) -- (7.12,0) -- cycle    ;
 	\draw [color={rgb, 255:red, 65; green, 117; blue, 5 }  ,draw opacity=1 ]   (320,213.5) -- (511.97,281.57) ;
 	\draw [shift={(420.7,249.2)}, rotate = 199.52] [fill={rgb, 255:red, 65; green, 117; blue, 5 }  ,fill opacity=1 ][line width=0.08]  [draw opacity=0] (10.72,-5.15) -- (0,0) -- (10.72,5.15) -- (7.12,0) -- cycle    ;
 		\draw [shift={(320,213.5)}, rotate = 17.58] [color={rgb, 255:red, 65; green, 117; blue, 5 }  ,draw opacity=1 ][fill={rgb, 255:red, 65; green, 117; blue, 5 }  ,fill opacity=1 ][line width=0.75]      (0, 0) circle [x radius= 3.35, y radius= 3.35]   ;
 	\draw [color={rgb, 255:red, 65; green, 117; blue, 5 }  ,draw opacity=1 ]   (320,213.5) -- (505.47,154.07) ;
 	\draw [shift={(417.49,182.26)}, rotate = 162.23] [fill={rgb, 255:red, 65; green, 117; blue, 5 }  ,fill opacity=1 ][line width=0.08]  [draw opacity=0] (10.72,-5.15) -- (0,0) -- (10.72,5.15) -- (7.12,0) -- cycle    ;
 	\draw [color={rgb, 255:red, 65; green, 117; blue, 5 }  ,draw opacity=1 ]   (56,254.27) -- (223.9,211.4) ;
 	\draw [shift={(144.79,231.6)}, rotate = 165.68] [fill={rgb, 255:red, 65; green, 117; blue, 5 }  ,fill opacity=1 ][line width=0.08]  [draw opacity=0] (10.72,-5.15) -- (0,0) -- (10.72,5.15) -- (7.12,0) -- cycle    ;
 	
 	\draw (173.5,228.5) node [anchor=north west][inner sep=0.75pt]  [font=\footnotesize,xscale=0.8,yscale=0.8] [align=left] {$\displaystyle E_{j_0}$};
 	\draw (287,226.5) node [anchor=north west][inner sep=0.75pt]  [font=\footnotesize,xscale=0.8,yscale=0.8] [align=left] {$\displaystyle \hat{E}_{j_0}$};
 	\draw (410,227.2) node [anchor=north west][inner sep=0.75pt]  [font=\footnotesize,xscale=0.8,yscale=0.8] [align=left] {$\displaystyle E_{j_0+1}$};
 	\draw (494.4,228.39) node [anchor=north west][inner sep=0.75pt]  [font=\footnotesize,xscale=0.8,yscale=0.8] [align=left] {$\displaystyle \hat{E}_{j_0+1}$};
 	\draw (311,179.9) node [anchor=north west][inner sep=0.75pt]  [xscale=0.8,yscale=0.8]  {$z_{0}$};
 	\draw (147,165.9) node [anchor=north west][inner sep=0.75pt]  [font=\small,xscale=0.8,yscale=0.8]  {$\Sigma _{2}$};
 	\draw (147,243.4) node [anchor=north west][inner sep=0.75pt]  [font=\small,xscale=0.8,yscale=0.8]  {$\Sigma _{2}^{*}$};
 	\draw (354.5,172.4) node [anchor=north west][inner sep=0.75pt]  [font=\small,xscale=0.8,yscale=0.8]  {$\Sigma _{1}$};
 	\draw (352.5,236.9) node [anchor=north west][inner sep=0.75pt]  [font=\small,xscale=0.8,yscale=0.8]  {$\Sigma _{1}^{*}$};
 	\draw (49.5,182.9) node [anchor=north west][inner sep=0.75pt]  [font=\small,xscale=0.8,yscale=0.8]  {$\Omega _{2}$};
 	\draw (49.5,220.4) node [anchor=north west][inner sep=0.75pt]  [font=\small,xscale=0.8,yscale=0.8]  {$\Omega _{2}^{*}$};
 	\draw (462,183.4) node [anchor=north west][inner sep=0.75pt]  [font=\small,xscale=0.8,yscale=0.8]  {$\Omega _{1}$};
 	\draw (462,231.9) node [anchor=north west][inner sep=0.75pt]  [font=\small,xscale=0.8,yscale=0.8]  {$\Omega _{1}^{*}$};
 	\draw (215,177.4) node [anchor=north west][inner sep=0.75pt]  [font=\small,xscale=0.8,yscale=0.8]  {$z_{j_0}$};
 \end{tikzpicture}
 	\caption{\footnotesize  The contours $\Sigma_i$ and the regions $\Omega_i$, $i=1,2$, in \eqref{def Omega sigma1} and \eqref{def Omega sigma2}.}
 	\label{fig I 1}
 \end{figure}

\begin{remark}
From Proposition \ref{Pro asy ar},  it follows that $ 
1-r_1(p)r_2(p)=0$ for $p\in\{E_j \}_{j=0}^{j_0}\cup\{\hat{E}_j \}_{j=0}^{j_0-1}$. This leads to a pole singularity for the function $G$ in \eqref{def:G}. By introducing the auxiliary function $\delta$, we could improve the singularity as stated in item (d) of the RH problem for $M^{(1)}$.
\end{remark}

We are now at the stage of approximating $M^{(1)}$ in global and local manners. To that end, let 
\begin{align}\label{def Uxi I}
	U:=   \{z:\ |z-\hat{E}_{j_0}|\leq c_0\}
\end{align}
be a small disk around $\hat{E}_{j_0}$, where 
\begin{align}
	c_0:=\min\left\{\frac{1}{2}(\hat{E}_{j_0}-z_{j_0}),\ \frac{1}{2}(E_{j_0+1}-\hat{E}_{j_0}),\ 2(z_0-\hat{E}_{j_0})t^\epsilon\right\},\qquad\epsilon\in {(1/3,2/3)}.\label{def c0}
\end{align}
Note that as $t\to \infty$, we have
\begin{align}\label{eq:estzEj0}
z_0-\hat{E}_{j_0} =\mathcal{O} (\hat{\xi}_{j_0}-\xi)=\mathcal{O}(t^{-2/3}).
\end{align}
Indeed, as discussed in Section \ref{subsec2.3}, for $\xi\in[\xi_{j_0+1},\hat{\xi}_{j_0}]$, $z_0(\xi)$ is defined implicitly by the equation $F(z_0(\xi);\xi)=0$ with $F$ given in \eqref{def F} and $z_0(\hat{\xi}_{j_0})=\hat{E}_{j_0}$. Since $\hat{E}_{j_0}$ is not a multiple root  of $F(z;\hat{\xi}_{j_0})$, one has $\partial_z F(\hat{E}_{j_0};\hat{\xi}_{j_0})\neq0$, which implies \eqref{eq:estzEj0}. Thus, $U$ is a shrinking disk and  $z_0 \in U$ for large $t$.

In what follows, we will construct global parametrix $M^{(glo)}$ and local parametrix $M^{(loc)}$ in Sections \ref{subsec mod I} and \ref{subsec loc I}, respectively, such that 
 	\begin{align}\label{construct M1}
 	M^{(1)}(z)=	\left\{ \begin{array}{ll}
 	E(z)M^{(glo)}(z), & z\in \mathbb{C}\setminus U ,\\
 	 E(z)M^{(loc)}(z),&  z\in U,\\
 	\end{array}\right.
 \end{align}
where $U$ is defined in \eqref{def Uxi I} and $E$ is the error function that satisfies a small norm RH problem as shown in Section \ref{subsec E}.

\subsection{Global parametrix}\label{subsec mod I}
By ignoring the exponentially small part of $J^{(1)}$ in \eqref{def:J1} according to \eqref{eq:inequality1},  we arrive at the following global parametrix.
 \begin{RHP}
 	\hfill\label{RHPM mod}
 	\begin{itemize}
 		\item [$\mathrm{(a)}$] $M^{(glo)}(z)$ is analytic in $\mathbb{C}\setminus(\cup_{j=0}^n[E_j,\hat{E}_j])$.
 		\item [$\mathrm{(b)}$]  For $z\in\cup_{j=0}^n(E_j,\hat{E}_j)$, $M^{(glo)}(z)$  satisfies the jump condition \begin{equation}\label{eq:Mglojump}
 		    M^{(glo)}_{+}(z)=M^{(glo)}_{-}(z)J^{(glo)}(z),
 		\end{equation}
        where for $j=0,\ldots,n$,
 		\begin{align*}
 			\small	J^{(glo)}(z;x,t)=
 				e^{i(f_0x+g_0t-(B_j^fx+B_j^gt+\phi_j-\delta_j)/2)\hat{\sigma}_3}\begin{pmatrix}
 					0 & -i\\
 					-i& 0
 				\end{pmatrix}, \qquad z\in (E_j,\hat{E}_j).
 		\end{align*}
 		\item [$\mathrm{(c)}$] As $z\to\infty$, we have $M^{(glo)}(z)=I+\mathcal{O}(z^{-1})$.
 		\item [$\mathrm{(d)}$] Let $p\in\{E_j,\hat{E}_j\}_{j=0}^n$, we have
 		\begin{align*}
 			M^{(glo)}(z)=\mathcal{O}((z-p)^{-1/4}),\qquad z\to p.
 		\end{align*}
 	\end{itemize}
 \end{RHP}
Similar to RH problem \ref{Pro mu0}, the above RH problem can be solved explicitly by using the planar matrix Baker-Akhiezer function. More precisely, we have
\begin{align}
	\label{def Mmod}
    M^{(glo)}(z)=e^{-i(f_0x+g_0t)\sigma_3}N^{(glo)}(\infty)^{-1}N^{(glo)}(z)e^{i(f_0x+g_0t)\sigma_3},
\end{align}
where $N^{(glo)}(z)=\left(N^{(glo)}_{ij}(z)\right)_{i,j=1,2}$ with 
	\begin{align}
	&	N^{(glo)}_{11}(z)=\frac{1}{2}(\varkappa(z)+\varkappa(z)^{-1})\dfrac{\Theta(\mathcal{A}(z)+\boldsymbol{C}(x,t;\boldsymbol{\phi-\delta})+\mathcal{A}(\mathcal{D})+\boldsymbol{K})}{\Theta(\mathcal{A}(z)+\mathcal{A}(\mathcal{D})+\boldsymbol{K})},\label{N}\\ 
    & N^{(glo)}_{12}(z)=\frac{1}{2}(\varkappa(z)-\varkappa(z)^{-1})\dfrac{\Theta(-\mathcal{A}(z)+\boldsymbol{C}(x,t;\boldsymbol{\phi-\delta})+\mathcal{A}(\mathcal{D})+\boldsymbol{K})}{\Theta(-\mathcal{A}(z)+\mathcal{A}(\mathcal{D})+\boldsymbol{K})},\\
	&	N^{(glo)}_{21}(z)=\frac{1}{2}(\varkappa(z)-\varkappa(z)^{-1})\dfrac{\Theta(\mathcal{A}(z)+\boldsymbol{C}(x,t;\boldsymbol{\phi-\delta})-\mathcal{A}(\mathcal{D})-\boldsymbol{K})}{\Theta(\mathcal{A}(z)-\mathcal{A}(\mathcal{D})-\boldsymbol{K})} ,\\
	&N^{(glo)}_{22}(z)=\frac{1}{2}(\varkappa(z)+\varkappa(z)^{-1})\dfrac{\Theta(\mathcal{A}(z)-\boldsymbol{C}(x,t;\boldsymbol{\phi-\delta})+\mathcal{A}(\mathcal{D})+\boldsymbol{K})}{\Theta(\mathcal{A}(z)+\mathcal{A}(\mathcal{D})+\boldsymbol{K})}.\label{eq:Nglob22}
	\end{align}
Here, $\varkappa$ is defined in \eqref{def:varkappa}, $\Theta$ is the Riemann theta function given in \eqref{def Theta} with the same  $\boldsymbol{C}$, $\boldsymbol{K}$ and $\mathcal{D}$ as in \eqref{def q AG},  
$\boldsymbol{\phi}=(\phi_1,\cdots,\phi_n)$ and 
$\boldsymbol{\delta}=(\delta_1,\cdots,\delta_n)$. 
In addition, we have 
 \begin{align}
 \label{asy Mmod}	M^{(glo)}(z)=I+\frac{M^{(glo)}_1}{z}+\mathcal{O}(z^{-2}), \qquad z\to\infty, 
 \end{align}
 and
 \begin{align}
 	\left(M^{(glo)}_{1}\right)_{12}=-\frac{i}{2}q^{(AG)}(x,t;\E, \hat{\E}, \boldsymbol{\phi-\delta}),\label{asy Mmod12}
 \end{align}
where $q^{(AG)}$ is defined in \eqref{def q AG}.

\subsection{Local parametrix}\label{subsec loc I}
The local parametrix $M^{(loc)}(z)$ reads as follows.
\begin{RHP}
	\hfil\label{Mloc1}
	\begin{itemize}
		\item [$\mathrm{(a)}$] $M^{(loc)}(z)$ is analytic in $U\setminus\Sigma^{(1)}$, where $U$ and $\Sigma^{(1)}$ are defined in \eqref{def Uxi I} and \eqref{def:Sigma1}, respectively. 
		\item [$\mathrm{(b)}$]  For $z\in\Sigma^{(1)}\cap U$, $M^{(loc)}(z)$  satisfies the jump condition $$M^{(loc)}_{+}(z)=M^{(loc)}_{-}(z)J^{(1)}(z),$$
        where $J^{(1)}$ is defined in \eqref{def:J1}. 
		\item [$\mathrm{(c)}$] $M^{(loc)}(z)$ has at most $-1/4$-singularity at $\hat{E}_{j_0}$.
         \item [$\mathrm{(d)}$] As $t\to\infty$, $M^{(loc)}(z)$ matches $M^{(glo)}(z)$ on the boundary $\partial U$ of $U$.
	\end{itemize}
\end{RHP}

To solve the above RH problem, we need to understand the local behavior of the function $\theta$ defined in \eqref{theta} near $z=\hat{E}_{j_0}$, which is given in the following proposition.
	\begin{proposition}
	For any fixed $\xi$, we have, as  $z\to\hat{E}_{j_0}$, 
\begin{align}\label{eq:thetaexp}
\theta(z)=\theta^{(0,\hat{j_0})}+(\xi-\hat{\xi}_{j_0})\theta^{(1,\hat{j_0})}(z-\hat{E}_{j_0})^{1/2}+\frac{2}{3}
\theta^{(3,\hat{j_0})}(z-\hat{E}_{j_0})^{3/2}+\mathcal{O}\left((z-\hat{E}_{j_0})^{5/2}\right),
\end{align}
where  
\begin{align}
&\theta^{(0,\hat{j_0})}=\theta(\hat{E}_{j_0};\xi)=f_0\xi+g_0-\frac{1}{2}(B_{j_0}^f\xi+B_{j_0}^g),\\
&\theta^{(1,\hat{j_0})}=-\frac{2\prod_{j=0}^n(\hat{E}_{j_0}-z_j^f)}{\hat{w}^{(j_0)}(\hat{E}_{j_0})},\label{theta hj01}
    \end{align}
    and
\begin{align*}
\theta^{(3,\hat{j_0})}:=\theta^{(3,\hat{j_0})}(\xi)=\frac{-1}{\hat{w}^{(j_0)}(\hat{E}_{j_0})^2}\left[ \partial_z F(\hat{E}_{j_0};\xi)\hat{w}^{(j_0)}(\hat{E}_{j_0})
-F(\hat{E}_{j_0};\xi)(\hat{w}^{(j_0)})'(\hat{E}_{j_0})\right].
\end{align*}
Here, $f_0$, $g_0$ and $B_{j_0}^f$, $B_{j_0}^g$ are given in \eqref{def f0g0} and \eqref{def:Bjfg}, the functions $\hat{w}^{(j_0)}$ and $F$ are defined in \eqref{def hatw} and \eqref{def F}, respectively. Moreover,  $\theta^{(3,\hat{j_0})}(\xi) < 0$ in a small neighborhood of $\hat{\xi}_{j_0}$.
\end{proposition}
\begin{proof}
Recall that $\theta(z)=\theta(z;\xi):=-(f(z)-f_0)\xi-\left(g(z)-g_0 \right)$.  By \eqref{def fg}, it follows that 
		\begin{align*}
			\theta(z)& =f_0\xi+g_0-\int_{\hat{E}_{0}}^z\frac{\xi\prod_{k=0}^n(s-z_k^f)+4\prod_{k=0}^{n+1}(s-z_k^g)}{w(s)}\dif s
            \\
			& =\theta(\hat{E}_{j_0};\xi)-\int_{\hat{E}_{j_0}}^z\frac{F(s;\xi)}{\hat{w}^{(j_0)}(s)\sqrt{s-\hat{E}_{j_0}}}\dif s,
		\end{align*}
where $\hat{w}^{(j_0)}$ and $F$ are given in \eqref{def hatw} and \eqref{def F}, respectively.
 We then obtain from integration by parts that
		\begin{align*}
			\theta(z)&=\theta(\hat{E}_{j_0};\xi)-2(z-\hat{E}_{j_0})^{1/2}\frac{F(z;\xi)}{\hat{w}^{(j_0)}(z)}+ \int_{\hat{E}_{j_0}}^z2(s-\hat{E}_{j_0})^{1/2}\partial_s\left(\frac{F(s;\xi)}{\hat{w}^{(j_0)}(s)}\right) \dif s\\
			&=\theta(\hat{E}_{j_0};\xi)-2(z-\hat{E}_{j_0})^{1/2}\frac{F(\hat{E}_{j_0};\xi)}{\hat{w}^{(j_0)}(\hat{E}_{j_0})}-2(z-\hat{E}_{j_0})^{1/2}\left(\frac{F(z;\xi)}{\hat{w}^{(j_0)}(z)}-\frac{F(\hat{E}_{j_0};\xi)}{\hat{w}^{(j_0)}(\hat{E}_{j_0})} \right) \\
			&~~~~+\frac{4}{3}(z-\hat{E}_{j_0})^{3/2}\partial_z\left(\frac{F(z;\xi)}{\hat{w}^{(j_0)}(z)}\right) -\int_{\hat{E}_{j_0}}^z\frac{4}{3}(s-\hat{E}_{j_0})^{3/2}\partial_s^2\left(\frac{F(s;\xi)}{\hat{w}^{(j_0)}(s)}\right) \dif s.
		\end{align*}
By expanding the functions $\frac{F(z;\xi)}{\hat{w}^{(j_0)}(z)}$ and $\partial_z\left(\frac{F(z;\xi)}{\hat{w}^{(j_0)}(z)}\right)$ at $z=\hat{E}_{j_0}$, it is readily seen that, for any fixed $\xi$, 
		\begin{align}
			\theta(z)& = \theta(\hat{E}_{j_0};\xi)-2(z-\hat{E}_{j_0})^{1/2}\frac{F(\hat{E}_{j_0};\xi)}{\hat{w}^{(j_0)}(\hat{E}_{j_0})}\nonumber\\
			&~~-\frac{2}{3}(z-\hat{E}_{j_0})^{3/2}\left(\frac{\hat{w}^{(j_0)}(\hat{E}_{j_0})\partial_zF(\hat{E}_{j_0};\xi)-F(\hat{E}_{j_0};\xi)(\hat{w}^{(j_0)})'(\hat{E}_{j_0})}{\hat{w}^{(j_0)}(\hat{E}_{j_0})^2}\right)
            \nonumber 
            \\
            &~~+\mathcal{O}\left((z-\hat{E}_{j_0})^{5/2}\right), \qquad z\to\hat{E}_{j_0}. \label{equ 4.35}
		\end{align} 
		Substituting \eqref{def xij} into \eqref{def F}, we have  $F(\hat{E}_{j_0};\hat{\xi}_{j_0})=0$. It thereby follows that
		  \begin{align*}
		  	F(\hat{E}_{j_0};\xi)=F(\hat{E}_{j_0};\xi)-F(\hat{E}_{j_0};\hat{\xi}_{j_0})=(\xi-\hat{\xi}_{j_0})\prod_{j=0}^n(\hat{E}_{j_0}-z^f_j).
		  \end{align*}
	 Combining the above equation and \eqref{equ 4.35}, we finally obtain \eqref{eq:thetaexp}. In addition, the fact  $F(\hat{E}_{j_0};\hat{\xi}_{j_0})=0$ also implies that
		\begin{align}\label{theta hj03}
			\theta^{(3,\hat{j_0})}(\hat{\xi}_{j_0})=-\frac{\partial_z F(\hat{E}_{j_0};\hat{\xi}_{j_0})}{\hat{w}^{(j_0)}(\hat{E}_{j_0})},
		\end{align}
		and we actually have $\theta^{(3,\hat{j_0})}(\hat{\xi}_{j_0})<0$. By continuity, $\theta^{(3,\hat{j_0})}(\xi) < 0$ in a small neighborhood of $\hat{\xi}_{j_0}$. This completes the proof.
	\end{proof}
	

The expansion \eqref{eq:thetaexp} invokes us to define
\begin{align}\label{def zeta1}
    \zeta(z):=\zeta(z;\xi)= 
    \left(\frac{3it}{2}(\theta(\hat{E}_{j_0})+(\xi-\hat{\xi}_{j_0})\theta^{(1,\hat{j_0})}(z-\hat{E}_{j_0})^{1/2}-
    \theta(z))\right)^{2/3}, \quad  z\in U. 
\end{align}
which is a one-to-one conformal mapping in $U$ with respect to $z$. Moreover, it is easily seen that 
\begin{align}
    \zeta(\hat{E}_{j_0})=0,\qquad \zeta'(\hat{E}_{j_0})=-|\theta^{(3,\hat{j_0})}(\xi)|^{2/3}t^{2/3}<0.\label{d zeta1}
\end{align}
By taking the principal branch for the fractions, we have 
\begin{align*}
 \zeta(z)^{3/2}=\left\{\begin{array}{ll}
-it\theta^{(3,\hat{j_0})}(\xi)(z-\hat{E}_{j_0})^{3/2}(1+\mathcal{O}(z-\hat{E}_{j_0})) ,&z\to \hat{E}_{j_0}~\textrm{from}~\mathbb{C}^+,\\
 it\theta^{(3,\hat{j_0})}(\xi)(z-\hat{E}_{j_0})^{3/2}(1+\mathcal{O}(z-\hat{E}_{j_0})) ,&z\to \hat{E}_{j_0}~\textrm{from}~\mathbb{C}^-.
\end{array}\right.
\end{align*}
and 
\begin{align}
 \zeta(z)^{1/2}=\left\{\begin{array}{ll}
-it^{1/3}|\theta^{(3,\hat{j_0})}(\xi)|^{1/3}(z-\hat{E}_{j_0})^{1/2}(1+\mathcal{O}(z-\hat{E}_{j_0})) ,&z\to \hat{E}_{j_0}~\textrm{from}~ \mathbb{C}^+,\\
 it^{1/3}|\theta^{(3,\hat{j_0})}(\xi)|^{1/3}(z-\hat{E}_{j_0})^{1/2}(1+\mathcal{O}(z-\hat{E}_{j_0})) ,&z\to \hat{E}_{j_0}~\textrm{from}~\mathbb{C}^-.
\end{array}\right.\label{equzeta1/2}
\end{align}

Let
\begin{align}\label{def:S}
S(z;\xi):=\left\{\begin{array}{ll}
 -it(\xi-\hat{\xi}_{j_0})\theta^{(1,\hat{j_0})}\frac{(z-\hat{E}_{j_0})^{1/2}}{\zeta(z)^{1/2}},&z\in \mathbb{C}^+\cap U,\\
  it(\xi-\hat{\xi}_{j_0})\theta^{(1,\hat{j_0})}\frac{(z-\hat{E}_{j_0})^{1/2}}{\zeta(z)^{1/2}},&z\in \mathbb{C}^-\cap U.
\end{array}\right.
\end{align}
Although  it appears to have different definitions in $\mathbb{C}^\pm\cap U$, by \eqref{equzeta1/2}, $S(z;\xi)$  is also analytic in $U$ with respect to $z$. Note that for large positive $t$, $\partial_zS(\hat{E}_{j_0};\xi)\neq0$ and
\begin{align}	\label{def s1}
	 s:=S(\hat{E}_{j_0};\xi)=-\frac{\theta^{(1,\hat{j_0})}}{|\theta^{(3,\hat{j_0})}(\xi)|^{1/3}}(\xi-\hat{\xi}_{j_0})t^{2/3} \in \mathbb{R}.
\end{align}
From \eqref{def zeta1} and \eqref{def:S}, it is readily seen that 
\begin{align}\label{asy theta+}
	it\theta(z)=\left\{ \begin{array}{ll}
 	it\theta(\hat{E}_{j_0})-S(z;\xi)\zeta(z)^{1/2}-\frac{2}{3}\zeta(z)^{3/2}, & z\in \mathbb{C}^+\cap U, 
     \\
it\theta(\hat{E}_{j_0})+S(z;\xi)\zeta(z)^{1/2}+\frac{2}{3}\zeta(z)^{3/2},&  z\in \mathbb{C}^-\cap U.
 	\end{array}\right.
\end{align}

The expansion of $\theta$ above inspires us to work in the $\zeta$-plane. We thus set 
\begin{align}
  U^{(\zeta)}:= \zeta(U),  \label{def Uzeta}
\end{align}
a neighborhood of the origin in the $\zeta$-plane and define the contours 
\begin{equation}\label{def Sigmalocal}
    \begin{aligned}
    \Sigma_1^{(loc)}&:=
 	(\zeta(\Sigma_1^*\cap U))\cup\{
 	\zeta(\Sigma_1^*\cap \partial U)+e^{(\pi-\varphi)i}\mathbb{R}^+\},
    \\
    \Sigma^{(loc,\zeta)}&:=(\zeta(z_0),+\infty)\cup\Sigma_1^{(loc)}\cup\Sigma_1^{(loc)*},
\end{aligned}
\end{equation}
see the right picture of Figure \ref{Fig loc jump}. Clearly, 
$\zeta(z)$ gives a bijection between $U^{(\zeta)}$ in the $\zeta$-plane and $U$ in the $z$-plane. 
\begin{figure}	
\tikzset{every picture/.style={line width=0.75pt}} 
	\begin{tikzpicture}[x=0.75pt,y=0.75pt,yscale=-0.7,xscale=0.7]

\draw [color={rgb, 255:red, 74; green, 144; blue, 226 }  ,draw opacity=1 ]   (125.9,223.9) -- (162,224.5) ;
\draw [shift={(148.95,224.28)}, rotate = 180.95] [fill={rgb, 255:red, 74; green, 144; blue, 226 }  ,fill opacity=1 ][line width=0.08]  [draw opacity=0] (10.72,-5.15) -- (0,0) -- (10.72,5.15) -- (7.12,0) -- cycle    ;
\draw [line width=1.5]    (17.4,223.9) -- (125.9,223.9) ;
\draw [shift={(125.9,223.9)}, rotate = 0] [color={rgb, 255:red, 0; green, 0; blue, 0 }  ][fill={rgb, 255:red, 0; green, 0; blue, 0 }  ][line width=1.5]      (0, 0) circle [x radius= 4.36, y radius= 4.36]   ;
\draw [shift={(78.45,223.9)}, rotate = 180] [fill={rgb, 255:red, 0; green, 0; blue, 0 }  ][line width=0.08]  [draw opacity=0] (13.4,-6.43) -- (0,0) -- (13.4,6.44) -- (8.9,0) -- cycle    ;
\draw [color={rgb, 255:red, 65; green, 117; blue, 5 }  ,draw opacity=1 ]   (162,224.5) -- (228,259.21) ;
\draw [shift={(199.43,244.18)}, rotate = 207.74] [fill={rgb, 255:red, 65; green, 117; blue, 5 }  ,fill opacity=1 ][line width=0.08]  [draw opacity=0] (10.72,-5.15) -- (0,0) -- (10.72,5.15) -- (7.12,0) -- cycle    ;
\draw [shift={(162,224.5)}, rotate = 27.74] [color={rgb, 255:red, 65; green, 117; blue, 5 }  ,draw opacity=1 ][fill={rgb, 255:red, 65; green, 117; blue, 5 }  ,fill opacity=1 ][line width=0.75]      (0, 0) circle [x radius= 3.35, y radius= 3.35]   ;
\draw [color={rgb, 255:red, 65; green, 117; blue, 5 }  ,draw opacity=1 ]   (162,224.5) -- (230,194.21) ;
\draw [shift={(200.57,207.32)}, rotate = 155.99] [fill={rgb, 255:red, 65; green, 117; blue, 5 }  ,fill opacity=1 ][line width=0.08]  [draw opacity=0] (10.72,-5.15) -- (0,0) -- (10.72,5.15) -- (7.12,0) -- cycle    ;
\draw  [dash pattern={on 0.84pt off 2.51pt}] (17.4,223.9) .. controls (17.4,163.98) and (65.98,115.4) .. (125.9,115.4) .. controls (185.82,115.4) and (234.4,163.98) .. (234.4,223.9) .. controls (234.4,283.82) and (185.82,332.4) .. (125.9,332.4) .. controls (65.98,332.4) and (17.4,283.82) .. (17.4,223.9) -- cycle ;
\draw [line width=1.5]    (334.1,223.42) -- (259,222.26) ;
\draw [shift={(256,222.21)}, rotate = 0.89] [color={rgb, 255:red, 0; green, 0; blue, 0 }  ][line width=1.5]    (14.21,-4.28) .. controls (9.04,-1.82) and (4.3,-0.39) .. (0,0) .. controls (4.3,0.39) and (9.04,1.82) .. (14.21,4.28)   ;
\draw [shift={(337.1,223.47)}, rotate = 180.89] [color={rgb, 255:red, 0; green, 0; blue, 0 }  ][line width=1.5]    (14.21,-4.28) .. controls (9.04,-1.82) and (4.3,-0.39) .. (0,0) .. controls (4.3,0.39) and (9.04,1.82) .. (14.21,4.28)   ;
\draw [color={rgb, 255:red, 74; green, 144; blue, 226 }  ,draw opacity=1 ]   (470,227.21) -- (539.9,226.9) ;
\draw [shift={(509.95,227.03)}, rotate = 179.75] [fill={rgb, 255:red, 74; green, 144; blue, 226 }  ,fill opacity=1 ][line width=0.08]  [draw opacity=0] (10.72,-5.15) -- (0,0) -- (10.72,5.15) -- (7.12,0) -- cycle    ;
\draw [line width=1.5]    (641,227.21) -- (539.9,226.9) ;
\draw [shift={(539.9,226.9)}, rotate = 180.17] [color={rgb, 255:red, 0; green, 0; blue, 0 }  ][fill={rgb, 255:red, 0; green, 0; blue, 0 }  ][line width=1.5]      (0, 0) circle [x radius= 4.36, y radius= 4.36]   ;
\draw [shift={(598.75,227.08)}, rotate = 180.17] [fill={rgb, 255:red, 0; green, 0; blue, 0 }  ][line width=0.08]  [draw opacity=0] (13.4,-6.43) -- (0,0) -- (13.4,6.44) -- (8.9,0) -- cycle    ;
\draw [color={rgb, 255:red, 65; green, 117; blue, 5 }  ,draw opacity=1 ]   (364.1,139.38) -- (470,227.21) ;
\draw [shift={(420.9,186.49)}, rotate = 219.67] [fill={rgb, 255:red, 65; green, 117; blue, 5 }  ,fill opacity=1 ][line width=0.08]  [draw opacity=0] (10.72,-5.15) -- (0,0) -- (10.72,5.15) -- (7.12,0) -- cycle    ;
\draw [color={rgb, 255:red, 65; green, 117; blue, 5 }  ,draw opacity=1 ]   (356.1,304.38) -- (470,227.21) ;
\draw [shift={(470,227.21)}, rotate = 325.88] [color={rgb, 255:red, 65; green, 117; blue, 5 }  ,draw opacity=1 ][fill={rgb, 255:red, 65; green, 117; blue, 5 }  ,fill opacity=1 ][line width=0.75]      (0, 0) circle [x radius= 3.35, y radius= 3.35]   ;
\draw [shift={(417.19,262.99)}, rotate = 145.88] [fill={rgb, 255:red, 65; green, 117; blue, 5 }  ,fill opacity=1 ][line width=0.08]  [draw opacity=0] (10.72,-5.15) -- (0,0) -- (10.72,5.15) -- (7.12,0) -- cycle    ;
\draw  [dash pattern={on 0.84pt off 2.51pt}] (361.5,227.21) .. controls (361.5,167.29) and (410.08,118.71) .. (470,118.71) .. controls (529.92,118.71) and (578.5,167.29) .. (578.5,227.21) .. controls (578.5,287.13) and (529.92,335.71) .. (470,335.71) .. controls (410.08,335.71) and (361.5,287.13) .. (361.5,227.21) -- cycle ;

\draw (120,235.5) node [anchor=north west][inner sep=0.75pt]  [font=\footnotesize,xscale=0.8,yscale=0.8] [align=left] {$\displaystyle \hat{E}_{j_0}$};
\draw (144,201.9) node [anchor=north west][inner sep=0.75pt]  [xscale=0.8,yscale=0.8]  {$z_{0}$};
\draw (181.5,184.4) node [anchor=north west][inner sep=0.75pt]  [font=\small,xscale=0.8,yscale=0.8]  {$\Sigma _{1}$};
\draw (187.5,248.9) node [anchor=north west][inner sep=0.75pt]  [font=\small,xscale=0.8,yscale=0.8]  {$\Sigma _{1}^{*}$};
\draw (96,47) node [anchor=north west][inner sep=0.75pt]  [xscale=0.8,yscale=0.8] [align=left] {$\displaystyle z$-plane};
\draw (108,96.4) node [anchor=north west][inner sep=0.75pt]  [font=\footnotesize,xscale=0.8,yscale=0.8]  {$U$};
\draw (269,198.4) node [anchor=north west][inner sep=0.75pt]  [xscale=0.8,yscale=0.8]  {$z\longleftrightarrow \zeta $};
\draw (536,234) node [anchor=north west][inner sep=0.75pt]  [xscale=0.8,yscale=0.8] [align=left] {0};
\draw (422.5,161.4) node [anchor=north west][inner sep=0.75pt]  [font=\small,xscale=0.8,yscale=0.8]  {$\Sigma _{1}^{( loc)}$};
\draw (420.5,261.4) node [anchor=north west][inner sep=0.75pt]  [font=\small,xscale=0.8,yscale=0.8]  {$\Sigma _{1}^{( loc) *}$};
\draw (469,202.4) node [anchor=north west][inner sep=0.75pt]  [font=\small,xscale=0.8,yscale=0.8]  {$\zeta ( z_{0})$};
\draw (456,54) node [anchor=north west][inner sep=0.75pt]  [xscale=0.8,yscale=0.8] [align=left] {$\displaystyle \zeta $-plane};
\draw (461,97.4) node [anchor=north west][inner sep=0.75pt]  [font=\footnotesize,xscale=0.8,yscale=0.8]  {$U^{( \zeta )}$};
	\end{tikzpicture}
\caption{\footnotesize  The contour  $\Sigma^{(1)}\cap U$ in the $z$-plane and  the contour $\Sigma^{(loc,\zeta)}\cap U^{(\zeta)}$ in the $\zeta$-plane under the mapping \eqref{def zeta1}. }
\label{Fig loc jump}
\end{figure}
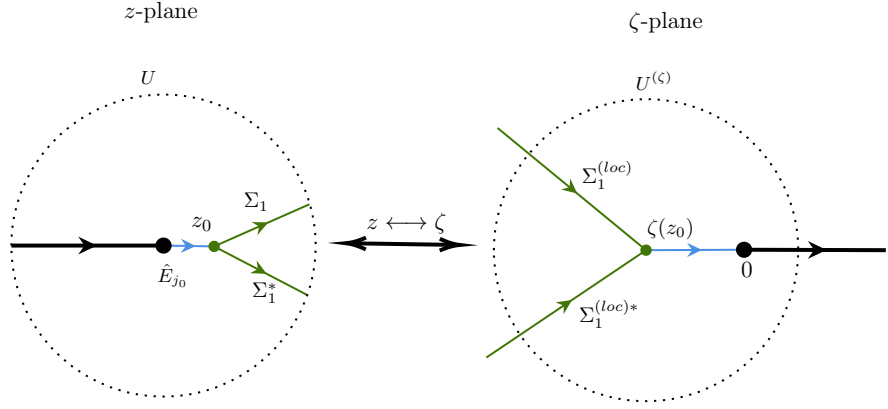
By defining
\begin{align}\label{def:Jloc}
	J^{(loc)}(\zeta) :=	\left\{ \begin{array}{ll}
		e^{i(f_0x+g_0t-(B_{j_0}^fx+B_{j_0}^gt+\phi_{j_0}-\delta_{j_0})/2)\hat{\sigma}_3} \begin{pmatrix} 
			0 & i\\
			i& 0
		\end{pmatrix} , & \zeta\in \mathbb{R}^+,\\[12pt]
		 \begin{pmatrix} 
			1 & -r_1(z(\zeta))e^{2it\theta(z(\zeta))+2\delta(z(\zeta))} \\
			0 & 1
		\end{pmatrix} ,&  \zeta\in\Sigma_1^{(loc)}\cap U^{( \zeta )},\\[12pt]
		 \begin{pmatrix} 
			1 & 0\\
		r_2(z(\zeta))e^{-2it\theta(z(\zeta))-2\delta(z(\zeta))} & 1	
		\end{pmatrix}, &\zeta\in\Sigma_1^{(loc)*}\cap U^{( \zeta )},\\[12pt]
		\begin{pmatrix} 
			1 & -r_1(z(\zeta))e^{2it\theta(z(\zeta))+2\delta(z(\zeta))}\\
			r_2(z(\zeta))e^{-2it\theta(z(\zeta))-2\delta(z(\zeta))} & 1-r_1(z(\zeta))r_2(z(\zeta))
		\end{pmatrix} , & \zeta\in (\zeta(z_0),0),
	\end{array}\right.
\end{align}
it is readily seen that $J^{(loc)}(\zeta(z))=J^{(1)}(z)$ for $z\in\Sigma^{(1)}\cap U$. We split the analysis into three steps in what follows.

\vspace{3mm}

\paragraph{\textbf{Step 1: construction of a model RH problem}}
As $t\to\infty$,  it follows from the definition of $U$ in \eqref{def Uxi I} and \eqref{def c0} that 
\begin{align*}
    z-\hat{E}_{j_0}=\mathcal{O}(t^{\epsilon-2/3}), \qquad z\in U.
\end{align*}
In view of \eqref{eq:riloc} and item (c) of  Proposition \ref{Pro delta1 }, one has, for $z\in U \setminus (-\infty, \hat{E}_{j_0})$ and large $t$, 
	\begin{align*}
	&r_1(z)=-e^{-i\phi_{j_0}}+\mathcal{O}((z-\hat{E}_{j_0})^{1/2})=-e^{-i\phi_{j_0}}+\mathcal{O}(t^{-1/3+\epsilon/2}),\\
    &\delta(z)=\frac{i}{2}\delta_{j_0}+\mathcal{O}((z-\hat{E}_{j_0})^{1/2})=\frac{i}{2}\delta_{j_0}+\mathcal{O}(t^{-1/3+\epsilon/2}),
\end{align*} 
where recall that the constants $\delta_j$ are determined through the linear systems \eqref{def deltaj I}. Together with the expansion of $\theta$ in \eqref{asy theta+} and the fact that $2t\theta(\hat{E}_{j_0})=f_0x+g_0t-(B_{j_0}^fx+B_{j_0}^gt)$ (see  \eqref{theta Ej}), we are lead to consider the following matrix-valued function defined on $\Sigma^{(loc,\zeta)}$:
	\begin{align}\label{def Jloc1}
 	J^{(loc,1)}(\zeta):& = e^{it\theta(\hat{E}_{j_0})\hat{\sigma}_3}\nonumber\\
	&\quad \times	\left\{ \begin{array}{ll}
		e^{-i(\phi_{j_0}-\delta_{j_0})\hat{\sigma}_3/2}\begin{pmatrix} 
			0 & i\\
			i& 0
		\end{pmatrix}, & \zeta\in \mathbb{R}^+,\\
		\begin{pmatrix} 
			1 & e^{-i\phi_{j_0}+2\hat{\theta}(\zeta)+i\delta_{j_0}} \\
			0 & 1
		\end{pmatrix} ,&  \zeta\in\Sigma_1^{(loc)},\\
		\begin{pmatrix} 
			1 & 0\\
			e^{i\phi_{j_0}+2\hat{\theta}(\zeta)-i\delta_{j_0}} & 1	
		\end{pmatrix} , &\zeta\in\Sigma_1^{(loc)*},\\
		\begin{pmatrix} 
			1 & 0\\
			e^{i\phi_{j_0}+2\hat{\theta}(\zeta)-i\delta_{j_0}} & 1	
		\end{pmatrix} \begin{pmatrix} 
			1 & e^{-i\phi_{j_0}+2\hat{\theta}(\zeta)+i\delta_{j_0}} \\
			0 & 1
		\end{pmatrix}, & \zeta\in (\zeta(z_0),0),
	\end{array}\right.
\end{align}	
where
\begin{align}\label{def:htheta}
	\hat{\theta}(\zeta):=\hat{\theta}(\zeta;s)=s\zeta^{1/2}+\frac{2}{3}\zeta^{3/2}.
\end{align}

Note that as $t\to\infty$, $J^{(loc)}(\zeta)$ is well-approximated by $J^{(loc,1)}(\zeta)$ for $\zeta \in \Sigma_1^{(loc)}\cap U^{(\zeta)}$. We next build a matrix-valued function whose jump on $\Sigma^{(loc,\zeta)}$ is given by $J^{(loc,1)}$. The construction is made with the aid of the 
Painlev\'e XXXIV parametrix $M^{(P_{34})}(\zeta;s,\alpha,\omega)$ introduced in Appendix \ref{APP P34}. This parametrix depends on two parameters $\alpha$ and $\omega$ in general, and the one relevant to our case corresponds to $\alpha=-1/4$ and $\omega=0$. More precisely, we define 
\begin{align}\label{def MP34}
	M^{(loc,1)}(\zeta):=\begin{pmatrix}
		1 & 0\\
		ia(s) & 1
	\end{pmatrix}M^{(P_{34})}(\zeta;s,-1/4,0)e^{\hat{\theta}(\zeta)\sigma_3}G_1(\zeta)G_2(\zeta),
\end{align}
where $a$ defined in \eqref{def:a} is an integral of the Painlev\'e XXXIV transcendent,  \begin{align}\label{def G1 I}
	G_1(\zeta) :=	\left\{ \begin{array}{ll}
		\begin{pmatrix} 
			0 &  	-1\\
			1 & 0
		\end{pmatrix}e^{\pi i \sigma_3/4}e^{-i(t\theta(\hat{E}_{j_0})-\phi_{j_0}+\delta_{j_0})\sigma_3/2},&  \zeta\in\mathbb{C}^+,\\
		e^{\pi i \sigma_3/4}e^{-i(t\theta(\hat{E}_{j_0})-\phi_{j_0}+\delta_{j_0})\sigma_3/2} , & \zeta\in\mathbb{C}^-.
	\end{array}\right.
\end{align}
and 
\begin{equation} \label{def:G2}
G_2(\zeta):=e^{it\theta(\hat{E}_{j_0})\hat{\sigma}_3}	\left\{ \begin{array}{ll}
		\begin{pmatrix} 
			1 & e^{-i\phi_{j_0}+2\hat{\theta}(\zeta)+i\delta_{j_0}} \\
			0 & 1
		\end{pmatrix} , &  \quad \zeta\in\Omega_1^{(loc)},\\
		\begin{pmatrix} 
			1 & 0\\
			e^{i\phi_{j_0}+2\hat{\theta}(\zeta)-i\delta_{j_0}} & 1	
		\end{pmatrix} , &\quad \zeta\in\Omega_1^{(loc)*},\\
        I, &\quad \text{ elsewhere},
	\end{array}\right.
\end{equation}
with the region $\Omega_1^{(loc)}$ illustrated in the right picture of Figure \ref{Fig move loc}. 
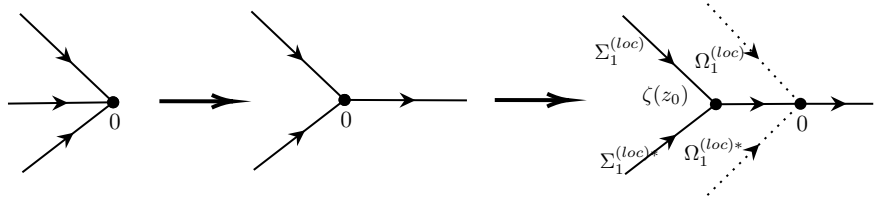
\begin{figure}[h]
	\tikzset{every picture/.style={line width=0.75pt}} 
	\begin{tikzpicture}[x=0.75pt,y=0.75pt,yscale=-0.75,xscale=0.75]
\draw [color={rgb, 255:red, 0; green, 0; blue, 0 }  ,draw opacity=1 ]   (549.07,224.03) -- (605.38,223.73) ;
\draw [shift={(582.23,223.86)}, rotate = 179.7] [fill={rgb, 255:red, 0; green, 0; blue, 0 }  ,fill opacity=1 ][line width=0.08]  [draw opacity=0] (10.72,-5.15) -- (0,0) -- (10.72,5.15) -- (7.12,0) -- cycle    ;
\draw [line width=0.75]    (653.8,223.71) -- (605.38,223.73) ;
\draw [shift={(605.38,223.73)}, rotate = 179.97] [color={rgb, 255:red, 0; green, 0; blue, 0 }  ][fill={rgb, 255:red, 0; green, 0; blue, 0 }  ][line width=0.75]      (0, 0) circle [x radius= 3.35, y radius= 3.35]   ;
\draw [shift={(636.09,223.72)}, rotate = 179.97] [fill={rgb, 255:red, 0; green, 0; blue, 0 }  ][line width=0.08]  [draw opacity=0] (10.72,-5.15) -- (0,0) -- (10.72,5.15) -- (7.12,0) -- cycle    ;
\draw [color={rgb, 255:red, 0; green, 0; blue, 0 }  ,draw opacity=1 ]   (487.04,164.66) -- (549.07,224.03) ;
\draw [shift={(521.67,197.8)}, rotate = 223.75] [fill={rgb, 255:red, 0; green, 0; blue, 0 }  ,fill opacity=1 ][line width=0.08]  [draw opacity=0] (10.72,-5.15) -- (0,0) -- (10.72,5.15) -- (7.12,0) -- cycle    ;
\draw [color={rgb, 255:red, 0; green, 0; blue, 0 }  ,draw opacity=1 ]   (488.65,270.96) -- (549.07,224.03) ;
\draw [shift={(549.07,224.03)}, rotate = 322.17] [color={rgb, 255:red, 0; green, 0; blue, 0 }  ,draw opacity=1 ][fill={rgb, 255:red, 0; green, 0; blue, 0 }  ,fill opacity=1 ][line width=0.75]      (0, 0) circle [x radius= 3.35, y radius= 3.35]   ;
\draw [shift={(522.81,244.43)}, rotate = 142.17] [fill={rgb, 255:red, 0; green, 0; blue, 0 }  ,fill opacity=1 ][line width=0.08]  [draw opacity=0] (10.72,-5.15) -- (0,0) -- (10.72,5.15) -- (7.12,0) -- cycle    ;
\draw [line width=0.75]    (383.38,221.43) -- (301.93,221.13) ;
\draw [shift={(301.93,221.13)}, rotate = 180.21] [color={rgb, 255:red, 0; green, 0; blue, 0 }  ][fill={rgb, 255:red, 0; green, 0; blue, 0 }  ][line width=0.75]      (0, 0) circle [x radius= 3.35, y radius= 3.35]   ;
\draw [shift={(349.15,221.31)}, rotate = 180.21] [fill={rgb, 255:red, 0; green, 0; blue, 0 }  ][line width=0.08]  [draw opacity=0] (10.72,-5.15) -- (0,0) -- (10.72,5.15) -- (7.12,0) -- cycle    ;
\draw [color={rgb, 255:red, 0; green, 0; blue, 0 }  ,draw opacity=1 ]   (239.9,161.76) -- (301.93,221.13) ;
\draw [shift={(274.52,194.9)}, rotate = 223.75] [fill={rgb, 255:red, 0; green, 0; blue, 0 }  ,fill opacity=1 ][line width=0.08]  [draw opacity=0] (10.72,-5.15) -- (0,0) -- (10.72,5.15) -- (7.12,0) -- cycle    ;
\draw [color={rgb, 255:red, 0; green, 0; blue, 0 }  ,draw opacity=1 ]   (241.51,268.06) -- (301.93,221.13) ;
\draw [shift={(301.93,221.13)}, rotate = 322.17] [color={rgb, 255:red, 0; green, 0; blue, 0 }  ,draw opacity=1 ][fill={rgb, 255:red, 0; green, 0; blue, 0 }  ,fill opacity=1 ][line width=0.75]      (0, 0) circle [x radius= 3.35, y radius= 3.35]   ;
\draw [shift={(275.67,241.53)}, rotate = 142.17] [fill={rgb, 255:red, 0; green, 0; blue, 0 }  ,fill opacity=1 ][line width=0.08]  [draw opacity=0] (10.72,-5.15) -- (0,0) -- (10.72,5.15) -- (7.12,0) -- cycle    ;
\draw [line width=1.5]    (401.9,221.2) -- (447.98,221.28) ;
\draw [shift={(450.98,221.28)}, rotate = 180.1] [color={rgb, 255:red, 0; green, 0; blue, 0 }  ][line width=1.5]    (14.21,-4.28) .. controls (9.04,-1.82) and (4.3,-0.39) .. (0,0) .. controls (4.3,0.39) and (9.04,1.82) .. (14.21,4.28)   ;
\draw  [dash pattern={on 0.84pt off 2.51pt}]  (543.43,156.68) -- (605.38,223.73) ;
\draw [shift={(577.8,193.88)}, rotate = 227.27] [fill={rgb, 255:red, 0; green, 0; blue, 0 }  ][line width=0.08]  [draw opacity=0] (10.72,-5.15) -- (0,0) -- (10.72,5.15) -- (7.12,0) -- cycle    ;
\draw  [dash pattern={on 0.84pt off 2.51pt}]  (543.43,285) -- (605.38,223.73) ;
\draw [shift={(577.96,250.85)}, rotate = 135.32] [fill={rgb, 255:red, 0; green, 0; blue, 0 }  ][line width=0.08]  [draw opacity=0] (10.72,-5.15) -- (0,0) -- (10.72,5.15) -- (7.12,0) -- cycle    ;
\draw [line width=0.75]    (77.9,223.6) -- (147.93,222.73) ;
\draw [shift={(147.93,222.73)}, rotate = 359.29] [color={rgb, 255:red, 0; green, 0; blue, 0 }  ][fill={rgb, 255:red, 0; green, 0; blue, 0 }  ][line width=0.75]      (0, 0) circle [x radius= 3.35, y radius= 3.35]   ;
\draw [shift={(117.91,223.11)}, rotate = 179.29] [fill={rgb, 255:red, 0; green, 0; blue, 0 }  ][line width=0.08]  [draw opacity=0] (10.72,-5.15) -- (0,0) -- (10.72,5.15) -- (7.12,0) -- cycle    ;
\draw [color={rgb, 255:red, 0; green, 0; blue, 0 }  ,draw opacity=1 ]   (85.9,163.36) -- (147.93,222.73) ;
\draw [shift={(120.52,196.5)}, rotate = 223.75] [fill={rgb, 255:red, 0; green, 0; blue, 0 }  ,fill opacity=1 ][line width=0.08]  [draw opacity=0] (10.72,-5.15) -- (0,0) -- (10.72,5.15) -- (7.12,0) -- cycle    ;
\draw [color={rgb, 255:red, 0; green, 0; blue, 0 }  ,draw opacity=1 ]   (87.51,269.66) -- (147.93,222.73) ;
\draw [shift={(147.93,222.73)}, rotate = 322.17] [color={rgb, 255:red, 0; green, 0; blue, 0 }  ,draw opacity=1 ][fill={rgb, 255:red, 0; green, 0; blue, 0 }  ,fill opacity=1 ][line width=0.75]      (0, 0) circle [x radius= 3.35, y radius= 3.35]   ;
\draw [shift={(121.67,243.13)}, rotate = 142.17] [fill={rgb, 255:red, 0; green, 0; blue, 0 }  ,fill opacity=1 ][line width=0.08]  [draw opacity=0] (10.72,-5.15) -- (0,0) -- (10.72,5.15) -- (7.12,0) -- cycle    ;
\draw [line width=1.5]    (178.7,222) -- (224.78,222.08) ;
\draw [shift={(227.78,222.08)}, rotate = 180.1] [color={rgb, 255:red, 0; green, 0; blue, 0 }  ][line width=1.5]    (14.21,-4.28) .. controls (9.04,-1.82) and (4.3,-0.39) .. (0,0) .. controls (4.3,0.39) and (9.04,1.82) .. (14.21,4.28)   ;

\draw (601.17,230.17) node [anchor=north west][inner sep=0.75pt]  [xscale=0.8,yscale=0.8] [align=left] {0};
\draw (466.65,177.45) node [anchor=north west][inner sep=0.75pt]  [font=\small,xscale=0.8,yscale=0.8]  {$\Sigma _{1}^{( loc)}$};
\draw (470.15,250.64) node [anchor=north west][inner sep=0.75pt]  [font=\small,xscale=0.8,yscale=0.8]  {$\Sigma _{1}^{( loc) *}$};
\draw (498.33,209.53) node [anchor=north west][inner sep=0.75pt]  [font=\small,xscale=0.8,yscale=0.8]  {$\zeta ( z_{0})$};
\draw (297.72,227.57) node [anchor=north west][inner sep=0.75pt]  [xscale=0.8,yscale=0.8] [align=left] {0};
\draw (533.4,183.11) node [anchor=north west][inner sep=0.75pt]  [font=\small,xscale=0.8,yscale=0.8]  {$\Omega _{1}^{( loc)}$};
\draw (526.25,244.9) node [anchor=north west][inner sep=0.75pt]  [font=\small,xscale=0.8,yscale=0.8]  {$\Omega _{1}^{( loc) *}$};
\draw (143.72,229.17) node [anchor=north west][inner sep=0.75pt]  [xscale=0.8,yscale=0.8] [align=left] {0};
\end{tikzpicture}
	\caption{\footnotesize  The jump contours of the RH problems for $M^{(P_{34})}$
		(left), $\hat{M}$ (middle) and $M^{(loc,1)}$ (right). }
	\label{Fig move loc}
\end{figure}

\begin{proposition}\label{Mloc,1}
The function $M^{(loc,1)}$ defined in \eqref{def MP34} solves the following RH problem.
		\begin{itemize}
			\item [$\mathrm{(a)}$] $M^{(loc,1)}(\zeta)$ is analytic in $\mathbb{C}\setminus\Sigma^{(loc,\zeta)}$, where $\Sigma^{(loc,\zeta)}$ is defined in \eqref{def Sigmalocal}. 
			\item [$\mathrm{(b)}$]  
			For $\zeta\in\Sigma^{(loc,\zeta)}$, $M^{(loc,1)}(\zeta)$ satisfies the jump condition $$M^{(loc,1)}_+(\zeta)=M^{(loc,1)}_{-}(\zeta)J^{(loc,1)}(\zeta),$$
			where $J^{(loc,1)}$ is defined in \eqref{def Jloc1}.
			\item [$\mathrm{(c)}$] $M^{(loc,1)}(\zeta)$ has at most $-1/4$-singularity at $0$.
			\item [$\mathrm{(d)}$] As $\zeta\to\infty$, we have
			\begin{align}\label{eq:asyMloc1}
				M^{(loc,1)}(\zeta)=	\left (I
				+\mathcal{O}\left(  \zeta^{-1}\right) \right)
				\frac{\zeta^{-\frac{1}{4}\sigma_3}}{\sqrt{2}}
				\begin{pmatrix}
					1 & i
					\\
					i & 1
				\end{pmatrix}G_1(\zeta),
			\end{align}
			with $G_1(\zeta) $ given in \eqref{def G1 I}.
		\end{itemize}
\end{proposition}
\begin{proof}
Items (c) and (d) can be checked directly from items (c) and (d) of RH problem \ref{RH P34} for $M^{(P_{34})}$ and the definition of $M^{(loc,1)}$. We thus focus on items (a) and (b). To proceed, we first consider 
	\begin{align}\label{tranP341}
		\hat{M}(\zeta):=\begin{pmatrix}
			1 & 0\\
			ia(s) & 1
		\end{pmatrix}M^{(P_{34})}(\zeta;s,-1/4,0)e^{\hat{\theta}(\zeta)\sigma_3}G_1(\zeta). 
	\end{align}
The jump contour of $M^{(P_{34})}(\zeta;s,-1/4,0)$ is shown in the left picture of Figure \ref{Fig move loc}. Thus, it is readily seen that $\hat{M}(\zeta)$ is analytic for  $\zeta\in \mathbb{C} \setminus \{\cup_{j=1}^4\Sigma_j\}$, where $\Sigma_j$, $j=1,2,3,4$, is defined in \eqref{def jumpcureve 34}. Moreover, with the jump matrix $J^{(P_{34})}$ of $M^{(P_{34})}$ given in \eqref{eq:Psi-jump},
the jump of $\hat{M}$ on each $\Sigma_j$ reads
	\begin{align*}
		\hat{M}_-(\zeta)^{-1}\hat{M}_+(\zeta)&=G_{1,-}(\zeta)^{-1}e^{-\hat{\theta}_-(\zeta)\sigma_3}J^{(P_{34})}(\zeta)e^{\hat{\theta}_+(\zeta)\sigma_3}G_{1,+}(\zeta)\\
&=	e^{it\theta(\hat{E}_{j_0})\hat{\sigma}_3}\left\{ \begin{array}{ll}
				e^{-i(\phi_{j_0}-\delta_{j_0})\hat{\sigma}_3/2}\begin{pmatrix} 
					0 & i\\
					i& 0
				\end{pmatrix}, & \zeta\in \Sigma_1,\\
				\begin{pmatrix} 
					1 & e^{-i\phi_{j_0}+2\hat{\theta}(\zeta)+i\delta_{j_0}} \\
					0 & 1
				\end{pmatrix} ,&  \zeta\in\Sigma_2,\\
				I&  \zeta\in\Sigma_3,\\
				\begin{pmatrix} 
					1 & 0\\
					e^{i\phi_{j_0}+2\hat{\theta}(\zeta)-i\delta_{j_0}} & 1	
				\end{pmatrix} , &\zeta\in\Sigma_4.
			\end{array}\right.
	\end{align*}
Since $e^{it\theta(\hat{E}_{j_0})\hat{\sigma}_3}I=I$, $\hat{M}$ has no jump on the negative axis and its jump contour is shown in the middle picture of Figure \ref{Fig move loc}. 
Next, in view of \eqref{tranP341} and \eqref{def MP34}, we have
 \begin{align*}
     M^{(loc,1)}(\zeta)=\hat{M}(\zeta)G_2(\zeta).
 \end{align*}
By \eqref{def:G2}, this transformation particularly removes the jump of $\hat{M}$ on $\Sigma_2$ and $\Sigma_4$ to  $\Sigma_{1}^{( loc)}$ and $\Sigma_{1}^{( loc),*}$, respectively. It is then easily seen that the jump contour of $M^{(loc,1)}(\zeta)$ is shown in the right picture of Figure \ref{Fig move loc} and satisfies the jump condition indicated in item (b).  
\end{proof}

\vspace{2mm}
\paragraph{\textbf{Step 2: error estimate}}
As aforementioned, $J^{(loc,1)}$ in \eqref{def Jloc1} is an approximation of $J^{(loc)}$ for large $t$. It is natural to expect that  $M^{(loc)}$  is well-approximated by  $M^{(loc,1)}$ for large $t$. To see this, we start with introduction of an RH problem, whose  jump is the same as in $M^{(loc)}$ and large $\zeta$ asymptotics is the same as  $M^{(loc,1)}$.
\begin{RHP}\label{RHP:Mloc0}
	\hfil
	\begin{itemize}
		\item [$\mathrm{(a)}$] $M^{(loc,0)}(\zeta)$ is analytic in $\mathbb{C}\setminus\Sigma^{(loc,\zeta)}$, 
        where $\Sigma^{(loc,\zeta)}$ is defined in \eqref{def Sigmalocal}. 
		\item [$\mathrm{(b)}$]  For $\zeta\in\Sigma^{(loc,\zeta)}$, $M^{(loc,0)}(\zeta)$ satisfies the jump condition $$M^{(loc,0)}_+(\zeta)=M^{(loc,0)}_{-}(\zeta)J^{(loc)}(\zeta),$$
        where $J^{(loc)}$ is defined in \eqref{def:Jloc}.   
	\item [$\mathrm{(c)}$] $M^{(loc,0)}(\zeta)$ has at most $-1/4$-singularity at $0$.	
    \item [$\mathrm{(d)}$] As $\zeta\to\infty$, we have 
	\begin{align}\label{M loc asyz1}
		M^{(loc,0)}(\zeta)=\left (I
		+\mathcal{O}\left(  \zeta^{-1}\right) \right)	\frac{\zeta^{-\frac{1}{4}\sigma_3}}{\sqrt{2}}
		\begin{pmatrix}
			1 & i
			\\
			i & 1
		\end{pmatrix}G_1(\zeta)\textit{.}
	\end{align}
		
	\end{itemize}
\end{RHP}
 This invokes us to consider the following RH problem \begin{align}\label{def Xi}
	\Xi(\zeta):=M^{(loc,0)}(\zeta)M^{(loc,1)}(\zeta)^{-1}. 
\end{align} We give a rigorous proof of this fact by establishing an estimate of 
$\Xi$ in the following proposition.
\begin{proposition}\label{Pro Xi 1}
	As $t\to\infty$, $\Xi(\zeta)$ exists uniquely and satisfies
	\begin{align}\label{eq:estXi}
		\Xi(\zeta)=I+\mathcal{O}(t^{-1/3}).
	\end{align}
\end{proposition}
\begin{proof}
In view of RH problems for  $M^{(loc,0)}$ and $M^{(loc,1)}$, it is readily seen from \eqref{def Xi} that $\Xi(\zeta)$ is analytic in $\mathbb{C}\setminus ((\zeta(z_0),0)\cup\Sigma_1^{(loc)}\cup\Sigma_1^{(loc)*})$ and satisfies
$$\Xi_+(\zeta)=\Xi_-(\zeta)J^{(\Xi)}, 
\qquad \zeta \in (\zeta(z_0),0)\cup\Sigma_1^{(loc)}\cup\Sigma_1^{(loc)*},$$ 
where
	\begin{align*}
		J^{(\Xi)}(\zeta)=M^{(loc,1)}_-(\zeta)J^{(loc)}(\zeta)J^{(loc,1)}(\zeta)^{-1}M^{(loc,1)}_-(\zeta)^{-1}.
	\end{align*}

To proceed, we recall the following  factorizations of matrices $J^{(loc)}$ and $J^{(loc,1)}$:
\begin{align*}
	J^{(loc)}(\zeta)=(I-w^{(-)}(\zeta))^{-1}(I+w^{(+)}(\zeta)),\qquad	J^{(loc,1)}(\zeta)=(I-w^{(-,1)}(\zeta))^{-1}(I+w^{(+,1)}(\zeta)),
\end{align*}
where
\begin{align}
	&w^{(+)}(\zeta)=\left\{ \begin{array}{ll}
	J^{(loc)}(\zeta)-I, & \zeta\in \mathbb{R}^+,\\
		\begin{pmatrix} 
			0 & r_1(z(\zeta))e^{2it\theta(z(\zeta))+2\delta(z(\zeta))} \\
			0 & 0
		\end{pmatrix} ,&  \zeta\in(\Sigma_1^{(loc)}\cap U^{( \zeta )})\cup(\zeta(z_0),0),\\
		0, & \text{elsewhere},
	\end{array}\right. \label{def:w+}
      \\[4pt]
	&w^{(-)}(\zeta)=\left\{ \begin{array}{ll}
	\begin{pmatrix} 
		0 & 0\\ r_2(z(\zeta))e^{-2it\theta(z(\zeta))+2\delta(z(\zeta))}  & 0
	\end{pmatrix} ,&  \zeta\in(\Sigma_1^{(loc)*} \cap U^{( \zeta )})\cup(\zeta(z_0),0),\\
	0, & \text{elsewhere},
\end{array}\right.
\label{def:w-}
\\[4pt]
&w^{(+,1)}(\zeta)=	\left\{ \begin{array}{ll}
	J^{(loc,1)}(\zeta)-I, & \zeta\in \mathbb{R}^+,\\
	e^{it\theta(\hat{E}_{j_0})\hat{\sigma}_3}\begin{pmatrix} 
		0 & -e^{-i\phi_{j_0}+2\hat{\theta}(\zeta)+i\delta_{j_0}} \\
		0 & 0
	\end{pmatrix} ,&  \zeta\in\Sigma_1^{(loc)}\cup(\zeta(z_0),0),\\
	0, & \text{elsewhere},
\end{array}\right. \label{def:w+1}\\[4pt]
&w^{(-,1)}(\zeta)=	\left\{ \begin{array}{ll}
e^{it\theta(\hat{E}_{j_0})\hat{\sigma}_3}\begin{pmatrix}
	0 &  0\\
	-e^{i\phi_{j_0}+2\hat{\theta}(\zeta)-i\delta_{j_0}} & 0
\end{pmatrix},&  \zeta\in\Sigma_1^{(loc)*}\cup(\zeta(z_0),0),\\
0, & \text{elsewhere}.
\end{array}\right. \label{def:w-1}
\end{align}
We see from RH problem for $M^{(loc,1)}$ given in Proposition \ref{Mloc,1} that  there exists a positive constant $C$ such that 
$$
|M^{(loc,1)}|\lesssim \left\{ \begin{array}{ll}
	|\zeta|^{-1/4}, & |\zeta|\leq C, 
    \\
	|\zeta|^{1/4}, & |\zeta|\geq C. 
\end{array}\right. 
$$
Here, for a $2\times 2$ matrix $A$, $|A|:=\max_{i,j=1,2}\{|A_{ij}|\}$. Thus, if $\zeta\in\Sigma_1^{(loc)}\cup(\zeta(z_0),0)$ and  $|\zeta|\leq C$, we obtain from Propositions \ref{Pro asy ar} and \ref{Pro delta1 }, the definition of $\zeta$ in \eqref{def zeta1}, \eqref{asy theta+}, \eqref{def:w+} and \eqref{def:w+1} that
\begin{align*}
	&|w^{(+,1)}(\zeta)-w^{(+)}(\zeta)|=| r_1(\hat{E}_{j_0})e^{2\hat{\theta}(\zeta)+2\delta(\hat{E}_{j_0})+2it\theta(\hat{E}_{j_0})}-r_1(z(\zeta))e^{2it\theta(z(\zeta))+2\delta(z(\zeta))}|\\
	&\leq |r_1(\hat{E}_{j_0})e^{2\delta(\hat{E}_{j_0})}-r_1(z(\zeta))\textbf{}e^{2\delta(z(\zeta))}||e^{2it\theta(z(\zeta))}|+|r_1(z(\zeta))e^{2it\theta(z(\zeta))+2\delta(z(\zeta))}||1-e^{(S(z(\zeta))-s)\zeta^{1/2}}|\\
	&\lesssim t^{-1/3}|\zeta^{1/2}|+t^{-2/3}|\zeta^{3/2}|,
\end{align*}
which gives us that
\begin{align*}
	\frac{1}{|\zeta|^{1/2}}|w^{(+,1)}(\zeta)-w^{(+)}(\zeta)|\lesssim t^{-1/3}+t^{-2/3}|\zeta|\lesssim t^{-1/3}.
\end{align*}
Similarly,  by \eqref{def:w-} and \eqref{def:w-1}, one has $|\frac{1}{\zeta^{1/2}}(w^{(-,1)}(\zeta)-w^{(-)}(\zeta))|\lesssim t^{-1/3}+t^{-2/3}|\zeta|\lesssim t^{-1/3}$. This, together with the above estimate, implies that
\begin{align*}
	&\|M^{(loc,1)}_-(\zeta)(J^{(loc)}(\zeta)J^{(loc,1)}(\zeta)^{-1}-I)M^{(loc,1)}_-(\zeta)^{-1}\|_{L^\infty\cap L^1\cap L^2(\Sigma_1^{(loc)}\cup(\zeta(z_0),0)\cap\{\zeta:|\zeta\leq C\})}\\
	&\lesssim\|\frac{1}{\zeta^{1/2}}(w^{(+,1)}(\zeta)-w^{(+)}(\zeta))\|_{L^\infty\cap L^1\cap L^2(\Sigma_1^{(loc)}\cup(\zeta(z_0),0)\cap\{\zeta:|\zeta\leq C\})}\\
	&\quad+\|\frac{1}{\zeta^{1/2}}(w^{(-,1)}(\zeta)-w^{(-)}(\zeta))\|_{L^\infty\cap L^1\cap L^2(\Sigma_1^{(loc)}\cup(\zeta(z_0),0)\cap\{\zeta:|\zeta\leq C\})}
\lesssim  t^{-1/3}.
\end{align*}
As for $|\zeta|\geq C$, if $\zeta\notin  U^{( \zeta )}$, from the definition of $U^{( \zeta )}$ in \eqref{def Uzeta}, it is obvious that $|w^{(+,1)}(\zeta)|\lesssim e^{-ct^{\epsilon/2}}$ for some $c>0$, and if  $\zeta\in  U^{( \zeta )}$, 
\begin{align*}
	&|w^{(+,1)}(\zeta)-w^{(+)}(\zeta)|\\
	&\leq |r_1(\hat{E}_{j_0})e^{2\delta(\hat{E}_{j_0})}-r_1(z(\zeta))e^{2\delta(z(\zeta))}||e^{2it\theta(z(\zeta))}|+|r_1(z(\zeta))e^{2it\theta(z(\zeta))+2\delta(z(\zeta))}||1-e^{(S(z(\zeta))-s)\zeta^{1/2}}|\\
	&\lesssim |e^{2it\theta(z(\zeta))}|(1+t^{-2/3}|\zeta|^{3/2}).
\end{align*}
Since $|w^{(-,1)}(\zeta)-w^{(-)}(\zeta)|$ admits similar estimates, it follows that 
\begin{align*}
	&\|M^{(loc,1)}_-(\zeta)(J^{(loc)}(\zeta)J^{(loc,1)}(\zeta)^{-1}-I)M^{(loc,1)}_-(\zeta)^{-1}\|_{L^\infty\cap L^1\cap L^2(\Sigma_1^{(loc)}\cup(\zeta(z_0),0)\cap\{\zeta:|\zeta\geq C\})}\\	&\lesssim\||\zeta|^{1/2}(w^{(+,1)}(\zeta)-w^{(+)}(\zeta))\|_{L^\infty\cap L^1\cap L^2(\Sigma_1^{(loc)}\cup(\zeta(z_0),0)\cap\{\zeta:|\zeta\geq C\})}\\
	&\quad +\||\zeta|^{1/2}(w^{(-,1)}(\zeta)-w^{(-)}(\zeta))\|_{L^\infty\cap L^1\cap L^2(\Sigma_1^{(loc)}\cup(\zeta(z_0),0)\cap\{\zeta:|\zeta\geq C\})}\lesssim e^{-\tilde{C}t^{\epsilon/2}},
\end{align*}
for some $\tilde{C}>0$. In conclusion, we arrive at 
\begin{align*}
	\|M^{(loc,1)}_-(\zeta)(J^{(loc)}(\zeta)J^{(loc,1)}(\zeta)^{-1}-I)M^{(loc,1)}_-(\zeta)^{-1}\|_{L^\infty\cap L^1\cap L^2(\Sigma_1^{(loc)}\cup(\zeta(z_0),0))}\lesssim t^{-1/3}.
\end{align*}
It is thereby inferred from small norm RH problem arguments \cite{PX3} that $\Xi$ exists uniquely,  which also yields the estimate \eqref{eq:estXi}.
\end{proof}

From Proposition \ref{Pro Xi 1}, it is immediate that $M^{(loc,0)}$ is uniquely solvable for large positive $t$. Moreover, as $\zeta\to\infty$, we see from \eqref{def MP34},  \eqref{eq:Psi-infinity} and \eqref{eq:12entry} that  
\begin{align}\label{asy Mloc0I}
	M^{(loc,0)}=\left (I+\frac{M^{(loc,0)}_1}{\zeta}
	+\mathcal{O}\left(  \zeta^{-2}\right) \right)	\frac{\zeta^{-\frac{1}{4}\sigma_3}}{\sqrt{2}}
	\begin{pmatrix}
		1 & i
		\\
		i & 1
	\end{pmatrix}G_1(\zeta),
\end{align}
where  
\begin{align}\label{asy loc1}
	(M^{(loc,0)}_1)_{12}=(M^{(P_{34})}_1)_{12}+\mathcal{O}(t^{-1/3}) =ia(s)+\mathcal{O}(t^{-1/3})
\end{align}
with $a(s)$ and $s$ given in \eqref{def:a}
and \eqref{def s1}, respectively. 

\vspace{2mm}
\paragraph{\textbf{Step 3: construction of $M^{(loc)}$}}
With the aid of $M^{(loc,0)}$, we finally define
\begin{align}\label{def Mloc}
	M^{(loc)}(z)=H_1(z)t^{\sigma_3/6}M^{(loc,0)}(\zeta(z)), \qquad z\in U, 
\end{align}
where
\begin{align}\label{def H1 I}
	H_1(z)=	M^{(glo)}(z)G_1(\zeta(z))^{-1}\frac{1}{\sqrt{2}}(I-i\sigma_1)\left( \zeta(z)t^{-2/3}\right) ^{\sigma_3/4}.
\end{align}
Here, $M^{(glo)}$, $G_1$ and $\theta^{(3,\hat{j_0})}$ are defined in \eqref{def Mmod}, \eqref{def G1 I} and \eqref{theta hj03}, respectively. 

We claim that $H_1(z)$ is analytic in $U$. Indeed,  for $z\in U\cap(-\infty, \hat{E}_{j_0})$, from the jump condition of $M^{(glo)}$ in \eqref{eq:Mglojump}, straightforward calculations shows that
\begin{align*}
	&H_{1,+}(z)=M^{(glo)}_+(z)G_{1,+}(\zeta(z))^{-1}\frac{1}{\sqrt{2}}(I-i\sigma_1)\left( \zeta(z)t^{-2/3}\right) ^{\sigma_3/4}\\
    &=M^{(glo)}_-(z)e^{i(t\theta(\hat{E}_{j_0})-\phi_{j_0}+\delta_{j_0})\hat{\sigma}_3/2}\begin{pmatrix} 
		0 & -i\\
		-i& 0
	\end{pmatrix} 	e^{-\pi i \sigma_3/4}e^{i(t\theta(\hat{E}_{j_0})-\phi_{j_0}+\delta_{j_0})\sigma_3/2}\\
	&\quad \times \frac{1}{\sqrt{2}}(I-i\sigma_1)\left( \zeta(z)t^{-2/3}\right) ^{\sigma_3/4}\\
	&=M^{(glo)}_-(z)\begin{pmatrix} 
		0 &  	e^{-\pi i /4}e^{i(t\theta(\hat{E}_{j_0})-\phi_{j_0}+\delta_{j_0})/2}\\
		-e^{\pi i /4}e^{-i(t\theta(\hat{E}_{j_0})-\phi_{j_0}+\delta_{j_0})/2} & 0
	\end{pmatrix}\\&\quad \times \frac{1}{\sqrt{2}}(I-i\sigma_1)\left( \zeta(z)t^{-2/3}\right) ^{\sigma_3/4}=H_{1,-}(z).
\end{align*}
On the other hand,  for $z \in U\cap(\hat{E}_{j_0},+\infty)$, we have 
\begin{align*}
	&H_{1,+}(z)=M^{(glo)}(z)G_{1,+}(\zeta(z))^{-1}\frac{1}{\sqrt{2}}(I-i\sigma_1)\left( \zeta(z)t^{-2/3}\right)_+ ^{\sigma_3/4}\\
    &=M^{(glo)}(z)G_{1,-}(\zeta(z))^{-1}G_{1,-}(\zeta(z))G_{1,+}(\zeta(z))^{-1}\frac{1}{\sqrt{2}}(I-i\sigma_1)(-i\sigma_3)\left( \zeta(z)t^{-2/3}\right) ^{\sigma_3/4}_-\\
	&=M^{(glo)}(z)G_{1,-}(\zeta(z))^{-1}\begin{pmatrix} 
		0 &  	-e^{-\pi i /4}e^{i(t\theta(\hat{E}_{j_0})-\phi_{j_0}+\delta_{j_0})/2}\\
		e^{\pi i /4}e^{-i(t\theta(\hat{E}_{j_0})-\phi_{j_0}+\delta_{j_0})/2} & 0
	\end{pmatrix} 	e^{i(t\theta(\hat{E}_{j_0})-\phi_{j_0}+\delta_{j_0})\sigma_3/2}
   \\
	&\quad \times \frac{e^{-\pi i \sigma_3/4}}{\sqrt{2}}(I-i\sigma_1)(-i\sigma_3)\left( \zeta(z)t^{-2/3}\right) _-^{\sigma_3/4}\\
    &=M^{(glo)}(z)G_{1,-}(\zeta(z))^{-1}\frac{1}{\sqrt{2}}(I-i\sigma_1)\left( \zeta(z)t^{-2/3}\right) _-^{\sigma_3/4}=H_{1,-}(z).
\end{align*}
It then follows from the above two formulas that $H_1$ is meromorphic in $U$. By its definition, $H_1$ admits at most $-1/2$-singularity at $\hat{E}_{j_0}$. Thus, $\hat{E}_{j_0}$ is a removable singular point, as required. As a consequence, one can check from \eqref{def Mloc} and items (a)--(c) of RH problem \ref{RHP:Mloc0} that $M^{(loc)}$ satisfies items (a)--(c) of RH problem \ref{Mloc1}. The next proposition shows that $M^{(loc)}$ also satisfies the matching condition, hence solves RH problem \ref{Mloc1}.

\begin{proposition}\label{Pro asy mloc}
	For $z\in \partial U$, we have, as $t\to\infty$,
	\begin{align}\label{eq:matching}
		M^{(loc)}(z)M^{(glo)}(z)^{-1}=I-\frac{t^{1/3}}{\zeta(z)}H_1(z)\begin{pmatrix} 
			0 & ia(s)\\
			0& 0
		\end{pmatrix} H_1(z)^{-1}+\mathcal{O}(t^{1/3-2\epsilon}),
	\end{align}
	where $a$ is given in \eqref{def:a}.
\end{proposition}
\begin{proof}
For $z\in \partial U$, we note from \eqref{def Uxi I} and \eqref{def zeta1} that  $|\zeta(z)| \sim t^{\epsilon}$ as $t\to \infty$. 
It then follows from
\eqref{def Mloc} and the large $\zeta$ asymptotics of $M^{(loc,0)}$ given in \eqref{asy Mloc0I} and \eqref{asy loc1} that, as $t\to\infty$,
\begin{align*}
&M^{(loc)}(z)M^{(glo)}(z)^{-1}-I=H_1(z)t^{\sigma_3/6}M^{(loc,0)}(\zeta(z))M^{(glo)}(z)^{-1}-I\\
&=H_1(z)t^{\sigma_3/6}\left (I+\frac{M^{(loc,0)}_1}{\zeta(z)}
+\mathcal{O}\left(  \zeta^{-2}\right) \right)	t^{-\sigma_3/6}H_1(z)^{-1}-I\\
&=H_1(z)\left(t^{\hat{\sigma}_3/6}\frac{M^{(loc,0)}_1}{\zeta(z)}+\mathcal{O}(t^{-2\epsilon})\right)H_1(z)^{-1}\\
&=\frac{t^{1/3}}{\zeta(z)}H_1(z)\begin{pmatrix} 
	0 & ia(s)\\
	0& 0
\end{pmatrix} H_1(z)^{-1}+\mathcal{O}(t^{1/3-2\epsilon}),
\end{align*}
as indicated in \eqref{eq:matching}. Here
we note that $t^{1/3}/\zeta(z) =\mathcal{O}(t^{1/3-\epsilon}) = o(1) $ for $z\in \partial U$, since  $\epsilon \in (1/3,2/3)$.
\end{proof}

\subsection{The small norm RH problem}\label{subsec E}
From \eqref{construct M1} and the RH problems for $M^{(1)}$, $M^{(glo)}$ and $M^{(loc)}$, it follows that
\begin{equation}
    E(z):=\left\{ \begin{array}{ll}
			M^{(1)}(z)M^{(glo)}(z)^{-1}, & z\in \mathbb{C}\setminus U ,\\
		M^{(1)}(z)M^{(loc)}(z)^{-1},&  z\in U,
	\end{array}\right.
\end{equation}
is meromorphic in $U$ with at most $-1/2$-singularity at $z=\hat{E}_{j_0}$, so $E$ is analytic in $U$. Moreover, it is readily seen that $E$ satisfies the following RH problem. 
\begin{RHP}
 \label{RHP E I} \hfil
	\begin{itemize}
		\item [\rm(a)] $E(z)$ is analytic in  $\mathbb{C}\setminus\Sigma^{(E)}$, where
		\begin{align*}
		\Sigma^{(E)}:=\partial U\cup \Sigma_2\cup \Sigma_2^*\cup\Sigma_1\cup \Sigma_1^*\setminus U;	
		\end{align*}
	see Figure \ref{Fig E I} for an illustration.
		\item [\rm(b)]  For $z\in\Sigma^{(E)}$, $E(z)$ satisfies the jump condition
		$$E_+(z)=E_-(z)J^{(E)}(z),$$ where
		\begin{align*}
			J^{(E)}(z)=	\left\{ \begin{array}{ll}
				M^{(glo)}(z)J^{(1)}(z)M^{(glo)}(z)^{-1}, & z\in \Sigma^{(E)} \setminus \partial U,\\
				M^{(loc)}(z)M^{(glo)}(z)^{-1}, & z\in \partial U.\\	\end{array}\right.
		\end{align*}
		\item [\rm(c)] As $z\to\infty$, we have $E(z)=I+\mathcal{O}(z^{-1})$.
	\end{itemize}
\end{RHP}
\begin{figure}	
	\tikzset{every picture/.style={line width=0.75pt}} 
	
	\begin{tikzpicture}[x=0.75pt,y=0.75pt,yscale=-0.6,xscale=0.6]
		
		\draw [color={rgb, 255:red, 74; green, 144; blue, 226 }  ,draw opacity=1 ] [dash pattern={on 0.84pt off 2.51pt}]  (82,209.42) -- (493,215.42) ;
		\draw [line width=1.5]    (114,209.92) -- (302,211.92) ;
		\draw [shift={(302,211.92)}, rotate = 0.61] [color={rgb, 255:red, 0; green, 0; blue, 0 }  ][fill={rgb, 255:red, 0; green, 0; blue, 0 }  ][line width=1.5]      (0, 0) circle [x radius= 4.36, y radius= 4.36]   ;
		\draw [shift={(114,209.92)}, rotate = 0.61] [color={rgb, 255:red, 0; green, 0; blue, 0 }  ][fill={rgb, 255:red, 0; green, 0; blue, 0 }  ][line width=1.5]      (0, 0) circle [x radius= 4.36, y radius= 4.36]   ;
		\draw [color={rgb, 255:red, 65; green, 117; blue, 5 }  ,draw opacity=1 ]   (85.83,165.4) -- (223.9,211.4) ;
		\draw [shift={(223.9,211.4)}, rotate = 18.43] [color={rgb, 255:red, 65; green, 117; blue, 5 }  ,draw opacity=1 ][fill={rgb, 255:red, 65; green, 117; blue, 5 }  ,fill opacity=1 ][line width=0.75]      (0, 0) circle [x radius= 3.35, y radius= 3.35]   ;
		\draw [shift={(159.61,189.98)}, rotate = 198.43] [fill={rgb, 255:red, 65; green, 117; blue, 5 }  ,fill opacity=1 ][line width=0.08]  [draw opacity=0] (10.72,-5.15) -- (0,0) -- (10.72,5.15) -- (7.12,0) -- cycle    ;
		\draw [color={rgb, 255:red, 65; green, 117; blue, 5 }  ,draw opacity=1 ]   (363.67,232.83) -- (493.27,271.11) ;
		\draw [shift={(433.26,253.39)}, rotate = 196.45] [fill={rgb, 255:red, 65; green, 117; blue, 5 }  ,fill opacity=1 ][line width=0.08]  [draw opacity=0] (10.72,-5.15) -- (0,0) -- (10.72,5.15) -- (7.12,0) -- cycle    ;
		\draw [color={rgb, 255:red, 65; green, 117; blue, 5 }  ,draw opacity=1 ]   (364,191.67) -- (494.27,163.11) ;
		\draw [shift={(434.02,176.32)}, rotate = 167.64] [fill={rgb, 255:red, 65; green, 117; blue, 5 }  ,fill opacity=1 ][line width=0.08]  [draw opacity=0] (10.72,-5.15) -- (0,0) -- (10.72,5.15) -- (7.12,0) -- cycle    ;
		\draw [color={rgb, 255:red, 65; green, 117; blue, 5 }  ,draw opacity=1 ]   (82.33,272.4) -- (223.9,211.4) ;
		\draw [shift={(157.71,239.92)}, rotate = 156.69] [fill={rgb, 255:red, 65; green, 117; blue, 5 }  ,fill opacity=1 ][line width=0.08]  [draw opacity=0] (10.72,-5.15) -- (0,0) -- (10.72,5.15) -- (7.12,0) -- cycle    ;
		\draw  [color={rgb, 255:red, 65; green, 117; blue, 5 }  ,draw opacity=1 ] (237.68,211.92) .. controls (237.68,176.39) and (266.47,147.59) .. (302,147.59) .. controls (337.53,147.59) and (366.32,176.39) .. (366.32,211.92) .. controls (366.32,247.44) and (337.53,276.24) .. (302,276.24) .. controls (266.47,276.24) and (237.68,247.44) .. (237.68,211.92) -- cycle ;
		\draw  [fill={rgb, 255:red, 0; green, 0; blue, 0 }  ,fill opacity=1 ] (347.38,212.92) .. controls (347.38,213.76) and (346.7,214.44) .. (345.86,214.44) .. controls (345.01,214.44) and (344.33,213.76) .. (344.33,212.92) .. controls (344.33,212.08) and (345.01,211.4) .. (345.86,211.4) .. controls (346.7,211.4) and (347.38,212.08) .. (347.38,212.92) -- cycle ;
		\draw [color={rgb, 255:red, 65; green, 117; blue, 5 }  ,draw opacity=1 ]   (298.33,147.67) -- (309.04,148.11) ;
		\draw [shift={(308.68,148.1)}, rotate = 182.38] [fill={rgb, 255:red, 65; green, 117; blue, 5 }  ,fill opacity=1 ][line width=0.08]  [draw opacity=0] (10.72,-5.15) -- (0,0) -- (10.72,5.15) -- (7.12,0) -- cycle    ;
		
		\draw (106,218.7) node [anchor=north west][inner sep=0.75pt]  [font=\footnotesize,xscale=0.8,yscale=0.8] [align=left] {$\displaystyle E_{j_{0}}$};
		\draw (295.9,224.39) node [anchor=north west][inner sep=0.75pt]  [font=\footnotesize,xscale=0.8,yscale=0.8] [align=left] {$\displaystyle \hat{E}_{j_{0}}$};
		\draw (339.33,216.57) node [anchor=north west][inner sep=0.75pt]  [xscale=0.8,yscale=0.8]  {$z_{0}$};
		\draw (163.5,164.9) node [anchor=north west][inner sep=0.75pt]  [font=\small,xscale=0.8,yscale=0.8]  {$\Sigma _{2}$};
		\draw (131,250.4) node [anchor=north west][inner sep=0.75pt]  [font=\small,xscale=0.8,yscale=0.8]  {$\Sigma _{2}^{*}$};
		\draw (383.17,163.57) node [anchor=north west][inner sep=0.75pt]  [font=\small,xscale=0.8,yscale=0.8]  {$\Sigma _{1}$};
		\draw (382.5,249.4) node [anchor=north west][inner sep=0.75pt]  [font=\small,xscale=0.8,yscale=0.8]  {$\Sigma _{1}^{*}$};
		\draw (82.23,186.68) node [anchor=north west][inner sep=0.75pt]  [font=\small,xscale=0.8,yscale=0.8]  {$\Omega _{2}$};
		\draw (85,234.4) node [anchor=north west][inner sep=0.75pt]  [font=\small,xscale=0.8,yscale=0.8]  {$\Omega _{2}^{*}$};
		\draw (455,197.4) node [anchor=north west][inner sep=0.75pt]  [font=\small,xscale=0.8,yscale=0.8]  {$\Omega _{1}$};
		\draw (455,215.82) node [anchor=north west][inner sep=0.75pt]  [font=\small,xscale=0.8,yscale=0.8]  {$\Omega _{1}^{*}$};
		\draw (214,219.9) node [anchor=north west][inner sep=0.75pt]  [font=\small,xscale=0.8,yscale=0.8]  {$z_{j_0}$};
		\draw (283.33,122.4) node [anchor=north west][inner sep=0.75pt]  [font=\small,xscale=0.8,yscale=0.8]  {$\partial U$};
	\end{tikzpicture}
	\caption{\footnotesize  The  jump contour $\Sigma^{(E)}$ of RH problem \ref{RHP E I} for $E$ (the green curves).}
	\label{Fig E I}
\end{figure}
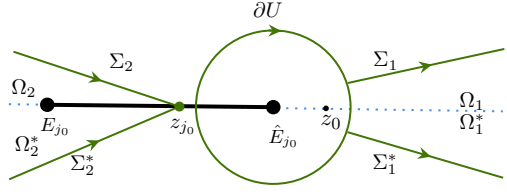
A direct calculation  shows that
\begin{equation}\label{E(z):BCrep}
	\parallel J^{(E)}(z)-I\parallel_{L^p}=\left\{\begin{array}{llll}
		\mathcal{O}(\exp\left\{-ct^{\epsilon/2}\right\}),  & z\in  \Sigma^{(E)}\setminus  \partial U,\\[6pt]
		\mathcal{O}(t^{-\kappa_p}),   & z\in\partial U,
	\end{array}\right.
\end{equation}
for some positive $c$ with $\kappa_\infty=\epsilon-1/3$ and $\kappa_2=\epsilon/2$. It then follows from the small norm RH problem theory \cite{PX3} that there exists a unique solution to RH problem \ref{RHP E I} for large positive $t$. Moreover, according to \cite{BC}, we have
\begin{equation}\label{expression:E(z)}
	E(z)=I+\frac{1}{2\pi i}\int_{\Sigma^{(E)}}\dfrac{\varpi(\lambda  ) (J^{(E)}(\lambda  )-I)}{\lambda  -z}\dif \lambda  ,
\end{equation}
where $\varpi\in I+ L^2(\Sigma^{(E)})$ is the unique solution of the Fredholm-type equation
\begin{align}
	\varpi=I+\mathcal{C}_E\varpi.\label{equ: varpi}
\end{align}
Here, $\mathcal{C}_E$: $L^2(\Sigma^{(E)})\to L^2(\Sigma^{(E)})$ is an integral operator defined by $\mathcal{C}_E(f)(z)=\mathcal{C}_-\left( f(J^{(E)}(z) -I)\right) $
with $\mathcal{C}_-$ being the Cauchy projection operator on $\Sigma^{(E)}$ defined in \eqref{def:opCpm}. Thus, 
\begin{align}\label{eq:estCE}
	\parallel \mathcal{C}_E\parallel\leqslant\parallel \mathcal{C}_-\parallel \parallel J^{(E)}(z)-I\parallel_{L^\infty }\lesssim t^{-(\epsilon-1/3)},
\end{align}
which implies that  $1-\mathcal{C}_E$ is invertible  for   sufficiently large $t$ (recall that  $\epsilon\in (1/3, 2/3)$) and $\varpi$  exists  uniquely with
\begin{align*}
	\varpi=I+(1-\mathcal{C}_E)^{-1}(\mathcal{C}_EI).
\end{align*}
Moreover,
\begin{align}
	\|\varpi-I\|_{L^2}\lesssim \|\mathcal{C}_EI\|_{L^2}\lesssim\|J^{(E)}(z)-I\|_{L^2}\lesssim  t^{-\epsilon/2}.\label{normrho}
\end{align}

By \eqref{expression:E(z)}, it is readily seen that
\begin{align}\label{eq:Eneari}
	E(z)=I+\frac{E_1}{z}+\mathcal{O}(z^{-2}), \qquad z\to \infty,
\end{align}
where
\begin{align}
	E_1&=-\frac{1}{2\pi i}\int_{\Sigma^{(E)}}\varpi(\lambda  ) (J^{(E)}(\lambda  )-I)\dif \lambda. \label{def:E1}
\end{align}
We conclude this section with the large $t$ asymptotics  of $E_1$.
\begin{proposition}
	\label{asyE}
	As $t\to +\infty$, we have
	\begin{align}\label{eq:asyE1}
		E_1&=-\frac{t^{-1/3}}{|\theta^{(3,\hat{j_0})}(\hat{\xi}_{j_0})|^{2/3}}H_1(\hat{E}_{j_0})\begin{pmatrix} 
			0 & ia\\
			0& 0
		\end{pmatrix} H_1(\hat{E}_{j_0})^{-1}+\mathcal{O}(t^{-\epsilon}),
	\end{align}
    where $\theta^{(3,\hat{j_0})}$, $H_1$ and $a$ are defined in \eqref{theta hj03}, \eqref{def H1 I} and \eqref{def:a}, respectively. 
\end{proposition}
\begin{proof}
It follows from  the estimates in \eqref{E(z):BCrep} and \eqref{normrho} that
\begin{align*}
    |\int_{\partial U}(\varpi-I)(J^{(E)}(\lambda  )-I)\dif \lambda|\lesssim \|\varpi-I\|_{L^2}\|J^{(E)}-I\|_{L^2}\lesssim  t^{-\epsilon}.
\end{align*}
This, together with \eqref{def zeta1}, \eqref{d zeta1}, \eqref{def:E1} and Proposition \ref{Pro asy mloc}, implies that
\begin{align*}
		E_1&=-\frac{1}{2\pi i}\int_{\partial U}(J^{(E)}(\lambda  )-I)\dif \lambda+\mathcal{O}(t^{-\epsilon})\\
		&=\frac{1}{2\pi i}\int_{\partial U}\frac{t^{1/3}}{\zeta(\lambda)}H_1(\lambda) \begin{pmatrix} 
			0 & ia\\
			0& 0
		\end{pmatrix} H_1(\lambda)^{-1}\dif \lambda+\mathcal{O}(t^{-\epsilon})\\
		&=-\frac{t^{-1/3}}{|\theta^{(3,\hat{j_0})}(\hat{\xi}_{j_0})|^{2/3}}H_1(\hat{E}_{j_0}) \begin{pmatrix} 
			0 & ia\\
			0& 0
		\end{pmatrix}  H_1(\hat{E}_{j_0})^{-1}+\mathcal{O}(t^{-\epsilon}),
\end{align*}
where we have made use of  the residue theorem in the third equality.
\end{proof}

\section{Asymptotic analysis in transition  region II}\label{Sec 4}
In this section, we perform asymptotic analysis of RH problem \ref{RHP2} for $N$ to derive asymptotics of $q$ in transition region II. By Definition \ref{def:regions}, it is assumed that $0<(\xi-\xi_{j_0})t^{2/3}\leq C$ for some fixed $j_0\in\{0,\ldots,n\}$ throughout this section, since the analysis in the other half region is similar. Note that $\xi_{j_0}<\xi<\hat{\xi}_{j_0-1}$ for large $t$ in this case, it then follows from discussions in Section \ref{subsec2.3} that the saddle point $z_0\in[\hat{E}_{j_0-1},E_{j_0})$. As the analysis is similar to that given in  Section \ref{Sec 3}, we only give a sketch of the procedure in this case.

In the next two sections, we adopt the same notations (such as $\delta$, $M^{(1)}$, $N$, $G$, ...) as those used before, they should be understood in different contexts. We believe this will not cause any confusion.

\vspace{2mm}
 \paragraph{\textbf{First transformation}}
To define the first transformation, we introduce, similar to \eqref{def delta 1}, the auxiliary function
	\begin{align}\label{def delta II}
	&\delta(z):=\frac{w(z)}{2\pi i}\left[\sum_{j=1}^n\delta_j\int_{E_j}^{\hat{E}_j}\frac{i\dif s}{w_+(s)(s-z)}+\int_{(\hat{E}_{j_0},+\infty)\setminus(\cup_{j=j_0}^{n}(E_j,\hat{E}_j))}\frac{\log (1-r_3(s)r_4(s))\dif s}{w(s)(s-z)} \right],
\end{align}
where the logarithm takes principal branch, the functions $w$, $r_3$ and $r_4$ are defined in \eqref{def:w} and \eqref{def r}, respectively. In \eqref{def delta II}, the real constants $\delta_j$, $j=1,\ldots,n$, are determined through the linear systems 
\begin{align}
	\int_{(\hat{E}_{j_0},+\infty)\setminus(\cup_{j=j_0}^{n}(E_j,\hat{E}_j))}\frac{\log (1-r_3(s)r_4(s))s^k\dif s}{w(s)}+\sum_{j=1}^n\delta_j\int_{E_j}^{\hat{E}_j}\frac{is^k\dif s}{w_+(s)}=0,\quad k=0,\ldots,n-1. \label{def deltaj II}
\end{align}
The following proposition is an analogue of Proposition \ref{Pro delta1 } in this case, we omit the proof here. 
\begin{proposition}
The function $\delta$ defined in \eqref{def delta II} admits the following properties.
\begin{itemize}
	\item [$\mathrm{(a)}$ ]$\delta$ is analytic on $\mathbb{C}\setminus ((z_0,+\infty)\cup (\cup_{j=0}^n[E_j,\hat{E}_j]))$ and satisfies
	the symmetry relation $\delta(z)=-\overline{\delta(\bar{z})}$. Moreover, as $z\to\infty$, we have $$\delta(z)=\delta(\infty)+\dfrac{\delta^{(1)}}{z}+\mathcal{O}(z^{-2}),$$ where
	\begin{align}\label{def delta infII}
		\delta(\infty)&=-\frac{1}{2\pi i}\left[\int_{(\hat{E}_{j_0},+\infty)\setminus(\cup_{j=j_0}^{n}(E_j,\hat{E}_j))}\frac{\log (1-r_3(s)r_4(s))s^n\dif s}{w(s)}+\sum_{j=1}^n\delta_j\int_{E_j}^{\hat{E}_j}\frac{is^n\dif s}{w_+(s)} \right], \\
		\delta^{(1)}&=-\delta(\infty)\sum_{j=0}^{n}(E_j+\hat{E}_j) \nonumber \\
		&\quad -\frac{1}{2\pi i}\left[\int_{(\hat{E}_{j_0},+\infty)\setminus(\cup_{j=j_0}^{n}(E_j,\hat{E}_j))}\frac{\log (1-r_3(s)r_4(s))s^{n+1}\dif s}{w(s)}+\sum_{j=1}^n\delta_j\int_{E_j}^{\hat{E}_j}\frac{is^{n+1}\dif s}{w_+(s)} \right]. \label{def delta1 II}
	\end{align}\normalsize
	\item [$\mathrm{(b)}$] $\delta$  satisfies the jump conditions
	\begin{align*}
		&\delta_+(z)=\delta_-(z)+\log (1-r_3(z)r_4(z)),&& \qquad z\in (\hat{E}_{j_0},+\infty)\setminus(\cup_{j=j_0}^{n}[E_j,\hat{E}_j]),\\
		&\delta_+(z)+\delta_-(z)=i\delta_j, && 
        \qquad z\in(E_j,\hat{E}_j), ~ j=1,\ldots,n.
	\end{align*}
	\item [$\mathrm{(c)}$ ] As $z\to p\in\{E_j \}_{j=j_0+1}^{n}\cup\{\hat{E}_j \}_{j=j_0}^{n}$ from $\mathbb{C}^+$, we have 
	\begin{align*}
		e^{\delta(z)}=\mathcal{O}((z-p)^{-1/2})
	\end{align*}
and
	\begin{align*}
		\delta(z)=\frac{i}{2}\delta_{j_0}+\mathcal{O}((z-\hat{E}_{j_0})^{1/2}), \qquad 
        \textrm{ $z\to E_{j_0}$ from $\mathbb{C}\setminus (-\infty, E_{j_0})$. }
	\end{align*}
\end{itemize}
\end{proposition}
We next set 
\begin{align}
	&\Omega_1:=\Omega_1(\xi)=\{z\in\mathbb{C}:\ \pi-\varphi_0\leq\arg(z-z_0)\leq\pi\},\ \Sigma_1:=\Sigma_1(\xi)=z_0+e^{i(\pi-\varphi_0)}\mathbb{R}^+,	\label{def:Omega1 II}\\
	&\Omega_2:=\Omega_2(\xi)=\{z\in\mathbb{C}:\ 0\leq\arg(z-z_{j_0})\leq\varphi_0\},\ \Sigma_2:=\Sigma_2(\xi)=z_{j_0}+e^{i\varphi_0}\mathbb{R}^+, \label{def:Omega2 II}
\end{align}
where $\varphi_0$ is sufficiently small such that 
\begin{equation}\label{eq:theta II}
    \Im \theta(z) \left\{
   \begin{array}{ll}
     <0, & \hbox{$z\in\Sigma_1\setminus\{z_0\}$,} \\
     >0, & \hbox{$z\in\Sigma_2\setminus\{z_{j_0}\}$,}
   \end{array}
 \right.
\end{equation}
see Figure \ref{fig II 1} for an illustration. 
\begin{figure}
	\tikzset{every picture/.style={line width=0.75pt}} 
	
	\begin{tikzpicture}[x=0.75pt,y=0.75pt,yscale=-0.6,xscale=0.6]
		
		\draw [color={rgb, 255:red, 74; green, 144; blue, 226 }  ,draw opacity=1 ]   (82,209.42) -- (493,215.42) ;
		\draw [line width=1.5]    (272,212.42) -- (460,214.42) ;
		\draw [shift={(460,214.42)}, rotate = 0.61] [color={rgb, 255:red, 0; green, 0; blue, 0 }  ][fill={rgb, 255:red, 0; green, 0; blue, 0 }  ][line width=1.5]      (0, 0) circle [x radius= 4.36, y radius= 4.36]   ;
		\draw [shift={(272,212.42)}, rotate = 0.61] [color={rgb, 255:red, 0; green, 0; blue, 0 }  ][fill={rgb, 255:red, 0; green, 0; blue, 0 }  ][line width=1.5]      (0, 0) circle [x radius= 4.36, y radius= 4.36]   ;
		\draw [color={rgb, 255:red, 65; green, 117; blue, 5 }  ,draw opacity=1 ]   (85,173.42) -- (223.9,211.4) ;
		\draw [shift={(223.9,211.4)}, rotate = 15.29] [color={rgb, 255:red, 65; green, 117; blue, 5 }  ,draw opacity=1 ][fill={rgb, 255:red, 65; green, 117; blue, 5 }  ,fill opacity=1 ][line width=0.75]      (0, 0) circle [x radius= 3.35, y radius= 3.35]   ;
		\draw [shift={(159.27,193.73)}, rotate = 195.29] [fill={rgb, 255:red, 65; green, 117; blue, 5 }  ,fill opacity=1 ][line width=0.08]  [draw opacity=0] (10.72,-5.15) -- (0,0) -- (10.72,5.15) -- (7.12,0) -- cycle    ;
		\draw [color={rgb, 255:red, 65; green, 117; blue, 5 }  ,draw opacity=1 ]   (320,213.5) -- (463,263.42) ;
		\draw [shift={(396.22,240.11)}, rotate = 199.24] [fill={rgb, 255:red, 65; green, 117; blue, 5 }  ,fill opacity=1 ][line width=0.08]  [draw opacity=0] (10.72,-5.15) -- (0,0) -- (10.72,5.15) -- (7.12,0) -- cycle    ;
		\draw [color={rgb, 255:red, 65; green, 117; blue, 5 }  ,draw opacity=1 ]   (320,213.5) -- (468,168.42) ;
		\draw [shift={(398.78,189.5)}, rotate = 163.06] [fill={rgb, 255:red, 65; green, 117; blue, 5 }  ,fill opacity=1 ][line width=0.08]  [draw opacity=0] (10.72,-5.15) -- (0,0) -- (10.72,5.15) -- (7.12,0) -- cycle    ;
		\draw [shift={(320,213.5)}, rotate = 343.06] [color={rgb, 255:red, 65; green, 117; blue, 5 }  ,draw opacity=1 ][fill={rgb, 255:red, 65; green, 117; blue, 5 }  ,fill opacity=1 ][line width=0.75]      (0, 0) circle [x radius= 3.35, y radius= 3.35]   ;
		\draw [color={rgb, 255:red, 65; green, 117; blue, 5 }  ,draw opacity=1 ]   (82,250.42) -- (223.9,211.4) ;
		\draw [shift={(157.77,229.58)}, rotate = 164.63] [fill={rgb, 255:red, 65; green, 117; blue, 5 }  ,fill opacity=1 ][line width=0.08]  [draw opacity=0] (10.72,-5.15) -- (0,0) -- (10.72,5.15) -- (7.12,0) -- cycle    ;
		
		\draw (264,227.2) node [anchor=north west][inner sep=0.75pt]  [font=\footnotesize,xscale=0.8,yscale=0.8] [align=left] {$\displaystyle E_{j_{0}}$};
		\draw (453.4,223.39) node [anchor=north west][inner sep=0.75pt]  [font=\footnotesize,xscale=0.8,yscale=0.8] [align=left] {$\displaystyle \hat{E}_{j_{0}}$};
		\draw (228,188.9) node [anchor=north west][inner sep=0.75pt]  [xscale=0.8,yscale=0.8]  {$z_{0}$};
		\draw (363,171.9) node [anchor=north west][inner sep=0.75pt]  [font=\small,xscale=0.8,yscale=0.8]  {$\Sigma _{2}$};
		\draw (374,244.4) node [anchor=north west][inner sep=0.75pt]  [font=\small,xscale=0.8,yscale=0.8]  {$\Sigma _{2}^{*}$};
		\draw (149.5,166.4) node [anchor=north west][inner sep=0.75pt]  [font=\small,xscale=0.8,yscale=0.8]  {$\Sigma _{1}$};
		\draw (163.5,235.9) node [anchor=north west][inner sep=0.75pt]  [font=\small,xscale=0.8,yscale=0.8]  {$\Sigma _{1}^{*}$};
		\draw (410.73,193.18) node [anchor=north west][inner sep=0.75pt]  [font=\small,xscale=0.8,yscale=0.8]  {$\Omega _{2}$};
		\draw (409.5,219.4) node [anchor=north west][inner sep=0.75pt]  [font=\small,xscale=0.8,yscale=0.8]  {$\Omega _{2}^{*}$};
		\draw (81.5,189.4) node [anchor=north west][inner sep=0.75pt]  [font=\small,xscale=0.8,yscale=0.8]  {$\Omega _{1}$};
		\draw (79,213.82) node [anchor=north west][inner sep=0.75pt]  [font=\small,xscale=0.8,yscale=0.8]  {$\Omega _{1}^{*}$};
		\draw (311,222.4) node [anchor=north west][inner sep=0.75pt]  [font=\small,xscale=0.8,yscale=0.8]  {$z_{j_0}$};	
	\end{tikzpicture}
	\caption{\footnotesize   The contours $\Sigma_j$ and the domains $\Omega_j$, $j=1,2$,  in \eqref{def:Omega1 II} and \eqref{def:Omega2 II}.}
\label{fig II 1}
\end{figure}
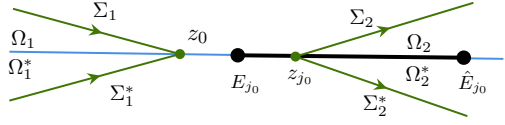
The first transformation is defined by
\begin{align}\label{trans 1 II}
 	M^{(1)}(z):=M^{(1)}(z;\xi,t)=e^{\delta(\infty)\sigma_3}N(z)G(z)e^{-\delta(z)\sigma_3},
 \end{align}
where $\delta$ is given in \eqref{def delta II} and 
\begin{equation}\label{def:G II}
 G(z): = G(z;\xi) =\left\{
   \begin{array}{ll}
     \begin{pmatrix}
 			1 & -r_4(z)e^{2it\theta(z)}\\
		0& 1
 	\end{pmatrix}, & \hbox{$z\in\Omega_1$,} \\[8pt]
     \begin{pmatrix}
 		1 & 0 \\
		-r_3(z)e^{-2it\theta(z)}  & 1
 	\end{pmatrix}, & \hbox{$z\in\Omega_1^*$,}
    \\[10pt]
   \begin{pmatrix}
 	1 &	0\\
\frac{r_3(z)e^{-2it\theta(z)}}{1-r_3(z)r_4(z)} & 1
 	\end{pmatrix}, & \hbox{$z\in\Omega_2$,} \\[12pt]
     \begin{pmatrix}
    1 &	\frac{ r_4(z)e^{2it\theta(z)}}{1-r_3(z)r_4(z)}
        \\
	0 & 1
 	\end{pmatrix}, & \hbox{$z\in\Omega_2^*$,}
    \\
     I, & \hbox{elsewhere,}
   \end{array}
 \right.
\end{equation}
In view of the factorizations of the jump matrix $\tilde{J}$ given in \eqref{open jump tJ1}--\eqref{open jump tJ4}, we see from RH problem \ref{RHP2} for $N$ that 
$M^{(1)}(z)$ is analytic in $\mathbb{C}\setminus\Sigma^{(1)}$, where  \begin{equation}\label{def:Sigma1caseIII}
		    \Sigma^{(1)}:=\Sigma^{(1)}(\xi)= \cup_{j=0}^n[E_j,\hat{E}_j]\cup\Sigma_1\cup\Sigma_1^*\cup\Sigma_2\cup\Sigma_2^*\cup [z_0,E_{j_0}],
		\end{equation}
and satisfies the jump condition $M^{(1)}_{+}(z)=M^{(1)}_{-}(z)J^{(1)}(z)$ for $z\in\Sigma^{(1)}$, where
		\begin{align}\label{def:J1 II}
		J^{(1)}(z)=	\left\{ \begin{array}{ll}
				e^{i(f_0x+g_0t-(B_j^fx+B_j^gt+\phi_j-\delta_j)/2)\hat{\sigma}_3} \begin{pmatrix} 
				0 & -i\\
				-i& 0\\
			\end{pmatrix} , & z\in (E_j,\hat{E}_j),\ j=0,\ldots,n,\\[12pt]
					\begin{pmatrix} 
					1 &  r_4(z)e^{2it\theta(z)+2\delta(z)}\\
					0 & 1	
				\end{pmatrix} ,&  z\in\Sigma_1,\\[12pt]
			 \begin{pmatrix}
					1 & 0\\
					-r_3(z)e^{-2it\theta(z)-2\delta(z)} & 1
				\end{pmatrix} , &z\in\Sigma_1^*,\\[12pt]
				 \begin{pmatrix} 
			1 &	0\\
			\frac{-r_3(z)e^{-2it\theta(z)-2\delta(z)}}{1-r_3(z)r_4(z)} & 1
			\end{pmatrix} ,&z\in\Sigma_2,\\[12pt]
			 \begin{pmatrix} 
				1 &	\frac{ r_4(z)e^{2it\theta(z)+2\delta(z)}}{1-r_3(z)r_4(z)}\\
			0 & 1
			\end{pmatrix} ,&z\in\Sigma_2^*,\\[12pt]
		 \begin{pmatrix} 
			1 & r_4(z)e^{2it\theta(z)+2\delta(z)}\\
			-r_3(z)e^{-2it\theta(z)-2\delta(z)} & 1-r_3(z)r_4(z)
		\end{pmatrix} , & z\in (z_0,E_{j_0}).
				\end{array}\right.
		\end{align}

\vspace{2mm} 
\paragraph{\textbf{Global and local parametrices}}
To approximate $M^{(1)}$ for large $t$ in global and local manners, we need to, as in  \eqref{construct M1}, construct global and local parametrices $M^{(glo)}$ and $M^{(loc)}$ such that
\begin{align}\label{def:E II}
	 M^{(1)}(z)= \left\{ \begin{array}{ll}
		E(z)M^{(glo)}(z), & z\in \mathbb{C}\setminus U ,\\
		E(z)M^{(loc)}(z),&  z\in U,
	\end{array}\right.
\end{align}
where
\begin{align}
	U:=\{z:\ |z-E_{j_0}|\leq c_0\}\label{def Uxi II}
\end{align}
with
\begin{align}
	c_0:=\min\left\{\frac{1}{2}(z_{j_0}-E_{j_0}),\ \frac{1}{2}(E_{j_0+1}-\hat{E}_{j_0}),\ 2(z_0-\hat{E}_{j_0})t^\epsilon\right\},\qquad\epsilon\in(1/3,2/3)\label{def c0 II}
\end{align}
is a shrinking disk around $E_{j_0}$.

By \eqref{eq:theta II} and \eqref{def:J1 II}, we have that the global parametrix  $M^{(glo)}$  is also given by \eqref{def Mmod}--\eqref{eq:Nglob22} but with the parameters $\{\delta_j\}_{j=1}^n$ therein replaced by those determined through \eqref{def deltaj II}.

The construction of local parametrix $M^{(loc)}$ is similar to that given in Section \ref{subsec loc I}, which involves a special case of the Painlev\'e XXXIV parametrix $M^{(P_{34})}$. To proceed, we observe that
\begin{align}\label{eq:thetaexp2}
	\theta(z)=\theta^{(0,j_0)}+(\xi-\xi_{j_0})\theta^{(1,j_0)}(E_{j_0}-z)^{1/2}+\frac{2}{3}\theta^{(3,j_0)}(E_{j_0}-z)^{3/2}+\mathcal{O}((E_{j_0}-z)^{5/2}),
\end{align}
as $z\to E_{j_0}$ from $\mathbb{C}\setminus(E_{j_0},+\infty)$, 
where
\begin{align}
&\theta^{(0,j_0)}=\theta(E_{j_0})=f_0\xi+g_0-\frac{1}{2}(B_{j_0}^f\xi+B_{j_0}^g), \qquad\theta^{(1,j_0)}=\frac{-2\prod_{j=0}^n(E_{j_0}-z_j^f)}{w^{(j_0)}(E_{j_0})},\label{theta j01}
    \\
	&\theta^{(3,j_0)}=\theta^{(3,j_0)}(\xi)=\frac{-1}{w^{(j_0)}(E_{j_0})^2}\left[ \partial_z F(E_{j_0};\xi)w^{(j_0)}(E_{j_0})-F(E_{j_0};\xi)(w^{(j_0)})'(E_{j_0})\right].\label{theta j03}
\end{align}
Here, $f_0$, $g_0$ and $B_{j_0}^f$, $B_{j_0}^g$ are given in \eqref{def f0g0} and \eqref{def:Bjfg}, the functions $F$  and $w^{(j_0)}$ are defined in \eqref{def F} and \eqref{def w}, respectively. Note that $\theta^{(3,j_0)}(\xi_{j_0})>0$, which implies $\theta^{(3,j_0)}(\xi)> 0$ in a neighborhood of $\xi_{j_0}$.  
Let
\begin{align}
	\zeta(z):=\big(\frac{3it}{2}(\theta(E_{j_0})+(\xi-\xi_{j_0})\theta^{(1,j_0)}(E_{j_0}-z)^{1/2}-
 	\theta(z))\big)^{2/3}.
\end{align}
Since \eqref{eq:thetaexp2},  $\zeta(z)$ is an analytic function in $U$ with 
\begin{align}
\zeta(E_{j_0})=0,\qquad\zeta'(E_{j_0})=|\theta^{(3,j_0)}(\xi_{j_0})|^{2/3}t^{2/3}>0.
\end{align}

For $z\in U$, the local parametrix $M^{(loc)}$ in transition region II is then given by
\begin{align*}
	M^{(loc)}(z)=\tilde{H}_1(z)t^{\sigma_3/6}\begin{pmatrix}
		1 & 0\\
		ia(s) & 1
	\end{pmatrix}M^{(P_{34})}(\zeta(z);s,-1/4,0)e^{\hat{\theta}(\zeta)\sigma_3}G_1(\zeta(z))G_2(\zeta(z)),
 \end{align*}
where
\begin{align}
	\tilde{H}_1(z)=M^{(glo)}(z)G_1(\zeta(z))^{-1}\frac{1}{\sqrt{2}}(I-i\sigma_1)((z-E_{j_0})|\theta^{(3,j_0)}(\xi)|^{2/3})^{\sigma_3/4}, \label{def TH1 II}
\end{align}
$a(s)$ is defined in \eqref{def:a} with
\begin{align}
    s:=-\frac{\theta^{(1,j_0)}}{(\theta^{(3,j_0)}(\xi_{j_0}))^{1/3}}(\xi-\xi_{j_0})t^{2/3},
\end{align}
\begin{align*}
\small	G_1(\zeta)=	\left\{ \begin{array}{ll}
		 e^{(-\pi i/4-i(f_0x+g_0t)+i(B_{j_0}^fx+B_{j_0}^gt+\phi_{j_0}-\delta_{j_0})/2)\sigma_3},&  \zeta\in\mathbb{C}^+,\\
		\begin{pmatrix} 
			0 & -e^{\pi i/4+i(f_0x+g_0t)-i(B_{j_0}^fx+B_{j_0}^gt+\phi_{j_0}-\delta_{j_0})/2} \\
			e^{-\pi i/4-i(f_0x+g_0t)+i(B_{j_0}^fx+B_{j_0}^gt+\phi_{j_0}-\delta_{j_0})/2} & 0
		\end{pmatrix} , &\zeta\in\mathbb{C}^-,
	\end{array}\right.
\end{align*}
and
\begin{align}\label{def:G2 II}
G_2(\zeta)=e^{it\theta(\hat{E}_{j_0})\hat{\sigma}_3}	\left\{ \begin{array}{ll}
		 \begin{pmatrix} 
			1 & -r_4(E_{j_0})e^{2\hat{\theta}(\zeta)+i\delta_{j_0}} \\
			0 & 1
		\end{pmatrix} ,&  \zeta\in\Omega_1^{(loc)},\\
		 \begin{pmatrix} 
			1 & 0\\
			-r_3(E_{j_0})e^{2\hat{\theta}(\zeta)-i\delta_{j_0}} & 1	
		\end{pmatrix} , &\zeta\in\Omega_1^{(loc)*},
	\end{array}\right.
\end{align}
In \eqref{def:G2 II}, the region $\Omega_1^{(loc)}$ is the same as in the case of transition region I, which is depicted in the right picture of Figure \ref{Fig move loc}.
\vspace{2mm}

\paragraph{\textbf{The small norm RH problem}}
Once we have completed the constructions of global and local parametrices, one can show that, by adopting the arguments in  Section \ref{subsec E}, the error function $E$ in \eqref{def:E II} exists uniquely for large $t$. Moreover, we have 
\begin{align}
	E(z)=I+\frac{E_1}{z}+\mathcal{O}(z^{-2}), \qquad z\to \infty,
\end{align}
where
\begin{align*}
E_1&=\frac{t^{-1/3}}{|\theta^{(3,j_0)}(\xi_{j_0})|^{2/3}}\tilde{H}_1(E_{j_0})
\begin{pmatrix} 
			0 & ia(s)\\
			0& 0
		\end{pmatrix}  \tilde{H}_1(E_{j_0})^{-1}+\mathcal{O}(t^{-\epsilon}), \quad t\to +\infty. 
\end{align*}

\section{Asymptotic analysis in Zakharov-Manakov region III and  fast decaying region IV}\label{Sec 5}

In this section, we perform asymptotic analysis of RH problem \ref{RHP1} for $M$ when $\xi$ belongs to Zakharov-Manakov region (i.e, $\xi \in(-\infty,\hat{\xi}_n)\cup_{j=1}^n(\xi_{j},\hat{\xi}_{j-1})\cup(\xi_0,+\infty)$) and fast decaying region (i.e., $\xi\in\cup_{j=0}^n(\hat{\xi}_j,\xi_{j})$). As discussed in Section \ref{subsec2.3}, the saddle point $z_0\in \mathbb{R}\setminus (\cup_{j=0}^n[E_j,\hat{E}_j])$ for Zakharov-Manakov region,
while  $z_0\in \cup_{j=0}^n(E_j,\hat{E}_j)$  for fast decaying region. Since the analysis is similar to those given in \cite{PX2,PX3,DM,MK,Alice}, we again only give a sketch of the analysis.


\vspace{2mm}
\paragraph{\textbf{First transformation}}
To define the first transformation, we introduce, similar to \eqref{def delta 1}, the auxiliary function
\begin{align}\label{def delta 1 III}
	\delta(z):=\frac{w(z)}{2\pi i}\left[\sum_{j=1}^n\delta_j\int_{E_j}^{\hat{E}_j}\frac{i\dif s}{w_+(s)(s-z)} -\int_{(-\infty,z_0)\setminus(\cup_{j=0}^{n}(E_j,\hat{E}_j))}\frac{\log (1-r_1(s)r_2(s))\dif s}{w(s)(s-z)}\right] ,
\end{align}
where the logarithm takes principal branch, the functions $w$, $r_1$ and $r_2$ are defined in \eqref{def:w} and \eqref{def r}, respectively. In \eqref{def delta 1 III}, the constants $\delta_j$, $j=1,\ldots,n$, are determined through the linear systems 
\begin{align}\label{def deltaj III}
	\int_{(-\infty,z_0)\setminus(\cup_{j=0}^{n}(E_j,\hat{E}_j))}\frac{\log (1-r_1(s)r_2(s))s^k\dif s}{w(s)}+\sum_{j=1}^n\delta_j\int_{E_j}^{\hat{E}_j}\frac{is^k\dif s}{w_+(s)} =0, \quad k=0,\ldots,n-1.
\end{align}
It is then readily seen that, as $z\to\infty$, $$\delta(z)=\delta(\infty)+\dfrac{\delta^{(1)}}{z}+\mathcal{O}(z^{-2}),$$ where
		\begin{align}\label{def delta inf III}
			\delta(\infty)&=\frac{1}{2\pi i}\left[\int_{(-\infty,z_0)\setminus(\cup_{j=0}^{n}(E_j,\hat{E}_j))}\frac{\log (1-r_1(s)r_2(s))s^n\dif s}{w(s)}-\sum_{j=1}^n\delta_j\int_{E_j}^{\hat{E}_j}\frac{is^n\dif s}{w_+(s)} \right]
            \end{align}
            and
            \begin{align}
			\delta^{(1)}&= -\delta(\infty)\sum_{j=0}^{n}(E_j+\hat{E}_j) \nonumber\\
			&\quad +\frac{1}{2\pi i}\left[\int_{(-\infty,z_0)\setminus(\cup_{n}^{j_0}(E_j,\hat{E}_j))}\frac{\log (1-r_1(s)r_2(s))s^{n+1}\dif s}{w(s)}-\sum_{j=1}^n\delta_j\int_{E_j}^{\hat{E}_j}\frac{is^{n+1}\dif s}{w_+(s)} \right].\label{def delta1 III}
		\end{align}
        
On account of the factorizations of the jump matrix $J$ given in \eqref{open jump J1}--\eqref{open jump J4}, we also need the matrix-valued function $G(z):=G(z;\xi)$ defined in \eqref{def:G} but with the regions $\Omega_1$ and $\Omega_2$ therein replaced by
\begin{align}
	&\Omega_1:=\Omega_1(\xi)=\{z\in\mathbb{C}:\ 0\leq\arg(z-z_0)\leq\varphi_0\},\ \Sigma_1:=\Sigma_1(\xi)=z_0+e^{i\varphi_0}\mathbb{R}^+,	\\
	&\Omega_2:=\Omega_2(\xi)=\{z\in\mathbb{C}:\ \pi-\varphi_0\leq\arg(z-z_0)\leq\pi\},\ \Sigma_2:=\Sigma_2(\xi)=z_0+e^{i(\pi-\varphi_0)}\mathbb{R}^+,
\end{align}
where $\varphi_0$ is chosen such that 
\begin{equation}\label{eq:thetainequality}
    \Im \theta(z) \left\{
   \begin{array}{ll}
     <0, & \hbox{$z\in\Sigma_1\setminus\{z_0\}$,} \\
     >0, & \hbox{$z\in\Sigma_2\setminus\{z_{0}\}$,}
   \end{array}
 \right.
\end{equation}
see Figure \ref{fig III 1} for an illustration. 
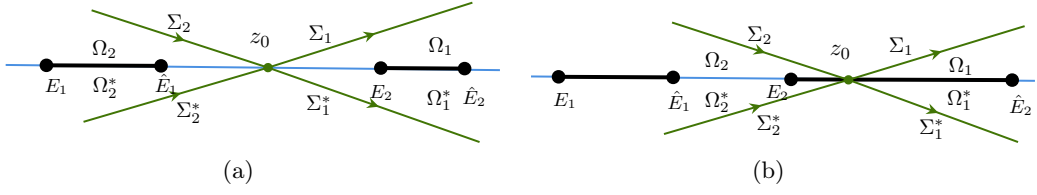
\begin{figure}
	\centering
	\subfigure[]{
		
		\tikzset{every picture/.style={line width=0.75pt}} 
		
		\begin{tikzpicture}[x=0.75pt,y=0.75pt,yscale=-0.55,xscale=0.55]
			
			\draw [color={rgb, 255:red, 74; green, 144; blue, 226 }  ,draw opacity=1 ]   (81.89,211.06) -- (530.47,215.8) ;
			\draw [line width=1.5]    (118.9,211.4) -- (222.9,211.9) ;
			\draw [shift={(222.9,211.9)}, rotate = 0.28] [color={rgb, 255:red, 0; green, 0; blue, 0 }  ][fill={rgb, 255:red, 0; green, 0; blue, 0 }  ][line width=1.5]      (0, 0) circle [x radius= 4.36, y radius= 4.36]   ;
			\draw [shift={(118.9,211.4)}, rotate = 0.28] [color={rgb, 255:red, 0; green, 0; blue, 0 }  ][fill={rgb, 255:red, 0; green, 0; blue, 0 }  ][line width=1.5]      (0, 0) circle [x radius= 4.36, y radius= 4.36]   ;
			\draw [line width=1.5]    (422.4,213.9) -- (498.4,214.39) ;
			\draw [shift={(498.4,214.39)}, rotate = 0.37] [color={rgb, 255:red, 0; green, 0; blue, 0 }  ][fill={rgb, 255:red, 0; green, 0; blue, 0 }  ][line width=1.5]      (0, 0) circle [x radius= 4.36, y radius= 4.36]   ;
			\draw [shift={(422.4,213.9)}, rotate = 0.37] [color={rgb, 255:red, 0; green, 0; blue, 0 }  ][fill={rgb, 255:red, 0; green, 0; blue, 0 }  ][line width=1.5]      (0, 0) circle [x radius= 4.36, y radius= 4.36]   ;
			\draw [color={rgb, 255:red, 65; green, 117; blue, 5 }  ,draw opacity=1 ]   (159.89,161.06) -- (320,213.5) ;
			\draw [shift={(320,213.5)}, rotate = 18.14] [color={rgb, 255:red, 65; green, 117; blue, 5 }  ,draw opacity=1 ][fill={rgb, 255:red, 65; green, 117; blue, 5 }  ,fill opacity=1 ][line width=0.75]      (0, 0) circle [x radius= 3.35, y radius= 3.35]   ;
			\draw [shift={(244.7,188.84)}, rotate = 198.14] [fill={rgb, 255:red, 65; green, 117; blue, 5 }  ,fill opacity=1 ][line width=0.08]  [draw opacity=0] (10.72,-5.15) -- (0,0) -- (10.72,5.15) -- (7.12,0) -- cycle    ;
			\draw [color={rgb, 255:red, 65; green, 117; blue, 5 }  ,draw opacity=1 ]   (320,213.5) -- (511.97,281.57) ;
			\draw [shift={(420.7,249.2)}, rotate = 199.52] [fill={rgb, 255:red, 65; green, 117; blue, 5 }  ,fill opacity=1 ][line width=0.08]  [draw opacity=0] (10.72,-5.15) -- (0,0) -- (10.72,5.15) -- (7.12,0) -- cycle    ;
			\draw [color={rgb, 255:red, 65; green, 117; blue, 5 }  ,draw opacity=1 ]   (320,213.5) -- (505.47,154.07) ;
			\draw [shift={(417.49,182.26)}, rotate = 162.23] [fill={rgb, 255:red, 65; green, 117; blue, 5 }  ,fill opacity=1 ][line width=0.08]  [draw opacity=0] (10.72,-5.15) -- (0,0) -- (10.72,5.15) -- (7.12,0) -- cycle    ;
			\draw [color={rgb, 255:red, 65; green, 117; blue, 5 }  ,draw opacity=1 ]   (152.56,263.06) -- (320,213.5) ;
			\draw [shift={(241.07,236.86)}, rotate = 163.51] [fill={rgb, 255:red, 65; green, 117; blue, 5 }  ,fill opacity=1 ][line width=0.08]  [draw opacity=0] (10.72,-5.15) -- (0,0) -- (10.72,5.15) -- (7.12,0) -- cycle    ;
			
			\draw (115.17,221.83) node [anchor=north west][inner sep=0.75pt]  [font=\footnotesize,xscale=0.8,yscale=0.8] [align=left] {$\displaystyle E_{1}$};
			\draw (215,217.17) node [anchor=north west][inner sep=0.75pt]  [font=\footnotesize,xscale=0.8,yscale=0.8] [align=left] {$\displaystyle \hat{E}_{1}$};
			\draw (410,227.2) node [anchor=north west][inner sep=0.75pt]  [font=\footnotesize,xscale=0.8,yscale=0.8] [align=left] {$\displaystyle E_{2}$};
			\draw (494.4,228.39) node [anchor=north west][inner sep=0.75pt]  [font=\footnotesize,xscale=0.8,yscale=0.8] [align=left] {$\displaystyle \hat{E}_{2}$};
			\draw (301,179.9) node [anchor=north west][inner sep=0.75pt]  [xscale=0.8,yscale=0.8]  {$z_{0}$};
			\draw (226,162.57) node [anchor=north west][inner sep=0.75pt]  [font=\small,xscale=0.8,yscale=0.8]  {$\Sigma _{2}$};
			\draw (234.33,243.4) node [anchor=north west][inner sep=0.75pt]  [font=\small,xscale=0.8,yscale=0.8]  {$\Sigma _{2}^{*}$};
			\draw (354.5,172.4) node [anchor=north west][inner sep=0.75pt]  [font=\small,xscale=0.8,yscale=0.8]  {$\Sigma _{1}$};
			\draw (352.5,236.9) node [anchor=north west][inner sep=0.75pt]  [font=\small,xscale=0.8,yscale=0.8]  {$\Sigma _{1}^{*}$};
			\draw (463.5,183.4) node [anchor=north west][inner sep=0.75pt]  [font=\small,xscale=0.8,yscale=0.8]  {$\Omega _{1}$};
			\draw (462,231.9) node [anchor=north west][inner sep=0.75pt]  [font=\small,xscale=0.8,yscale=0.8]  {$\Omega _{1}^{*}$};
			\draw (158.17,186.73) node [anchor=north west][inner sep=0.75pt]  [font=\small,xscale=0.8,yscale=0.8]  {$\Omega _{2}$};
			\draw (158.67,220.57) node [anchor=north west][inner sep=0.75pt]  [font=\small,xscale=0.8,yscale=0.8]  {$\Omega _{2}^{*}$};
					
	\end{tikzpicture}}
\subfigure[]{
	\tikzset{every picture/.style={line width=0.75pt}} 
	
	\begin{tikzpicture}[x=0.75pt,y=0.75pt,yscale=-0.55,xscale=0.55]
		
		\draw [color={rgb, 255:red, 74; green, 144; blue, 226 }  ,draw opacity=1 ]   (31.89,210.39) -- (489.23,214.39) ;
		\draw [line width=1.5]    (56.9,210.73) -- (160.9,211.23) ;
		\draw [shift={(160.9,211.23)}, rotate = 0.28] [color={rgb, 255:red, 0; green, 0; blue, 0 }  ][fill={rgb, 255:red, 0; green, 0; blue, 0 }  ][line width=1.5]      (0, 0) circle [x radius= 4.36, y radius= 4.36]   ;
		\draw [shift={(56.9,210.73)}, rotate = 0.28] [color={rgb, 255:red, 0; green, 0; blue, 0 }  ][fill={rgb, 255:red, 0; green, 0; blue, 0 }  ][line width=1.5]      (0, 0) circle [x radius= 4.36, y radius= 4.36]   ;
		\draw [line width=1.5]    (267.89,213.06) -- (468.56,213.73) ;
		\draw [shift={(468.56,213.73)}, rotate = 0.19] [color={rgb, 255:red, 0; green, 0; blue, 0 }  ][fill={rgb, 255:red, 0; green, 0; blue, 0 }  ][line width=1.5]      (0, 0) circle [x radius= 4.36, y radius= 4.36]   ;
		\draw [shift={(267.89,213.06)}, rotate = 0.19] [color={rgb, 255:red, 0; green, 0; blue, 0 }  ][fill={rgb, 255:red, 0; green, 0; blue, 0 }  ][line width=1.5]      (0, 0) circle [x radius= 4.36, y radius= 4.36]   ;
		\draw [color={rgb, 255:red, 65; green, 117; blue, 5 }  ,draw opacity=1 ]   (159.89,161.06) -- (320,213.5) ;
		\draw [shift={(320,213.5)}, rotate = 18.14] [color={rgb, 255:red, 65; green, 117; blue, 5 }  ,draw opacity=1 ][fill={rgb, 255:red, 65; green, 117; blue, 5 }  ,fill opacity=1 ][line width=0.75]      (0, 0) circle [x radius= 3.35, y radius= 3.35]   ;
		\draw [shift={(244.7,188.84)}, rotate = 198.14] [fill={rgb, 255:red, 65; green, 117; blue, 5 }  ,fill opacity=1 ][line width=0.08]  [draw opacity=0] (10.72,-5.15) -- (0,0) -- (10.72,5.15) -- (7.12,0) -- cycle    ;
		\draw [color={rgb, 255:red, 65; green, 117; blue, 5 }  ,draw opacity=1 ]   (320,213.5) -- (473.89,270.39) ;
		\draw [shift={(401.64,243.68)}, rotate = 200.29] [fill={rgb, 255:red, 65; green, 117; blue, 5 }  ,fill opacity=1 ][line width=0.08]  [draw opacity=0] (10.72,-5.15) -- (0,0) -- (10.72,5.15) -- (7.12,0) -- cycle    ;
		\draw [color={rgb, 255:red, 65; green, 117; blue, 5 }  ,draw opacity=1 ]   (320,213.5) -- (479.89,164.39) ;
		\draw [shift={(404.73,187.48)}, rotate = 162.93] [fill={rgb, 255:red, 65; green, 117; blue, 5 }  ,fill opacity=1 ][line width=0.08]  [draw opacity=0] (10.72,-5.15) -- (0,0) -- (10.72,5.15) -- (7.12,0) -- cycle    ;
		\draw [color={rgb, 255:red, 65; green, 117; blue, 5 }  ,draw opacity=1 ]   (152.56,263.06) -- (320,213.5) ;
		\draw [shift={(241.07,236.86)}, rotate = 163.51] [fill={rgb, 255:red, 65; green, 117; blue, 5 }  ,fill opacity=1 ][line width=0.08]  [draw opacity=0] (10.72,-5.15) -- (0,0) -- (10.72,5.15) -- (7.12,0) -- cycle    ;
		
		\draw (49.83,223.17) node [anchor=north west][inner sep=0.75pt]  [font=\footnotesize,xscale=0.8,yscale=0.8] [align=left] {$\displaystyle E_{1}$};
		\draw (153.67,221.17) node [anchor=north west][inner sep=0.75pt]  [font=\footnotesize,xscale=0.8,yscale=0.8] [align=left] {$\displaystyle \hat{E}_{1}$};
		\draw (242.67,215.87) node [anchor=north west][inner sep=0.75pt]  [font=\footnotesize,xscale=0.8,yscale=0.8] [align=left] {$\displaystyle E_{2}$};
		\draw (463.73,225.06) node [anchor=north west][inner sep=0.75pt]  [font=\footnotesize,xscale=0.8,yscale=0.8] [align=left] {$\displaystyle \hat{E}_{2}$};
		\draw (301,179.9) node [anchor=north west][inner sep=0.75pt]  [xscale=0.8,yscale=0.8]  {$z_{0}$};
		\draw (226,162.57) node [anchor=north west][inner sep=0.75pt]  [font=\small,xscale=0.8,yscale=0.8]  {$\Sigma _{2}$};
		\draw (234.33,243.4) node [anchor=north west][inner sep=0.75pt]  [font=\small,xscale=0.8,yscale=0.8]  {$\Sigma _{2}^{*}$};
		\draw (354.5,172.4) node [anchor=north west][inner sep=0.75pt]  [font=\small,xscale=0.8,yscale=0.8]  {$\Sigma _{1}$};
		\draw (381.83,247.57) node [anchor=north west][inner sep=0.75pt]  [font=\small,xscale=0.8,yscale=0.8]  {$\Sigma _{1}^{*}$};
		\draw (408.83,190.73) node [anchor=north west][inner sep=0.75pt]  [font=\small,xscale=0.8,yscale=0.8]  {$\Omega _{1}$};
		\draw (407.33,220.57) node [anchor=north west][inner sep=0.75pt]  [font=\small,xscale=0.8,yscale=0.8]  {$\Omega _{1}^{*}$};
		\draw (187.5,186.07) node [anchor=north west][inner sep=0.75pt]  [font=\small,xscale=0.8,yscale=0.8]  {$\Omega _{2}$};
		\draw (187.33,222.57) node [anchor=north west][inner sep=0.75pt]  [font=\small,xscale=0.8,yscale=0.8]  {$\Omega _{2}^{*}$};	
\end{tikzpicture}}
	\caption{\footnotesize The contours $\Sigma_i$ and the regions $\Omega_i$, $i=1,2$, in Zakharov-Manakov region if $\xi\in(\xi_2,\hat{\xi}_1)$ (Figure (a)) and in fast decaying region if $\xi\in(\hat{\xi}_2,\xi_2)$ (Figure (b)).}
	\label{fig III 1}
\end{figure}

The first transformation is the same as that given in \eqref{trans 1 I}
with different definitions of $\delta$, $\delta_{\infty}$ and $G$. A direct calculation shows that $M^{(1)}(z)$ is analytic in $\mathbb{C}\setminus\Sigma^{(1)}$, where  \begin{equation}\label{def:Sigma1caseIII}
		    \Sigma^{(1)}:=\Sigma^{(1)}(\xi)= \cup_{j=0}^n[E_j,\hat{E}_j]\cup\Sigma_1\cup\Sigma_1^*\cup\Sigma_2\cup\Sigma_2^*,
		\end{equation}
and satisfies the jump condition $M^{(1)}_{+}(z)=M^{(1)}_{-}(z)J^{(1)}(z)$ for $z\in \Sigma^{(1)}$, where
\begin{align}\label{def:J1caseIII}
    J^{(1)}(z)=\left\{ \begin{array}{ll}
				e^{i(f_0x+g_0t-(B_j^fx+B_j^gt+\phi_j-\delta_j)/2)\hat{\sigma}_3} \begin{pmatrix} 
				0 & -i\\
				-i& 0
			\end{pmatrix} , & z\in (E_j,\hat{E}_j),\ j=0,\ldots,n,\\
					 \begin{pmatrix} 
					1 & 0\\
					-r_2(z)e^{-2it\theta(z)-2\delta(z)} & 1	
				\end{pmatrix} ,&  z\in\Sigma_1,\\[12pt]
			 \begin{pmatrix} 
					1 & r_1(z)e^{2it\theta(z)+2\delta(z)} \\
					0 & 1
				\end{pmatrix} , &z\in\Sigma_1^*,\\[12pt]
				 \begin{pmatrix} 
			1 &	\frac{ r_1(z)e^{2it\theta(z)+2\delta(z)}}{1-r_1(z)r_2(z)}\\
			0 & 1
			\end{pmatrix} ,&z\in\Sigma_2,\\[12pt]
			 \begin{pmatrix} 
				1 &	0\\
			\frac{-r_2(z)e^{-2it\theta(z)-2\delta(z)}}{1-r_1(z)r_2(z)} & 1
			\end{pmatrix} ,&z\in\Sigma_2^*.
				\end{array}\right.
\end{align}

\vspace{2mm}
\paragraph{\textbf{Global and local parametrices }}

\begin{align}\label{def:EcaseIII}
	M^{(1)}(z)=\left\{ \begin{array}{ll}
		E(z)M^{(glo)}(z), & z\in \mathbb{C}\setminus U ,\\
		E(z)M^{(glo)}(z)M^{(loc)}(z),&  z\in U,
	\end{array}\right.
\end{align}
where 
\begin{align}
	U:=\left\{ \begin{array}{ll}
		\{z:\ |z-z_0|\leq c_0\},& \textrm{$\xi$ belongs to Zakharov-Manakov region}, 
        \\
	\emptyset, & \textrm{$\xi$ belongs to fast decaying region},
	\end{array}\right.\label{def Uxi III}
\end{align}
with 
\begin{align}\label{def c0 III}
	c_0:=\min_{j=0,\ldots,n}\left\{\frac{|\hat{E}_{j_0}-z_0|}{2},\ \frac{|E_j-z_0|}{2}\right\},
\end{align}

By \eqref{eq:thetainequality} and \eqref{def:J1caseIII}, we have that $M^{(glo)}$ also satisfies RH problem \ref{RHPM mod}, whose  solution is explicitly given through \eqref{def Mmod}--\eqref{eq:Nglob22}. 


For the local parametrix, we only need to consider the case 
when $\xi$ belongs to Zakharov-Manakov region, which reads as follows.
\begin{RHP}
	\hfil
	\begin{itemize}
		\item [$\mathrm{(a)}$] $M^{(loc)}(z)$ is analytic in $\mathbb{C}\setminus(\Sigma^{(1)}\cap U)$, where $\Sigma^{(1)}$ and $U$ are defined in  \eqref{def:Sigma1caseIII} and \eqref{def Uxi III}, respectively. 
		\item [$\mathrm{(b)}$]  For $z\in\Sigma^{(1)}\cap U$, $M^{(loc)}(z)$  satisfies the jump condition $$M^{(loc)}_{+}(z)=M^{(loc)}_{-}(z)J^{(1)}(z),$$
        where $J^{(1)}$ is given in \eqref{def:J1caseIII}.
		\item [$\mathrm{(c)}$] As $z\to\infty$, we have $M^{(loc)}(z)=I+\mathcal{O}(z^{-1})$.
	\end{itemize}
\end{RHP}

To solve the above RH problem, we note that
\begin{align}\label{eq:thetaexpIII}
	\theta(z)=\theta(z_0)-	\theta^{(z_0,2)}(\xi)(z-z_0)^2+\mathcal{O}\big((z-z_0)^3\big),\quad z\to z_0,
\end{align}
where
\begin{align}
	\theta^{(z_0,2)}(\xi):=\frac{\partial_z F(z_0;\xi)}{w(z_0)}>0 \label{def theta z0}
\end{align}
with $F$ and $w$ defined in \eqref{def F} and \eqref{def:w}, respectively.

Let
\begin{align}\label{def r0}
	r_0:&=-\left( 2\sqrt{\theta^{(z_0,2)}(\xi)}\right) ^{\frac{i}{2\pi}\log (1-r_1(z_0)r_2(z_0))}r_2(z_0)\nonumber
    \\
	&\qquad \times e^{ \frac{w(z_0)}{2\pi i}\left[\sum_{j=0}^n\delta_j\int_{E_j}^{\hat{E}_j}\frac{i\dif s}{w_+(s)(s-z_0)} \right]-2it\theta(z_0)+\frac{\log (1-r_1(z_0)r_2(z_0))\log(c_0)}{2\pi i} }
    \nonumber 
    \\
	&\qquad \times e^ {-\frac{w(z_0)}{2\pi i}\int_{(-\infty,z_0)\setminus(\cup_{j=0}^{n}(E_j,\hat{E}_j))}\left[\frac{\log (1-r_1(s)r_2(s))}{w(s)(s-z_0)} -\frac{\chi_{(z_0-c_0,z_0)}\log (1-r_1(z_0)r_2(z_0))}{w(z_0)(s-z_0)}\right]\dif s  } ,
\end{align}
where $c_0$ is given in \eqref{def c0 III} and $\chi_I$ stands for the characteristic function of the interval $I$.  Note that one can check $|r_0|=|r_2(z_0)| < 1$. From \eqref{def:J1caseIII}, \eqref{eq:thetaexpIII} and \eqref{def r0}, it follows that the jump of $M^{(loc)}$ is approximated by
\begin{align*}
	J^{(loc)}(z)=\left\{ \begin{array}{ll}
		 \begin{pmatrix} 
			1 & 0\\
			r_0\zeta^{-\log(1-r_1(z_0)r_2(z_0))/\pi i}e^{\frac{i\zeta^2}{2}} & 1	
		\end{pmatrix} ,&  z\in\Sigma_1\cap U,\\[12pt]
		 \begin{pmatrix} 
			1 & -\bar{r}_0\zeta^{\log(1-r_1(z_0)r_2(z_0))/\pi i}e^{-\frac{i\zeta^2}{2}} \\
			0 & 1
		\end{pmatrix}, &z\in\Sigma_1^*\cap U,\\[12pt]
		 \begin{pmatrix} 
			1 &	\frac{- \bar{r}_0\zeta^{\log(1-r_1(z_0)r_2(z_0))/\pi i}e^{-\frac{i\zeta^2}{2}}}{1-|r_0|^2}\\
			0 & 1
		\end{pmatrix} ,&z\in\Sigma_2\cap U,\\[12pt]
		 \begin{pmatrix} 
			1 &	0\\
			\frac{r_0\zeta^{\log(1-r_1(z_0)r_2(z_0))/\pi i}e^{-\frac{i\zeta^2}{2}}}{1-|r_0|^2} & 1
		\end{pmatrix} ,&z\in\Sigma_2^*\cap U,
	\end{array}\right.
\end{align*}
for large $t$, where $\zeta=(z-z_0)2\sqrt{\theta^{(z_0,2)}(\xi)}t^{1/2}$. Similar to \cite{sdmRHp,PX2,PX21,PX3,DM,MK}, this leads to an RH problem solved explicitly by using the parabolic cylinder parametrix given in \cite{its1}. At the end, it is accomplished that
\begin{align}
	M^{(loc)}(z)=I+\frac{\begin{pmatrix}
		0&\beta_{21}\\
	\beta_{12}&0\\
	\end{pmatrix}t^{-1/2}}{2(z-z_0)\sqrt{\theta^{(z_0,2)}(\xi)}}+\mathcal{O}(t^{-1}),\qquad t\to \infty,
\end{align}
 where
\begin{align}\label{def beta12}
	\beta_{12}=\frac{\sqrt{2\pi}e^{\frac{1}{4}(\pi i-\log(1-|r_0|^2))}}{r_0\Gamma(i\log(1-|r_0|^2)/2\pi)},\quad 
    \beta_{21}=\frac{\log(1-|r_0|^2)}{2\pi \beta_{12}}.
\end{align}

\vspace{2mm}
\paragraph{\textbf{The small norm RH problem}}
By using discussions similar to those presented in Section \ref{subsec E}, we have that the error function $E$ in \eqref{def:EcaseIII} exists uniquely for large $t$. When $\xi$ belongs to fast decaying region,  there exists a nonzero constant $C$ such that
\begin{align*}
    \theta(z)=C(z-z_0)+\mathcal{O}\big((z-z_0)^2\big), \qquad z\to z_0.
\end{align*}
This, together with the fact $J^{(1)}(z) \to I$ exponentially fast for $z\in \Sigma^{(1)}\setminus \mathbb{R}$ and bounded away from $z_0$, implies that 
$$E(z)=I+\mathcal{O}(t^{-1}).$$
When $\xi$ belongs to Zakharov-Manakov region, we have
\begin{align}
	E(z)=I+\frac{E_1}{z}+\mathcal{O}(z^{-2}), \qquad z\to \infty,
\end{align}
where	
\begin{align}\label{eq:asyE1 III}
	E_1&=\frac{t^{-1/2}}{2(z-z_0)\sqrt{\theta^{(z_0,2)}(\xi)}}M^{(glo)}(z_0)\begin{pmatrix}
		0&\beta_{21}\\
	\beta_{12}&0\\
	\end{pmatrix}M^{(glo)}(z_0)^{-1}+\mathcal{O}(t^{-1}),\qquad t\to +\infty. 
\end{align}

\section{Proof of Theorem \ref{mainthm}}\label{sec:proof}
In view of the reconstruction formulas established in Proposition \ref{prop:reconstuction}, the idea of the proof  
is to invert a series of transformations performed in previous  sections. 

\vspace{2mm}
\paragraph{\textbf{Proofs of \eqref{eq:asyq I} and \eqref{eq:asyq II}}}
For the asymptotics of $q$ in transition region I, we see from \eqref{trans 1 I} and \eqref{construct M1} that, as $t\to +\infty$,  
\begin{align*}
	M(z)=e^{-\delta(\infty)\sigma_3}E(z)M^{(glo)}(z)e^{\delta(z)\sigma_3}G(z)^{-1}, \quad z\in \mathbb{C}\setminus U,
\end{align*}
where $\delta(\infty)$, $\delta$, $E$, $M^{(glo)}$, $G$ and $U$ are defined in \eqref{def delta inf}, \eqref{def delta 1}, \eqref{expression:E(z)}, \eqref{def Mmod}, \eqref{def:G} and \eqref{def Uxi I}, respectively. This, together with \eqref{eq:Eneari}, \eqref{asy Mmod} and \eqref{eq:deltaexp}, implies that 
\begin{align}
	M(z)&=e^{-\delta(\infty)\sigma_3}\left(I+\frac{E_1}{z}\right)\left(I+\frac{M^{(glo)}_1}{z}\right)e^{\delta(\infty)\sigma_3}\left(I+\frac{\delta^{(1)}}{z}\sigma_3\right)+\mathcal{O}(z^{-2})
    \nonumber \\
	&=I+\frac{1}{z}\left( e^{-\delta(\infty)\hat{\sigma}_3}(E_1+M^{(glo)}_1)+\delta^{(1)}\sigma_3\right) +\mathcal{O}(z^{-2}), \qquad z\to \infty, \label{eq:M asy}
\end{align}
where $\delta^{(1)}$ is given in \eqref{def delta1}. 
From the reconstruction formula \eqref{res q},  we obtain from \eqref{eq:asyE1}, \eqref{asy Mmod12} and the above formula that, as $t \to +\infty$, 
\begin{align*}
q(x,t)&=2i\lim_{z\to\infty}zM_{12}=2ie^{-2\delta(\infty)}\left((E_{1})_{12}+\left(M^{(glo)}_{1}\right)_{12}\right)
    \\
	&= e^{-2\delta(\infty)}q^{(AG)}(x,t;\E, \hat{\E}, \boldsymbol{\phi-\delta})
	 +t^{-1/3}\frac{2e^{-2\delta(\infty)}}{|(\theta^{(3,\hat{j_0})}(\hat{\xi}_{j_0}))^{2/3}|}a(s)(H_{1})_{11}(\hat{E}_{j_0})^2+\mathcal{O}(t^{-\epsilon}).
\end{align*}
On account of the definition of $H_1$ in \eqref{def H1 I}, the asymptotic formula \eqref{eq:asyq I} follows directly by setting
\begin{align}\label{Hhj0}
	H_{\hat{E}_{j_0}}&= \frac{2e^{-2\delta(\infty)}}{|(\theta^{(3,\hat{j_0})}(\hat{\xi}_{j_0}))^{2/3}|}(H_{1})_{11}(\hat{E}_{j_0})^2
    \nonumber 
    \\
    &=\frac{e^{i(t\theta(\hat{E}_{j_0};\xi)-\phi_{j_0}+\delta_{j_0})-2\delta(\infty)}}{|\theta^{(3,\hat{j_0})}(\hat{\xi}_{j_0})|^{1/3}}\left( \frac{\prod_{j=0}^n(\hat{E}_{j_0}-E_j)}{\prod_{j=0,\cdots,n,j\neq j_0}(\hat{E}_{j_0}-\hat{E}_{j})}\right) ^{1/2}\nonumber\\
	&\quad \times \left( \dfrac{\Theta(\mathcal{A}(\hat{E}_{j_0})+\boldsymbol{C}(x,t;\boldsymbol{\phi-\delta})+\mathcal{A}(\mathcal{D})+\boldsymbol{K})\Theta(\mathcal{A}(\infty)+\mathcal{A}(\mathcal{D})+\boldsymbol{K})}{\Theta(\mathcal{A}(\hat{E}_{j_0})+\mathcal{A}(\mathcal{D})+\boldsymbol{K})\Theta(\mathcal{A}(\infty)+\boldsymbol{C}(x,t;\boldsymbol{\phi-\delta})+\mathcal{A}(\mathcal{D})+\boldsymbol{K})}\right)^2.
\end{align}

The asymptotics of $q$ in transition region II given in \eqref{eq:asyq II} can be proved in a similar manner by noting
	\begin{align}\label{HEj0}
			\tilde{H}_{E_{j_0}}&=\frac{e^{\pi i/2+i(t\theta(E_{j_0})-\phi_{j_0}+\delta_{j_0})-2\delta(\infty)}}{|\theta^{(3,j_0)}(\xi_{j_0})|^{1/3}}\left( \frac{\prod_{j=0}^n(E_{j_0}-\hat{E}_j)}{\prod_{j=0,\cdots,n,j\neq j_0}(E_{j_0}-E_{j})}\right) ^{1/2}\nonumber\\
			&\quad \times  \left(\dfrac{\Theta(\mathcal{A}(E_{j_0})+\boldsymbol{C}(x,t;\boldsymbol{\phi-\delta})+\mathcal{A}(\mathcal{D})+\boldsymbol{K})\Theta(\mathcal{A}(\infty)+\mathcal{A}(\mathcal{D})+\boldsymbol{K})}{\Theta(\mathcal{A}(E_{j_0})+\mathcal{A}(\mathcal{D})+\boldsymbol{K})\Theta(\mathcal{A}(\infty)+\boldsymbol{C}(x,t;\boldsymbol{\phi-\delta})+\mathcal{A}(\mathcal{D})+\boldsymbol{K})}\right)^2.
		\end{align}

\vspace{2mm}
\paragraph{\textbf{Proofs of \eqref{eq:asyIII} and \eqref{eq:asyq IV}}}
We only give the proof of \eqref{eq:asyIII}, which corresponds to the case of Zakharov-Manakov region. In this case, we still have \eqref{eq:M asy} but with the functions therein understood in the context of Section \ref{Sec 5}. 
Thus, a combination of \eqref{res q}, \eqref{asy Mmod12} and \eqref{eq:asyE1 III}
gives us that, as $t \to +\infty$, 
\begin{align*}
q(x,t)&=2i\lim_{z\to\infty}zM_{12}=2ie^{-2\delta(\infty)}\left((E_{1})_{12}+\left(M^{(glo)}_{1}\right)_{12}\right) \\
          & = e^{-2\delta(\infty)}q^{(AG)}(x,t;\E, \hat{\E}, \boldsymbol{\phi-\delta})\\
             &~~~ +\frac{2ie^{-2\delta(\infty)}}{2\sqrt{\theta^{(z_0,2)}(\xi)}}\left(\beta_{21}M^{(glo)}_{11}(z_0)^2-\beta_{12}M^{(glo)}_{12}(z_0)^2\right)t^{-1/2}+\mathcal{O}(t^{-1}),
\end{align*}
which is \eqref{eq:asyIII}. In particular, we see from the definition of $M^{(glo)}$ in \eqref{def Mmod} that
\begin{align}
			M^{(glo)}_{11}(z_0;\E, \hat{\E}, \boldsymbol{\phi-\delta})&=\frac{1}{2}\left(\left(\prod_{j=0}^n\frac{z_0-E_j}{z_0-\hat{E}_j} \right)^{1/4}+\left(\prod_{j=0}^n\frac{z_0-E_j}{z_0-\hat{E}_j} \right)^{-1/4}\right)\nonumber\\
			&\quad \times \dfrac{\Theta(\mathcal{A}(z_0)+\boldsymbol{C}(x,t;\boldsymbol{\phi-\delta})+\mathcal{A}(\mathcal{D})+\boldsymbol{K})\Theta(\mathcal{A}(\infty)+\mathcal{A}(\mathcal{D})+\boldsymbol{K})}{\Theta(\mathcal{A}(z_0)+\mathcal{A}(\mathcal{D})+\boldsymbol{K})\Theta(\mathcal{A}(\infty)+\boldsymbol{C}(x,t;\boldsymbol{\phi-\delta})+\mathcal{A}(\mathcal{D})+\boldsymbol{K})},\label{Mglo11}\\
			M^{(glo)}_{12}(z_0;\E, \hat{\E}, \boldsymbol{\phi-\delta})&=\frac{e^{-2i(f_0x+g_0t)}}{2}\left(\left(\prod_{j=0}^n\frac{z_0-E_j}{z_0-\hat{E}_j} \right)^{1/4}-\left(\prod_{j=0}^n\frac{z_0-E_j}{z_0-\hat{E}_j} \right)^{-1/4}\right)\nonumber\\
			&\quad \times 
            \dfrac{\Theta(-\mathcal{A}(z_0)+\boldsymbol{C}(x,t;\boldsymbol{\phi-\delta})+\mathcal{A}(\mathcal{D})+\boldsymbol{K})\Theta(-\mathcal{A}(\infty)+\mathcal{A}(\mathcal{D})+\boldsymbol{K})}{\Theta(-\mathcal{A}(z_0)+\mathcal{A}(\mathcal{D})+\boldsymbol{K})\Theta(-\mathcal{A}(\infty)+\boldsymbol{C}(x,t;\boldsymbol{\phi-\delta})+\mathcal{A}(\mathcal{D})+\boldsymbol{K})},\label{Mglo12}
		\end{align}
		with the two constants $f_0$ and $g_0$ given in \eqref{def f0g0}.

This completes the proof of Theorem \ref{mainthm}. \qed

\vspace{2mm}

%
%
%
%

\appendix

\section{The Painlev\'{e}  XXXIV ($P_{34}$)  parametrix}
\label{APP P34}

The $P_{34}$ parametrix $M^{(P_{34})}(\zeta)=M^{(P_{34})}(\zeta;s,\alpha,\omega)$ is a $2\times 2$  matrix-valued function
depending on the parameters $s,\alpha$ and $\omega$. It satisfies the following RH problem.


\begin{RHP}\label{RH P34}
	\hfill
	\begin{enumerate}
		\item [$\mathrm{(a)}$]
		$M^{(P_{34})}(\zeta)$ is analytic for $\zeta$ in
		$\mathbb{C} \setminus \{\cup_{j=1}^4\Sigma_j\cup\{0\}\}$, where
		\begin{align}
		\Sigma_1=\mathbb{R}^+, ~~ \Sigma_2=e^{\frac{2 \pi i}{3}}\mathbb{R}^+, ~~\Sigma_3=e^{   \pi i  }\mathbb{R}^+,~~\Sigma_4=e^{-\frac{2 \pi i}{3}}\mathbb{R}^+,\label{def jumpcureve 34}
		\end{align}
		with the orientations as shown in Figure \ref{fig:jumpsPsi}.
		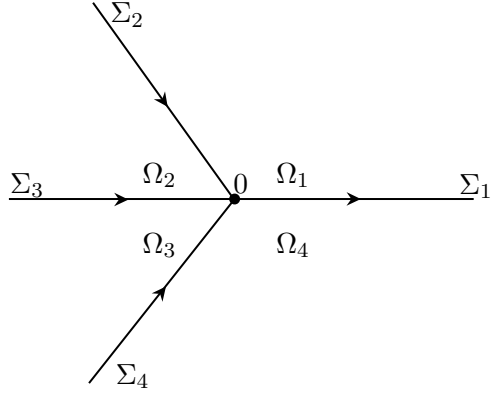
\begin{figure}[t]
			\begin{center}
				
			\tikzset{every picture/.style={line width=0.75pt}} 
			
			\begin{tikzpicture}[x=1pt,y=1pt,yscale=-0.5,xscale=0.5]
				
				\draw    (155.19,159) -- (324.19,159) ;
				\draw [shift={(324.19,159)}, rotate = 0.34] [color={rgb, 255:red, 0; green, 0; blue, 0 }  ][fill={rgb, 255:red, 0; green, 0; blue, 0 }  ][line width=0.75]      (0, 0) circle [x radius= 3.35, y radius= 3.35]   ;
				\draw [shift={(244.69,159.43)}, rotate = 180.34] [fill={rgb, 255:red, 0; green, 0; blue, 0 }  ][line width=0.08]  [draw opacity=0] (10.72,-5.15) -- (0,0) -- (10.72,5.15) -- (7.12,0) -- cycle    ;
				\draw    (324.19,159) -- (503.19,159) ;
				\draw [shift={(418.69,159)}, rotate = 180.32] [fill={rgb, 255:red, 0; green, 0; blue, 0 }  ][line width=0.08]  [draw opacity=0] (10.72,-5.15) -- (0,0) -- (10.72,5.15) -- (7.12,0) -- cycle    ;
				\draw    (218.19,10.9) -- (324.19,159.9) ;
				\draw [shift={(274.09,89.47)}, rotate = 234.57] [fill={rgb, 255:red, 0; green, 0; blue, 0 }  ][line width=0.08]  [draw opacity=0] (10.72,-5.15) -- (0,0) -- (10.72,5.15) -- (7.12,0) -- cycle    ;
				\draw    (215.19,297.9) -- (324.19,159) ;
				\draw [shift={(272.79,224.97)}, rotate = 128.3] [fill={rgb, 255:red, 0; green, 0; blue, 0 }  ][line width=0.08]  [draw opacity=0] (10.72,-5.15) -- (0,0) -- (10.72,5.15) -- (7.12,0) -- cycle    ;
				\draw (321,138) node [anchor=north west][inner sep=0.75pt]   [align=left] { $\displaystyle 0$};
				\draw (154,137) node [anchor=north west][inner sep=0.75pt]   [align=left]  {$\displaystyle \Sigma_{3}$};
				\draw (490,138) node [anchor=north west][inner sep=0.75pt]   [align=left]  {$\displaystyle \Sigma_{1}$};
				\draw (229,8.4) node [anchor=north west][inner sep=0.75pt]  [align=left]  {$\displaystyle \Sigma_{2}$};
				\draw (233,281.4) node [anchor=north west][inner sep=0.75pt]   [align=left]  {$\displaystyle \Sigma_{4}$};
				\draw (353,129.4) node [anchor=north west][inner sep=0.75pt] [align=left]   {$\displaystyle \Omega_{1}$};
				\draw (253,129.4) node [anchor=north west][inner sep=0.75pt]   [align=left]  {$\displaystyle \Omega_{2}$};
				\draw (353,182.4) node [anchor=north west][inner sep=0.75pt]   [align=left]  {$\displaystyle \Omega_{4}$};
				\draw (253,182.4) node [anchor=north west][inner sep=0.75pt]   [align=left]  {$\displaystyle \Omega_{3}$};		
			\end{tikzpicture}
			\caption{\footnotesize  The jump contours $\Sigma_j$ and the regions $\Omega_j$, $j=1,2,3,4$, of the RH problem for $M^{(P_{34})}$.}
				\label{fig:jumpsPsi}
			\end{center}
		\end{figure}
		
		\item[\rm(b)]  $M^{(P_{34})}$  satisfies the jump condition $	M^{(P_{34})}_+ (\zeta)=M^{(P_{34})}_- (\zeta)J^{(P_{34})}(\zeta),$ with 
		\begin{equation}\label{eq:Psi-jump}
		J^{(P_{34})}(\zeta)=
			\left\{ \begin{array}{ll}
				\begin{pmatrix}
					1 & \omega 
                    \\
					0 & 1
				\end{pmatrix}, &\quad \zeta \in {\Sigma}_1,
				\\[9pt]
				\begin{pmatrix}
					1 & 0 \\
					e^{2\alpha\pi i} & 1
				\end{pmatrix}, &\quad \zeta \in {\Sigma}_2,
				\\[9pt]
				\begin{pmatrix}
					0 & 1 \\
					-1 & 0
				\end{pmatrix},& \quad
				\zeta \in {\Sigma}_3,
				\\[9pt]
				\begin{pmatrix}
					1 & 0 \\
					e^{-2\alpha\pi i} & 1
				\end{pmatrix},  & \quad \zeta \in \Sigma_4.
			\end{array}
			\right .
		\end{equation}
		\item [\rm(c)]  As $\zeta \to \infty$, there exists a function $a(s)=a(s;\alpha,\omega)$ such that
		\begin{multline} \label{eq:Psi-infinity}
			M^{(P_{34})}(\zeta) = \begin{pmatrix}
				1 & 0\\
				-ia(s) & 1
			\end{pmatrix}
			\left(I+\frac {M^{(P_{34})}_1(s)}{\zeta}
			+\mathcal{O} \left( \zeta^{-2} \right) \right)
            \\ \times
			\frac{\zeta^{-\frac{1}{4}\sigma_3}}{\sqrt{2}}
			\begin{pmatrix}
				1 & i
				\\
				i & 1
			\end{pmatrix} e^{-(\frac{2}{3}\zeta^{3/2}+s\zeta^{1/2}) \sigma_3},
		\end{multline}
          where we take principle branch for the fractions and
		\begin{equation}\label{eq:12entry}
        \left(M^{(P_{34})}_1\right)_{12}(s)=ia(s).
        \end{equation}
		\item[\rm(d)] As $\zeta \to 0$, we have, if  $-1/2 < \alpha < 0$
        \begin{equation}\label{eq:resP34}
            M^{(P_{34})}(\zeta)=\mathcal{O}(\zeta^{\alpha}), 
        \end{equation}
        and if $\alpha \geq 0$,
         \begin{equation*}
            M^{(P_{34})}(\zeta)=\left\{ \begin{array}{ll}
				\begin{pmatrix}
					\mathcal{O}(\zeta^{\alpha}) & \mathcal{O}(\zeta^{-\alpha}) 
                    \\
					\mathcal{O}(\zeta^{\alpha}) & \mathcal{O}(\zeta^{-\alpha})
				\end{pmatrix}, &\quad \zeta \in \Omega_1\cup\Omega_4,
				\\
				\mathcal{O}(\zeta^{-\alpha}), &\quad \zeta \in \Omega_2\cup\Omega_3,
			\end{array}
			\right .
        \end{equation*}

    where  the regions $\Omega_j$, $j=1,2,3,4$, are shown in Figure \ref{fig:jumpsPsi}.
		
%
\end{enumerate}
\end{RHP}

By \cite{ikj2008,ikj2009,XZ2011}, the above RH problem is uniquely solvable for $\alpha>-1/2$, $\omega \in \mathbb{C}\setminus (-\infty,0)$, and $s\in \mathbb{R}$. Moreover, with $a(s)$ given in
\eqref{eq:Psi-infinity}, the function
$$
u(s):=u(s;\alpha,\omega)=a'(s;\alpha,\omega)-\frac s2
$$
satisfies the  Painlev\'{e}  XXXIV equation \eqref{eq:P34}
and is pole-free on the real axis. Particularly, one has 
\begin{equation}\label{eq:uasygen}
    u(s;\alpha, 0)=\left\{ \begin{array}{ll}
				\alpha/\sqrt{s}+\mathcal{O}(s^{-2}), &\qquad s\to +\infty,
				\\
				-s/2+\mathcal{O}(s^{-2}), &\qquad s\to -\infty.
			\end{array}
			\right .
\end{equation}
This, together with the fact that $a(s;\alpha,0)\to 0$ as $s\to -\infty$ (cf. \cite[Equations (2.10) and (3.60)]{DXZ20}), implies that
\begin{equation}
    a(s;\alpha,0)=\int_{-\infty}^s \left( u(t;\alpha, 0) + \frac t2\right) \dif t.
\end{equation}

\section*{Acknowledgements}
The authors are grateful to the anonymous reviewers for their careful reading of the manuscript, as well as their insightful comments and suggestions. This work is supported by the National Science Foundation of China under grant numbers 12271104, 12271105 and 12501325. Gaozhan Li is partially supported by China Postdoctoral Science Foundation. Yiling Yang is partially supported by the Fundamental Research Funds for the Central Universities under grant number 2025CDJ-IAISYB-011 and by Natural Science Foundation of Chongqing under grant number CSTB2025NSCQ-GPX0727.

\vspace{2mm}
\section*{Data availability statement}
Data sharing is not applicable to this article as no datasets were generated or analysed
during the current study.
\section*{Conflict of interest}
On behalf of all authors, the corresponding author states that there is no conflict of
interest. All authors read and approved the final manuscript.

\section*{Ethical Statement}
This manuscript is original and has not been published and will not be submitted elsewhere for publication. 
The study is not split up into several parts to increase the quantity of submissions and submitted to various journals or to one journal over time.

\hspace*{\parindent}
\\


\begin{thebibliography}{10}

\bibitem{Abl74}
M. J. Ablowitz, D. J. Kaup, A. C. Newell and H.Segur, The inverse scattering transform-Fourier analysis for
nonlinear problems, Stud. Appl. Math., \textbf{53} (1974), 249-315.
	
\bibitem{PRL2006}
J. D. Ania-Casta\~n\'on, T. J. Ellingham, R. Ibbotson, X. Chen, L. Zhang and S. K. Turitsyn, Ultralong raman fiber lasers as virtually lossless optical media, Phys. Rev. Lett., \textbf{96} (2006), 023902.

\bibitem{BC} 	
	R. Beals and R. R. Coifman, Scattering and inverse scattering for first order systems, Commun. Pure
	Appl. Math., \textbf{37} (1984), 39-90.
	


\bibitem{AGbook} D. Belokolos, A. I. Bobenko, V. Z. Enolskii, A. R. Its and V. B. Matveev, Algebro-geometric
approach to nonlinear integrable equations, Springer-Verlag, New York, (1994).


\bibitem{Miller1} D. Bilman, L. M.  Ling and  P. D. Miller,  Extreme superposition: rogue waves of infinite order and the Painlev\'e-III hierarchy, Duke Math. J.,
\textbf{169} (2020), 671-760.



\bibitem{Monvel}
A. Boutet de Monvel, A. Its and  D. Shepelsky, Painlev\'{e}-type asymptotics for the Camassa-Holm  equation,
SIAM J.  Math.  Anal.,  \textbf{42}  (2010), 1854-1873.

  \bibitem{Monvel2}   A. Boutet de Monvel,   J. Lenells and D. Shepelsky,
The focusing NLS equation with step-like oscillating background: scenarios of long-time asymptotics, Commun. Math. Phys., \textbf{383} (2021), 893-952.

\bibitem{Monvel5}   A. Boutet de Monvel,   J. Lenells and D. Shepelsky,
The focusing NLS equation with step-like oscillating background: asymptotics in a transition zone, J. Differential Equations, \textbf{429} (2025), 747-801.




\bibitem{lenellsboss}
C. Charlier and J. Lenells, On Boussinesq's equation for water waves, arXiv:2204.02365.

\bibitem{C1964} R. Y. Chiao, E. Garmire, and C. H. Townes, Self-trapping of optical beams, Phys. Rev. Lett., \textbf{13} (1964), 479-482.

\bibitem{PRL2005}
C. Connaughton, C. Josserand, A. Picozzi, Y. Pomeau and S. Rica, Condensation of classical nonlinear waves, Phys. Rev. Lett., \textbf{95} (2005), 263901.

\bibitem{CJ}
S. Cuccagna and R. Jenkins,
On the asymptotic stability of $N$-soliton solutions of the defocusing nonlinear Schr\"odinger equation,
\newblock  Commun. Math. Phys., \textbf{343}  (2016), 921-969.

\bibitem{DXZ20}
D. Dai, S. X. Xu and L.  Zhang,
On integrals of the tronqu\'{e}e solutions and the associated Hamiltonians for the Painlev\'{e} II equation,
J. Differential Equations, \textbf{269} (2020), 2430-2476.

\bibitem{DXZ2019}
D. Dai, S. X. Xu and L.  Zhang,
Gaussian unitary ensembles with pole singularities near the soft edge and a system of coupled Painlev\'e XXXIV equations,
Ann. Henri Poincar\'e, \textbf{20} (2019), 3313-3364.

\bibitem{DT1979}
P. Deift and E. Trubowitz,
Inverse scattering on the line,
Commun. Pure Appl. Math., \textbf{32} (1979), 121-251.

\bibitem{sdmRHp}
P. Deift and X. Zhou, A steepest descent method for oscillatory Riemann-Hilbert problems,
\newblock {Ann. Math.}, \textbf{137} (1993), 295-368.


\bibitem{PX2}
P. Deift and X. Zhou,
\newblock Long-time behavior of the non-focusing nonlinear Schr\"odinger equation, a case study, Lectures in Mathematical Sciences,
 New Ser., Vol.5, Graduate School of Mathematical Sciences, University of Tokyo, (1994).

\bibitem{PX21}  
P. Deift and X. Zhou, Long-time asymptotics for integrable systems. Higher order theory, Commun. Math. Phys., \textbf{165} (1994), 175-191.



\bibitem{PX3}
P. Deift and  X. Zhou,
\newblock Long-time asymptotics for solutions of the NLS equation with initial data in a weighted Sobolev space,
\newblock Commun. Pure Appl. Math., \textbf{56} (2002), 1029-1077.

%
\bibitem{DM}
M. Dieng and K. T. R. McLaughlin,
Long-time asymptotics for the NLS equation via $\bar{\partial}$ methods, arXiv:0805.2807.


\bibitem{MK}
M. Dieng, K. T. R. McLaughlin and P. D. Miller,
\newblock{Dispersive asymptotics for linear and integrable equations by the $\bar{\partial}$ steepest descent method,
Nonlinear dispersive partial differential equations and inverse scattering, Fields Inst. Comm., 83},
Springer, New York, (2019), 253-291.

\bibitem{Farbook}
H. Farkas and I. Kra, Riemann Surfaces, 2nd edition, Springer, New York, (1992).




\bibitem{lenells}
S. Fromm, J.  Lenells and R. Quirchmayr,
\newblock   The defocusing nonlinear Schr\"odinger equation with step-like oscillatory initial data, Adv. Differential Equations {\bf 30} (2025), 455-525. 

\bibitem{AGbook2}
F. Gesztesy and H. Holden,
\newblock
Soliton Equations and Their Algebro-Geometric Solutions, Cambridge University Press, (2003).

\bibitem{G1955}
V. L. Ginzburg, On the theory of superconductivity, Il Nuovo Cimento, \textbf{2}
(1955), 1234-1250.

\bibitem{killpnls} B. Harrop-Griffith, R. Killip and M. Vi\c{s}an, \newblock  Sharp wellposednes for the cubic NLS and mKdV in $H^s(\mathbb{R})$,  Forum Math. Pi, \textbf{12} (2024), e6, 1-86.

\bibitem{HL20}
L. Huang and J. Lenells, Asymptotics for the Sasa-Satsuma equation in terms of a modified Painlev\'e II transcendent, 
J. Differential Equations, \textbf{268} (2020), 7480-7504.

\bibitem{Huanglin}
L. Huang and  L. Zhang, Higher order Airy and Painlev\'e asymptotics for the mKdV hierarchy,
SIAM J. Math. Anal., \textbf{54} (2022), 5291-5334.


\bibitem{its1}  A. R. Its, Asymptotics of solutions of the nonlinear Schr\"odinger equation and isompnpdromic
deformations of systems of linear equation, Sov. Math. Dokl., \textbf{24} (1981), 452-456.





\bibitem{ikj2008}
A. R. Its, A. B. J. Kuijlaars  and J. \"{O}stensson, Critical edge behavior in unitary
random matrix ensembles and the thirty fourth Painlev\'{e}
transcendent, { Int. Math. Res. Not.}, {\bf 2008} (2008), rnn017.


\bibitem{ikj2009}
A. R. Its, A. B. J. Kuijlaars  and J. \"{O}stensson,  Asymptotics for a special solution of the thirty fourth Painlev\'e equation, { Nonlinearity}, \textbf{22} (2009), 1523-1558.

\bibitem{Jenkins} R. Jenkins,
\newblock{Regularization of a sharp shock by the defocusing nonlinear Schr\"odinger equation,
	Nonlinearity}, \textbf{28} (2015), 2131-2180.

\bibitem{Duke}
H. Koch and D. Tataru, Conserved energies for the cubic nonlinear Schr\"odinger equation in one dimension,
Duke Math. J., \textbf{167} (2018), 3207-3313.

\bibitem{BA} V. Kotlyarov and D. Shepelsky, Planar unimodular Baker-Akhiezer function for the nonlinear
Schr\"odinger equation, Ann. Math. Sci. Appl., \textbf{2} (2017), 343-384.





\bibitem{Lake1977}  B. M. Lake, H. C. Yuen, H. Rungaldier and W. E. Ferguson,
Nonlinear deep-water waves: theory
and experiment. part 2. evolution of a continuous wave train, J.  Fluid Mech.,
\textbf{11} (1977), 849-874.


\bibitem{LenellsRobin}
J. M. Lee and J. Lenells, The nonlinear Schr\"odinger equation on the half-line with homogeneous Robin
boundary conditions, Proc. Lond. Math. Soc., \textbf{126} (2023), 334-389.



\bibitem{Alice}
A. Mikikits-Leitner and G. Teschl,
Long-time asymptotics of perturbed finite-gap Korteweg-de Vries solutions,
J. Anal. Math., \textbf{116} (2012), 163-218.



\bibitem{OO2001}  M. Onorato, A. R. Osborne, M. Serio and S. Bertone, Freak waves in random oceanic sea states, Phys. Rev. Lett.,
\textbf{86} (2001), 5831-5834.

\bibitem{AblwzP1}
H. Segur and M. J. Ablowitz, Asymptotic solutions of nonlinear evolution equations and a Painlev\'e
transcendent, Phys D., \textbf{3} (1981), 165-184.


\bibitem{PRL1970}
R. Taylor, D. Baker and H. Ikezi, Observation of colisionless electrostatic shocks, Phys. Rev. Lett., \textbf{24} (1970), 206-208.


\bibitem{VAH1}
A. H. Vartanian,
\newblock Large-time continuum asymptotics of dark solitons, Inverse Problems, \textbf{16} (2000), 39-46.

\bibitem{VAH2}
A. H. Vartanian,
\newblock Long-time asymptotics of solutions to the Cauchy problem for the defocusing nonlinear
Schr\"odinger equation with finite-density initial data. II. Dark solitons on continua,
\newblock Math. Phys. Anal. Geom., \textbf{5} (2002), 319-413.

\bibitem{VAH3}
A. H. Vartanian,
\newblock Exponentially small asymptotics of solutions to the defocusing nonlinear Schr\"odinger
equation,
\newblock Appl. Math. Lett., \textbf{16} (2003), 425-434.



\bibitem{WF}   Z. Y. Wang and  E. G. Fan,  The defocusing NLS equation with nonzero background: Large-time asymptotics in the solitonless region,   J. Differential Equations, \textbf{336} (2022), 334-373.

\bibitem{WZY}   Z. Y. Wang and  E. G. Fan,  The defocusing nonlinear Schr\"odinger equation with a nonzero background: Painlev\'e asymptotics in two transition regions, Commun. Math. Phys., \textbf{402} (2023),  2879-2930.

\bibitem{wang} D. S. Wang and P. Yan, Rigorous asymptotic analysis for the Riemann problem of the defocusing nonlinear Schr\"odinger hydrodynamics, arXiv:2305.12968.


\bibitem{XZ2011}
S. X. Xu and Y. Q. Zhao,
Painlev\'e XXXIV asymptotics of orthogonal polynomials for the Gaussian weight with a jump at the edge, Stud. Appl. Math., \textbf{127} (2011), 67-105.

\bibitem{XYZ}
T. Y. Xu, Y. L. Yang and L. Zhang,
Transient asymptotics of the modified Camassa-Holm equation,
J. Lond. Math. Soc., \textbf{110} (2024), e12967.

%






\bibitem{ZS3}
V. E. Zakharov and  S. V. Manakov,
\newblock Asymptotic behavior of non-linear wave systems integrated by the inverse scattering method,
\newblock  {Sov. Phys. JETP}, \textbf{44} (1976), 106-112.5.




 \bibitem{ZS1}
    V. E. Zahkarov and  A. B. Shabat,
    \newblock Exact theory of two-dimensional self-focusing and one-dimensional self-modulation
    of waves in nonlinear media,
    \newblock {Sov. Phys. JETP}, \textbf{34} (1972), 62-69.


 \bibitem{ZS2}
    V. E. Zahkarov and  A. B. Shabat,
    \newblock Interaction between solitons in a stable medium,
    \newblock {Sov. Phys. JETP}, \textbf{37} (1973), 823-828.










\end{thebibliography}
\end{document}